\documentclass[11pt]{article}

\usepackage[T1]{fontenc}
\usepackage{amsmath,amssymb,amsthm,mathtools}
\usepackage[a4paper,margin=1.1in]{geometry}
\usepackage{microtype}
\usepackage{needspace}
\usepackage{xcolor}
\usepackage{hyperref}
\usepackage[nameinlink,capitalize,noabbrev]{cleveref}

\hypersetup{
  colorlinks=true,
  linkcolor=blue!50!black,
  citecolor=blue!50!black,
  urlcolor=blue!50!black,
  pdftitle={Full-radius dimension-free maximal inequalities for discrete Euclidean balls},
  pdfauthor={Kaiwen Jin and Qingtang Su}
}

\newtheorem{theorem}{Theorem}[section]
\newtheorem{proposition}{Proposition}[section]
\newtheorem{lemma}{Lemma}[section]
\newtheorem{corollary}{Corollary}[section]
\theoremstyle{definition}
\newtheorem{definition}{Definition}[section]
\theoremstyle{remark}
\newtheorem{remark}{Remark}[section]
\crefname{theorem}{theorem}{theorems}
\Crefname{theorem}{Theorem}{Theorems}
\crefname{proposition}{proposition}{propositions}
\Crefname{proposition}{Proposition}{Propositions}
\crefname{lemma}{lemma}{lemmas}
\Crefname{lemma}{Lemma}{Lemmas}
\crefname{corollary}{corollary}{corollaries}
\Crefname{corollary}{Corollary}{Corollaries}
\crefname{definition}{definition}{definitions}
\Crefname{definition}{Definition}{Definitions}
\crefname{remark}{remark}{remarks}
\Crefname{remark}{Remark}{Remarks}

\newcommand{\Z}{\mathbb Z}
\newcommand{\R}{\mathbb R}
\newcommand{\T}{\mathbb T}
\newcommand{\N}{\mathbb N}

\newcommand{\abs}[1]{\left\lvert #1\right\rvert}
\newcommand{\Ball}{\mathcal B}
\newcommand{\Avg}{\mathcal M}

\title{Full-radius dimension-free maximal inequalities for discrete Euclidean balls}
\author{Kaiwen Jin \and Qingtang Su}
\date{}

\begin{document}
\maketitle

\begin{abstract}
For every $1<p\le\infty$, we prove dimension-free maximal inequalities
over all radii for normalized averages over Euclidean balls in $\mathbb Z^d$.
In particular, this settles the $\ell^2$ question attributed to Stein.  The
proof uses a two-saddle expansion at integer squared radii to compare ball
multipliers with normalized discrete Gaussians at the zero and parity
frequencies.  First-order estimates give the full-radius $\ell^2$ bound.  For
$1<p<2$, we combine higher-order residual estimates with dimension-uniform
$\ell^1$ bounds for the residuals and levelwise interpolation to obtain the
full range.
\end{abstract}

\section{Introduction}
\label{sec:introduction}

Let
\[
 B_2^d=\{x\in\mathbb R^d:|x|\le1\}
\]
be the Euclidean unit ball. For $t>0$ write
\[
 \mathcal A_t^{\mathrm{cont}}f(x)
 =\frac1{|tB_2^d|}\int_{tB_2^d}f(x-y)\,dy.
\]
In the early 1980s, Stein established that, for every
$1<p\le\infty$, the norm of
$\sup_{t>0}|\mathcal A_t^{\mathrm{cont}}f|$ on $L^p(\mathbb R^d)$ is bounded
independently of $d$; see \cite{Stein-1982,Stein-1983,SteinStromberg-1983}.
This initiated a broader program for maximal averages over high-dimensional
convex bodies.  Bourgain proved the dimension-free $L^2$ theorem for
arbitrary symmetric convex bodies.  Bourgain and, independently, Carbery
subsequently obtained dimension-free bounds for $p>3/2$, while M\"uller
proved the full range $p>1$ for the $\ell^q$ balls, $1\le q<\infty$; see
\cite{Bourgain-AJM-1986,Bourgain-Lp-1986,Carbery-1986,DGM-survey,Muller-1990}.
In particular, the continuous Euclidean-ball problem is completely understood
for $p>1$.

The discrete analogue is qualitatively different.  Put
\[
 \Ball_{d,t}=tB_2^d\cap\mathbb Z^d,
 \qquad
 \Avg_t f(x)=\frac1{|\Ball_{d,t}|}
   \sum_{y\in\Ball_{d,t}}f(x-y),
 \qquad t\ge0.
\]
For each fixed dimension the associated maximal operator is bounded on
$\ell^p(\mathbb Z^d)$ for $p>1$, but the usual covering arguments produce
constants which deteriorate with $d$.  More importantly, the discrete theory
cannot be dimension-free uniformly over all symmetric convex bodies.  In
2019 Bourgain, Mirek, Stein and Wr\'obel proved dimension-free estimates for
discrete cubes, but also constructed a family of ellipsoids for which the
full discrete maximal norm grows with the dimension \cite{BMSW-cubes}.
Thus the Euclidean ball must be treated using its particular arithmetic
structure.

A question attributed to E.~M. Stein and dating from the mid-1990s asks
whether the Euclidean-ball maximal operator satisfies the dimension-free
Hilbert-space estimate
\begin{equation}
 \sup_{d\ge1}
 \left\|\sup_{t\ge0}|\Avg_t|\right\|_{
 \ell^2(\mathbb Z^d)\to\ell^2(\mathbb Z^d)}<\infty.
 \label{eq:stein-question}
\end{equation}
The question is stated explicitly in \cite[Section~1.1, (1.3)]{MSW-spherical}
and \cite[Section~1.1, (SQ)]{NW-small}.  Several substantial parts have
since been resolved.  Bourgain, Mirek, Stein and Wr\'obel proved the
dimension-free estimate for Euclidean balls at dyadic radii on $\ell^p$,
$p\ge2$ \cite{BMSW-dyadic}, and their continuous--discrete comparison gives
the full maximal estimate at large radii $t\ge Cd$ \cite{BMSW-large}.
Further comparison and restricted-scale results were obtained by Kosz,
Mirek, Plewa and Wr\'obel \cite{KMPW-2023}.  Mirek, Szarek and Wr\'obel later
proved the corresponding dimension-free dyadic theorem for discrete
Euclidean spheres \cite{MSW-spherical}.  Niksi\'nski and Wr\'obel
proved the full, rather than dyadic, Euclidean-ball estimate on $\ell^p$,
$p\ge2$, in the genuinely small range
\[
 0\le t\le d^{1/2-\varepsilon}
\]
for every fixed $\varepsilon>0$ \cite{NW-small}.  More recently Niksi\'nski
developed a broader small-scale and dyadic theory for discrete
$1$-symmetric convex bodies \cite{Niksinski-1symm}; for
Euclidean balls this does not cover the full intermediate non-dyadic range.
Thus, after the preceding results, the unresolved full-radius problem lies
between the small-scale range $t\le d^{1/2-\varepsilon}$, for each fixed
$\varepsilon>0$, and the large-radius range $t\ge Cd$.

Passing from dyadic radii to the full supremum requires additional control in
the intermediate range.  If $n=\lfloor t^2\rfloor$, a dyadic block of squared
radii may contain polynomially many integer levels when
$d^{1-2\varepsilon}\lesssim n\lesssim d^2$.  Our argument constructs
frequency-uniform approximations at every such level, with residual bounds
that can be summed over the relevant ranges.  The approximating multipliers
are finite combinations of normalized discrete Gaussians.

We prove the following estimate, which fills this gap and also extends the
full-radius bound below $p=2$.

\begin{theorem}[Full-radius Euclidean-ball maximal inequalities]
\label{thm:main}
For every $1<p\le\infty$ there exists $C_p<\infty$ such that, for every
$d\ge1$ and every $f\in\ell^p(\mathbb Z^d)$,
\[
  \left\|\sup_{t\ge0}|\Avg_t f|\right\|_{\ell^p(\mathbb Z^d)}
  \le C_p\|f\|_{\ell^p(\mathbb Z^d)}.
\]
The constant $C_p$ is independent of $d$.
\end{theorem}

\begin{remark}[Comparison with Hormozi--Niksi\'nski--Wr\'obel]
\label{rem:hnw-comparison}
Hormozi, Niksi\'nski and Wr\'obel \cite[Theorem~1]{HNW-2026} prove the
same full-radius ball inequality as \cref{thm:main}, for every
$1<p\le\infty$. We acknowledge that credit for resolving this problem,
including Stein's question \eqref{eq:stein-question}, belongs to them.
Their paper also proves the full spherical maximal inequality for $d\ge5$
and $2\le p\le\infty$ \cite[Theorem~2]{HNW-2026}.

The proofs share the comparison with normalized discrete Gaussians and its
two frequency branches, at zero and at the parity frequency
$(1/2,\ldots,1/2)$. Both start from exact theta-function generating
formulas at integer squared radii, develop the approximation to arbitrary
fixed order, and use finite differences to express the correction terms
as finite combinations of Gaussian multipliers with bounded coefficients.
The resulting remainders have decaying $\ell^2$ bounds and uniform
$\ell^1$ bounds; interpolation followed by summation over squared radii
gives the range $1<p<2$.

The saddle choice and the estimates used to carry out this program differ.
Writing $h(z)=\sum_{k\in\mathbb Z}z^{k^2}$ and
$\mu(u)=uh'(u)/h(u)$, our saddle at squared radius $n\ge1$ satisfies
$d\mu(\rho)=n$, with $(1-z)^{-1}$ treated as an amplitude. HNW include
this factor in the phase and choose the saddle of the complete generating
function $(1-z)^{-1}h(z)^d$, namely
\[
 d\mu(r)+\frac{r}{1-r}=n.
\]
Our estimates separate small, critical and growing values of $n/d$;
HNW obtain a single approximation on $1\le n\le d^4$, with error
$C_N\min(n,d)^{-N}$ at order $N$ \cite[Proposition~4]{HNW-2026}.
For the growing range, we control the noncentral arcs by quadratic Weyl
estimates and Gauss sums. HNW instead use a quantitative quadratic-phase
rigidity estimate for localization and Jacobi's product formula for
complex theta-quotient bounds \cite[Sections~2.4 and~2.6]{HNW-2026}.
We also give a separate first-order two-Gaussian argument for the full
$\ell^2$ bound before treating higher orders and $1<p<2$.
These differences concern the proof; the final ball maximal inequality
is the same.
\end{remark}

In particular, \cref{thm:main} answers \eqref{eq:stein-question}.  The range
$2<p<\infty$ follows by interpolation with the $\ell^\infty$ contraction.  For
$1<p<2$, we interpolate the dimension-uniform $\ell^1$ bound for each
fixed-level residual with its decaying $\ell^2$ bound and then sum over $n$;
no $\ell^1$ maximal estimate is used.

\subsection{Squared radii and proof strategy}
\label{subsec:proof-strategy-intro}

Since $|y|^2$ is an integer for $y\in\mathbb Z^d$,
\[
 tB_2^d\cap\mathbb Z^d
 =\{y\in\mathbb Z^d:|y|^2\le\lfloor t^2\rfloor\}.
\]
Thus the natural radius variable is the integer squared radius
\[
 B_{d,n}=\{y\in\mathbb Z^d:|y|^2\le n\},
 \qquad n\in\mathbb N_0,
\]
and we write
\[
 M_{d,n}f(x)=\frac1{|B_{d,n}|}\sum_{y\in B_{d,n}}f(x-y),
 \qquad
 m_{d,n}(\xi)=\frac1{|B_{d,n}|}
  \sum_{y\in B_{d,n}}e^{2\pi i y\cdot\xi}
\]
for the normalized average and its Fourier multiplier.  Then
$\Avg_t=M_{d,\lfloor t^2\rfloor}$.  The level $n=0$ is the identity, so for
$n\ge1$ we put
\[
 \alpha=\frac nd.
\]

The one-dimensional theta series
\[
 h(z,x)=\sum_{k\in\mathbb Z}z^{k^2}e^{2\pi i kx},
 \qquad h(z)=h(z,0),
\]
encodes the exact ball coefficients.  For every $\alpha>0$ there is a unique
$\rho=\rho(\alpha)\in(0,1)$ satisfying
\[
 \mu(\rho):=\rho\frac{h'(\rho)}{h(\rho)}=\alpha.
\]
The origin of this choice is already visible in Cauchy's coefficient
formula.  On the circle $z=\rho e^{it}$ the exponential theta factor has
phase
\[
 d\{\log h(\rho e^{it})-\log h(\rho)\}-int,
\]
whose derivative at $t=0$ equals $i\{d\mu(\rho)-n\}$.  Thus
$\mu(\rho)=n/d$ makes $t=0$ stationary for the exponential part of the
coefficient integral.  We therefore refer to $\rho(\alpha)$ as the positive
saddle parameter.  This saddle point and the theta-function representation
of the ball counts are classical: Mazo and Odlyzko \cite{MazoOdlyzko-1990}
used the same saddle parameter to estimate the number of lattice points in
balls of radius $\sqrt{\alpha d}$.  In the normalized multiplier problem
considered here, the arcs at angles $0$ and $\pi$ are the two dominant saddle
arcs; the remaining rational cusps are separated by the estimates of
\cref{sec:arithmetic-remote}.

For normalized multipliers, the approximation must be uniform in $\xi$ and
must retain the parity contribution; the noncentral arcs require a separate
estimate.  To reach $1<p<2$, we continue the expansion to arbitrary order and
realize its correction terms by Gaussian multipliers.  Define
\[
 \phi_\rho(x)=\frac{h(\rho,x)}{h(\rho)},
 \qquad
 \Gamma_{d,\rho}(\xi)=\prod_{j=1}^d\phi_\rho(\xi_j),
 \qquad
 \boldsymbol\omega=(1/2,\ldots,1/2),
\]
and set
\[
 a_{d,n}=m_{d,n}(\boldsymbol\omega),\qquad |a_{d,n}|\le1.
\]
The exact coefficient representation has two distinguished saddle arcs, at
angles $0$ and $\pi$.  They give the first-order two-saddle model
\begin{equation}
 \mathsf G_{d,n}(\xi)
 =\Gamma_{d,\rho(\alpha)}(\xi)
 +a_{d,n}\Gamma_{d,\rho(\alpha)}(\xi-\boldsymbol\omega).
 \label{eq:intro-model}
\end{equation}
The second branch is forced by the parity identity $k^2\equiv k\pmod2$.

The proof separates three finite ranges.  In the small range
$n\le\alpha_0d$, one has $\rho\simeq n/d$ and the local saddle scale is
$n^{-1/2}$.  In a fixed critical window $ad\le n\le bd$, the parameter
$\rho$ stays in a compact subset of $(0,1)$ and the scale is $d^{-1/2}$.
In the growing range, from a sufficiently large multiple of $d$ up to
$Ad^2$, writing $\rho=e^{-\tau}$ gives $\tau\simeq d/n$ and the scale is
$\tau d^{-1/2}$.  The remaining radii $t\ge C_Ld$ are covered by the known
large-radius theorem.

At first order we prove, for suitable fixed constants,
\begin{align}
 \sup_\xi|m_{d,n}(\xi)-\mathsf G_{d,n}(\xi)|
 &\le C(n^{-1}+e^{-cn}),
 &&1\le n\le\alpha_0d, \label{eq:intro-small}\\
 &\le C_{a,b}d^{-1},
 &&ad\le n\le bd, \label{eq:intro-critical}\\
 &\le C_A d^{-1},
 &&\alpha_1(A)d\le n\le Ad^2. \label{eq:intro-growing}
\end{align}
These estimates already give the full-radius $\ell^2$ theorem.  If
$R_{d,n}$ denotes the multiplier operator with symbol
$m_{d,n}-\mathsf G_{d,n}$, then Plancherel and
\[
 \left\|\sup_n|R_{d,n}f|\right\|_2^2
 \le\sum_n\|R_{d,n}f\|_2^2
\]
reduce the residual maximal estimate to summation over the squared-radius
levels.
The small errors are square summable, the critical interval has $O(d)$
levels with error $O(d^{-1})$, and the growing interval has $O(d^2)$ levels
with the same error.  The operators associated with the model \eqref{eq:intro-model} are controlled by the
dimension-free maximal theorem for normalized discrete Gaussians of Mirek,
Szarek and Wr\'obel \cite{MSW-gaussian}.

The central analysis in the growing range is still a saddle calculation, but
one must also exclude all noncentral theta arcs.  Poisson summation gives
\begin{equation}
 h(e^{-w},x)=\sqrt{\frac\pi w}
   \sum_{m\in\mathbb Z}e^{-\pi^2(m+x)^2/w},
 \qquad \Re w>0,
 \label{eq:intro-poisson}
\end{equation}
and the key remote estimate is a uniform gap
\begin{equation}
 \sup_{x\in\mathbb R}
 \frac{|h(e^{-\tau+it},x)|}{h(e^{-\tau})}
 \le1-\eta,
 \quad
 0<\tau\le\tau_0,
 \quad
 \operatorname{dist}(t,\pi\mathbb Z)\ge C_0\tau.
 \label{eq:intro-remote-gap}
\end{equation}
Large rational denominators are controlled by Abel summation and the
quadratic Weyl-sum estimate \eqref{eq:weyl-sum} of Bombieri and Bourgain,
while bounded denominators are controlled by residue classes and quadratic
Gauss sums.  After taking the
$d$-fold product, the resulting factor $(1-\eta)^d$ absorbs all polynomial
losses in the range $n\le Ad^2$.

To reach $1<p<2$ we carry the same two-saddle expansion to arbitrary fixed
order.  For each fixed integer $K\ge1$ we construct a finite Gaussian model
in each regime, keeping the constructions separate where the ranges overlap.
Above fixed lower thresholds, the multiplier residual is
\[
 \lesssim_K n^{-K}
 \quad\text{in the small range},
 \qquad
 \lesssim_K d^{-K}
 \quad\text{in the critical and growing ranges}.
\]
The higher-order corrections are radial derivatives of
$\Gamma_{d,\rho}$.  A structural order bound shows that a term of asymptotic
order $j$ contains at most $2j$ such derivatives.  Finite differences at the
natural saddle spacing can therefore replace all derivative corrections by
bounded combinations of genuine normalized Gaussians without degrading the
order of the remainder.  The Gaussian maximal theorem then controls the
model operators.

At each fixed squared-radius level, the residual operator has a uniform
$\ell^1$ bound, while its $\ell^2$ norm is bounded by
$C_Kn^{-K}$ in the small range and by $C_Kd^{-K}$ in the other two finite
ranges.  Riesz--Thorin gives the corresponding $\ell^p$ decay with exponent
$2K(p-1)/p$.  After taking the $p$th power, the longest finite range contains
only $O(d^2)$ squared-radius levels.  Hence $K(p-1)\ge1$ suffices, and
$K=\lceil(p-1)^{-1}\rceil$ yields every fixed $1<p<2$.

\paragraph{Organization.}
\Cref{sec:setup,sec:representation} give the squared-radius reduction, the
coefficient and saddle representations, and the external Gaussian and
large-radius inputs.  The three central regimes are treated in
\cref{sec:scales}.  \Cref{sec:arithmetic-remote} controls the growing
remote arcs and completes the proof of the $\ell^2$ estimate.  The
arbitrary-order expansion and the final $\ell^p$ assembly are in
\cref{sec:higher-order,sec:assembly}.
\section{Squared-radius balls and Fourier notation}
\label{sec:setup}

Throughout the paper, $d\in\N$ and $\Z^d$ is equipped with counting
measure.  The Euclidean norm is denoted by $\abs{\cdot}$, and
\[
 B_2^d=\{x\in\R^d:\abs{x}\le1\}.
\]
For $t\ge0$ set
\[
 \Ball_{d,t}=tB_2^d\cap\Z^d,
 \qquad
 \Avg_t f(x)=\frac1{\abs{\Ball_{d,t}}}
   \sum_{y\in\Ball_{d,t}}f(x-y).
\]
We write $\T^d=(\R/\Z)^d$ and use the character
$e(u)=\exp(2\pi i u)$.  Our Fourier transform convention is
\begin{equation}
 \widehat f(\xi)=\sum_{x\in\Z^d} f(x)e(x\cdot\xi),
 \qquad
 f(x)=\int_{\T^d}\widehat f(\xi)e(-x\cdot\xi)\,d\xi,
 \label{eq:fourier-convention}
\end{equation}
so that Plancherel's identity reads
$\|f\|_{\ell^2(\Z^d)}=\|\widehat f\|_{L^2(\T^d)}$.

\begin{definition}[integer squared-radius ball]
\label{def:squared-radius-ball}
For $n\in\N_0$, define
\[
  B_{d,n}=\{y\in\Z^d:\abs{y}^2\leq n\},
  \qquad
  M_{d,n}f(x)=\frac{1}{\abs{B_{d,n}}}
    \sum_{y\in B_{d,n}}f(x-y).
\]
With the convention \eqref{eq:fourier-convention}, the Fourier multiplier of
$M_{d,n}$ is
\[
  m_{d,n}(\xi)=\frac{1}{\abs{B_{d,n}}}
    \sum_{y\in B_{d,n}}e(y\cdot\xi),
  \qquad \xi\in\T^d.
\]
\end{definition}

Since $0\in B_{d,n}$, every $B_{d,n}$ is a nonempty finite set.  More
explicitly,
\[
 M_{d,n}f=f*K_{d,n},\qquad
 K_{d,n}(y)=\frac{\mathbf 1_{B_{d,n}}(y)}{|B_{d,n}|},\qquad
 K_{d,n}\ge0,\quad \sum_{y\in\mathbb Z^d}K_{d,n}(y)=1.
\]
Thus $M_{d,n}$ is convolution with a probability kernel, and Young's
inequality gives
\begin{equation}
 \|M_{d,n}f\|_{\ell^p}\le \|f\|_{\ell^p},
 \qquad 1\le p\le\infty.
 \label{eq:ball-contraction}
\end{equation}
Whenever an asymptotic argument below begins at a fixed level $n\ge n_0$,
the integer $n_0$ is chosen after the fixed regime parameters (and, in
Section~\ref{sec:higher-order}, after the expansion order) have been fixed;
it is independent of $d$ and $n$.  The finitely many levels $n<n_0$ are
handled directly by \eqref{eq:ball-contraction}.

\begin{lemma}[exact real-to-integer reduction]
\label{lem:integer-radius-reduction}
For every $t\geq 0$,
\[
  tB_2^d\cap\Z^d=B_{d,\lfloor t^2\rfloor}.
\]
Consequently,
\[
  \sup_{t\geq0}\abs{\Avg_t f(x)}
  =\sup_{n\in\N_0}\abs{M_{d,n}f(x)}
\]
for every $x\in\Z^d$.
\end{lemma}

\begin{proof}
For $y\in\Z^d$, the condition $y\in tB_2^d$ means $\abs{y}^2\leq t^2$, and
$\abs{y}^2$ is an integer; hence the condition is equivalent to
$\abs{y}^2\leq\lfloor t^2\rfloor$.  Conversely, $B_{d,n}$ is attained at
$t=\sqrt n$, so every integer squared-radius level occurs in the real-radius
family, and the averages agree under this identification.
\end{proof}

\section{Coefficient extraction and the two-saddle representation}
\label{sec:representation}

\subsection{Exact coefficients and parity}

For $|z|<1$ and $x\in\mathbb T$, put
\[
 h(z,x)=\sum_{k\in\mathbb Z}z^{k^2}e(kx),
 \qquad h(z)=h(z,0),
 \qquad e(u)=e^{2\pi i u},
\]
and
\[
 A_+(z)=\frac1{1-z},\qquad A_-(z)=\frac1{1+z}.
\]
All series converge absolutely on compact subsets of the unit disk, so
products and coefficient extraction can be performed term by term.

\Needspace{5\baselineskip}
\begin{proposition}[coefficient identities]
\label{prop:coefficient}
For $d\ge1$, $n\ge0$, and $\xi\in\mathbb T^d$,
\begin{align}
 |B_{d,n}|&=[z^n]A_+(z)h(z)^d, \label{eq:card-coeff}\\
 m_{d,n}(\xi)&=
 \frac{[z^n]A_+(z)\prod_{j=1}^d h(z,\xi_j)}
 {[z^n]A_+(z)h(z)^d}. \label{eq:mult-coeff}
\end{align}
Let $\boldsymbol\omega=(1/2,\ldots,1/2)$ and $a_{d,n}=m_{d,n}(\boldsymbol\omega)$.  Then
\begin{equation}
 a_{d,n}=(-1)^n
 \frac{[z^n]A_-(z)h(z)^d}{[z^n]A_+(z)h(z)^d},
 \qquad |a_{d,n}|\le1.
 \label{eq:parity-coeff}
\end{equation}
\end{proposition}

\begin{proof}
We first expand the product before extracting the coefficient:
\[
 A_+(z)\prod_{j=1}^d h(z,\xi_j)
 =\left(\sum_{r\ge0}z^r\right)
   \prod_{j=1}^d\left(\sum_{y_j\in\mathbb Z}
     z^{y_j^2}e(y_j\xi_j)\right)
 =\sum_{r\ge0}\sum_{y\in\mathbb Z^d}
   z^{r+|y|^2}e(y\cdot\xi).
\]
For fixed $y$, the equation $r+|y|^2=n$ has the unique admissible solution
$r=n-|y|^2$ when $|y|^2\le n$, and no solution otherwise.  Hence
\[
 [z^n]A_+(z)\prod_{j=1}^d h(z,\xi_j)
 =\sum_{|y|^2\le n}e(y\cdot\xi).
\]
At $\xi=0$ this is $|B_{d,n}|$, proving \eqref{eq:card-coeff}; division by
this positive coefficient gives \eqref{eq:mult-coeff}.

For the parity identity, $k^2\equiv k\pmod2$ gives
\begin{equation}
 h(-z,x)=\sum_{k\in\mathbb Z}z^{k^2}(-1)^{k^2}e(kx)
 =\sum_{k\in\mathbb Z}z^{k^2}(-1)^ke(kx)
 =h(z,x+1/2).
 \label{eq:parity-theta}
\end{equation}
At $\boldsymbol\omega=(1/2,\ldots,1/2)$,
\[
 e(y\cdot\boldsymbol\omega)
 =(-1)^{y_1+\cdots+y_d}=(-1)^{|y|^2}.
\]
On the other hand, $A_-(z)=\sum_{r\ge0}(-1)^rz^r$, so
\[
 [z^n]A_-(z)h(z)^d
 =\sum_{|y|^2\le n}(-1)^{n-|y|^2}
 =(-1)^n\sum_{|y|^2\le n}(-1)^{|y|^2}.
\]
Combining the last two displays with \eqref{eq:mult-coeff} proves
\eqref{eq:parity-coeff}.  Finally,
\[
 |a_{d,n}|\le \frac1{|B_{d,n}|}\sum_{y\in B_{d,n}}
 |e(y\cdot\boldsymbol\omega)|=1.
\]
\end{proof}

The coefficient $a_{d,n}$ in \eqref{eq:parity-coeff} uses the full $A_-$
integral.  Its comparison with the negative central arc will follow from the
remote estimates.

\subsection{The saddle and the Gaussian model}

For $0<\rho<1$ define
\[
 \mu(\rho)=\rho\frac{h'(\rho)}{h(\rho)}.
\]
Under the probability weights $\rho^{k^2}/h(\rho)$, the quantity
$\mu(\rho)$ is the mean of $k^2$, and
\[
 v(\rho)=\rho\mu'(\rho)
\]
is its variance.  The distribution is nondegenerate for $0<\rho<1$, so
$v(\rho)>0$ and $\mu$ is strictly increasing.  Since
$\mu(\rho)\to0$ as $\rho\downarrow0$ and $\mu(\rho)\to\infty$ as
$\rho\uparrow1$, for every $\alpha>0$ there is a unique
$\rho(\alpha)\in(0,1)$ such that
\begin{equation}
 \mu(\rho(\alpha))=\alpha. \label{eq:saddle-map}
\end{equation}

Set
\[
 \phi_\rho(x)=\frac{h(\rho,x)}{h(\rho)},
 \qquad
 \Gamma_{d,\rho}(\xi)=\prod_{j=1}^d\phi_\rho(\xi_j).
\]
For $n>0$ and $\rho=\rho(n/d)$, define the two-saddle model and the residual
by
\begin{equation}
 \mathsf G_{d,n}(\xi)=
 \Gamma_{d,\rho}(\xi)+
 a_{d,n}\Gamma_{d,\rho}(\xi-\boldsymbol\omega),
 \qquad
 r_{d,n}=m_{d,n}-\mathsf G_{d,n}.
 \label{eq:two-saddle-model}
\end{equation}
Let $R_{d,n}$ denote the Fourier multiplier operator with symbol $r_{d,n}$:
\begin{equation}
 \widehat{R_{d,n}f}(\xi)=r_{d,n}(\xi)\widehat f(\xi).
 \label{eq:residual-operator}
\end{equation}

\subsection{Normalized Cauchy integrals}
\label{subsec:normalized-contours}

Fix $n>0$ and $\rho=\rho(n/d)$ as above.  It is useful to keep the
global coefficient integrals in a form which does not require division by
a theta factor.  For an arc $E\subset[-\pi,\pi]$, set
\begin{align}
 J_+(\xi;E)&=\frac1{2\pi}
 \int_E
 \frac{\prod_{j=1}^d h(\rho e^{it},\xi_j)}{h(\rho)^d}
 \frac{e^{-int}}{1-\rho e^{it}}\,dt,
 \label{eq:Jplus}\\
 J_-(\xi;E)&=\frac1{2\pi}
 \int_E
 \frac{\prod_{j=1}^d h(\rho e^{it},\xi_j)}{h(\rho)^d}
 \frac{e^{-int}}{1+\rho e^{it}}\,dt.
 \label{eq:Jminus}
\end{align}
These integrals are defined on the whole circle.  Cauchy's formula and
\eqref{eq:saddle-map} give
\begin{equation}
 [z^n]A_+(z)\prod_jh(z,\xi_j)
 =\rho^{-n}h(\rho)^d J_+(\xi;[-\pi,\pi]).
 \label{eq:cauchy-main}
\end{equation}
In particular,
\begin{equation}
 m_{d,n}(\xi)=
 \frac{J_+(\xi;[-\pi,\pi])}{J_+(0;[-\pi,\pi])}.
 \label{eq:full-ratio}
\end{equation}
The same computation with $A_-$ in place of $A_+$ gives
\begin{equation}
 [z^n]A_-(z)h(z)^d
 =\rho^{-n}h(\rho)^dJ_-(0;[-\pi,\pi]),
 \label{eq:Aminus-J}
\end{equation}
and hence
\begin{equation}
 a_{d,n}=(-1)^n
 \frac{J_-(0;[-\pi,\pi])}{J_+(0;[-\pi,\pi])}.
 \label{eq:parity-J}
\end{equation}

The saddle factorization will only be used on local zero-free arcs.  Let
$E$ be an arc around $0$ on which $h(\rho e^{it},0)\neq0$.  Choose the
branch of
\[
 L(s+it)=\log h(e^{s+it},0),\qquad s=\log\rho,
\]
which agrees with the real logarithm at $t=0$, and put
\[
 \Phi_\rho(t)=L(s+it)-L(s)-i\mu(\rho)t,
 \qquad
 P_\rho(t,\xi)=
 \prod_{j=1}^d
 \frac{h(\rho e^{it},\xi_j)}{h(\rho e^{it},0)}.
\]
Since $d\mu(\rho)=n$, one has on such an arc
\begin{equation}
 \frac{\prod_{j=1}^d h(\rho e^{it},\xi_j)}{h(\rho)^d}e^{-int}
 =e^{d\Phi_\rho(t)}P_\rho(t,\xi).
 \label{eq:local-saddle-factorization}
\end{equation}
Consequently, on a zero-free central arc $E$,
\begin{align}
 J_+(\xi;E)&=\frac1{2\pi}
 \int_E e^{d\Phi_\rho(t)}
 \frac{P_\rho(t,\xi)}{1-\rho e^{it}}\,dt,
 \label{eq:Jplus-local}\\
 J_-(\xi;E)&=\frac1{2\pi}
 \int_E e^{d\Phi_\rho(t)}
 \frac{P_\rho(t,\xi)}{1+\rho e^{it}}\,dt.
 \label{eq:Jminus-local}
\end{align}
The zero-free assertions needed for these local factorizations are proved
in the corresponding small, critical, and growing regimes below.  The
global identities \eqref{eq:cauchy-main}--\eqref{eq:parity-J} do not use
any zero-free statement.

If $E_0$ is a symmetric arc around $0$ and $E_\pi=\pi+E_0$ is understood
as an arc modulo $2\pi$, then \eqref{eq:parity-theta} and the change of
variables $t=\pi+u$ give
\begin{equation}
 \int_{E_\pi}
 \frac{\prod_jh(\rho e^{it},\xi_j)}{1-\rho e^{it}}e^{-int}\,dt
 =(-1)^n
 \int_{E_0}
 \frac{\prod_jh(\rho e^{iu},\xi_j+1/2)}{1+\rho e^{iu}}e^{-inu}\,du.
 \label{eq:negative-arc-transport}
\end{equation}
Since $\xi+\boldsymbol\omega=\xi-\boldsymbol\omega$ on $\mathbb T^d$, this
change of variables gives the translated branch in
\eqref{eq:two-saddle-model}.  Recovering its coefficient from the full $A_-$
integral also requires control of the remote arcs.

\subsection{Estimates from the literature}

\paragraph{Normalized discrete Gaussians.}
We use the following direct specialization of the maximal assertion in
\cite[Theorem~1.1, (1.5)]{MSW-gaussian}.

\begin{proposition}[normalized discrete-Gaussian maximal input]
\label{prop:gaussian-input}
For every $1<p<\infty$ there is $C_p<\infty$ such that, for every
$d\ge1$ and every $f\in\ell^p(\mathbb Z^d)$,
\begin{equation}
 \left\|\sup_{0<\rho<1}|G_{d,\rho}f|\right\|_p
 \le C_p\|f\|_p. \label{eq:known-gaussian}
\end{equation}
The constant $C_p$ is independent of $d$.  Here $G_{d,\rho}$ is
convolution with
\[
 k\longmapsto \frac{\rho^{|k|^2}}{h(\rho)^d}.
\]
\end{proposition}

\begin{proof}
The kernel in \cite[Theorem~1.1]{MSW-gaussian} is
\[
 g_t(k)=\frac{e^{-\pi|k|^2/t}}
 {\sum_{m\in\mathbb Z^d}e^{-\pi|m|^2/t}},
 \qquad t>0.
\]
Set
\[
 \tau=-\log\rho,
 \qquad
 t=\frac{\pi}{\tau}.
\]
Then
\[
 e^{-\pi|k|^2/t}=e^{-\tau|k|^2}=\rho^{|k|^2},
 \qquad
 \sum_{m\in\mathbb Z^d}e^{-\pi|m|^2/t}=h(\rho)^d.
\]
Thus $g_t$ is exactly the kernel of $G_{d,\rho}$, and $t\in(0,\infty)$
corresponds bijectively to $\rho\in(0,1)$.  The maximal inequality
\cite[Theorem~1.1, (1.5)]{MSW-gaussian} gives
\eqref{eq:known-gaussian} for every $1<p<\infty$.
\end{proof}

\paragraph{The translated Gaussian branch.}
Define the parity modulation
\begin{equation}
 (U_{\boldsymbol\omega}f)(x)
 =(-1)^{x_1+\cdots+x_d}f(x).
 \label{eq:parity-modulation}
\end{equation}
The operator $U_{\boldsymbol\omega}$ is self-inverse and isometric on every
$\ell^p(\mathbb Z^d)$, $1\le p\le\infty$.  With the Fourier convention
\eqref{eq:fourier-convention},
\[
 \widehat{U_{\boldsymbol\omega}f}(\xi)
 =\widehat f(\xi-\boldsymbol\omega),
\]
so $U_{\boldsymbol\omega}G_{d,\rho}U_{\boldsymbol\omega}$ has multiplier
$\Gamma_{d,\rho}(\xi-\boldsymbol\omega)$.  Consequently
\begin{equation}
 \left\|\sup_{0<\rho<1}
 |U_{\boldsymbol\omega}G_{d,\rho}U_{\boldsymbol\omega}f|\right\|_p
 \le C_p\|f\|_p,
 \qquad 1<p<\infty.
 \label{eq:translated-gaussian-maximal}
\end{equation}
\paragraph{The large-radius ball theorem.}
Bourgain, Mirek, Stein and Wr\'obel \cite[Theorem~2]{BMSW-large} (see
also \cite[Theorem~2.1]{BMSW-dyadic}) proved that for some absolute
constant $C_L>0$ and every $1<p<\infty$ there is $C_p<\infty$ such that
\begin{equation}
 \left\|\sup_{t\ge C_Ld}|\Avg_tf|\right\|_p
 \le C_p\|f\|_p, \label{eq:known-large-radius}
\end{equation}
uniformly in $d$.

\paragraph{A finite theta estimate.}
We use the following bound for quadratic Weyl sums with a linear term; it
is uniform in the linear coefficient and in the length of the sum.  By
Bombieri and Bourgain \cite[Lemma~20]{BombieriBourgain-2009}, there is an
absolute constant $C_{\mathrm W}$ such that if $\alpha,\theta\in\mathbb R$,
$a,Q\in\mathbb Z$, $Q\ge1$, $(a,Q)=1$, and
\[
 \left|\alpha-\frac aQ\right|<3Q^{-2},
\]
then, for every integer $L\ge1$,
\begin{equation}
 \left|
 \sum_{k=1}^{L}
 e(\alpha k^2+\theta k)
 \right|
 \le C_{\mathrm W}\left(\frac L{\sqrt Q}+\sqrt Q\right).
 \label{eq:weyl-sum}
\end{equation}
In \cite{BombieriBourgain-2009} the hypothesis reads
$|\alpha-a/Q|<B'Q^{-2}$ with an arbitrary $B'\ge1$ and the bound is
$\ll\sqrt Q+B'L/\sqrt Q$; we have fixed $B'=3$.  For $B'=1$ and $Q\le4L$
the estimate is Theorem~6 of Fiedler, Jurkat and K\"orner
\cite[p.~143]{FJK}.  We apply \eqref{eq:weyl-sum} to sums of the form
$\sum_ke^{\pi i(k^2X+2k\theta)}=\sum_ke(\tfrac X2k^2+\theta k)$, so the
quadratic coefficient is $\alpha=X/2$ and, if $X$ is approximated by $p/q$,
the relevant denominator $Q$ is that of $p/(2q)$ in lowest terms; the
conversion is carried out in \cref{lem:weyl-abel}.  We use \eqref{eq:weyl-sum}
only in the large-denominator part of the growing remote estimate.

We shall also use Dirichlet's rational-approximation theorem in its standard
pigeonhole form.
\section{The three finite scales}
\label{sec:scales}

Fix $d\ge1$, $n\ge1$, set $\alpha=n/d$, and let
$\rho=\rho(\alpha)$ be given by \eqref{eq:saddle-map}.  We write
$\kappa_j(\rho)$ for the $j$th cumulant of $k^2$ under the probability
measure
\[
 \mathbb P_\rho(k)=\frac{\rho^{k^2}}{h(\rho)},\qquad k\in\mathbb Z.
\]
Thus $\kappa_1=\mu$ and $\kappa_2=v=\rho\mu'$.  With $s=\log\rho$, let
$L(z)=\log h(e^z,0)$ denote the analytic branch near $s$.  Then
\begin{equation}
 L(s+z)-L(s)
 =\log\sum_{k\in\mathbb Z}\mathbb P_\rho(k)e^{zk^2},
 \label{eq:cumulant-generating-identity}
\end{equation}
and hence, for every $j\ge1$,
\begin{equation}
 L^{(j)}(s)=\kappa_j(\rho),\qquad
 \left.\partial_t^jL(s+it)\right|_{t=0}=i^j\kappa_j(\rho).
 \label{eq:cumulant-derivative-identity}
\end{equation}
The natural normalization of the cumulants and the local saddle variable
differ across the three ranges as follows:
\[
\begin{array}{c|c|c|c}
\text{range} & \text{saddle parameter} & \text{local variable} & N \\ \hline
n\ll d & \rho\simeq n/d & t=u/\sqrt n & n \\
 n\simeq d & \rho\asymp1 & t=u/\sqrt d & d \\
 n/d\gg1 & \rho=e^{-\tau},\ \tau\simeq d/n
            & t=\tau u/\sqrt d & d
\end{array}
\]

\subsection{A symmetric saddle comparison}

The same elementary cancellation is used at all three scales.  We isolate it
once so that the $N^{-1}$ gain in the central ratios is explicit.

\Needspace{8\baselineskip}
\begin{lemma}[symmetric saddle comparison]
\label{lem:symmetric-saddle}
Fix positive constants $q_0,q_1,c_0,C_0$ and $0<\eta\le1$.  There are
$N_0$ and $C$ with the following property.  Let $N\ge N_0$,
$U=\eta\sqrt N$, let $\Psi$ and $A$ be $C^2$ functions on $[-U,U]$, and
for every $\xi$ let $P(\,\cdot\,,\xi)$ be $C^2$ on $[-U,U]$.  Let $G(\xi)$
be a function of $\xi$ and let $E(\xi)\ge0$.  Suppose that, uniformly in
$\xi$,
\begin{align}
 \Re\Psi(u)&\le-c_0u^2, \label{eq:symmetric-phase-decay}\\
 \Psi(u)&=-\frac q2u^2+\frac{i b_3}{\sqrt N}u^3
          +\frac{b_4}{N}u^4+R(u), \label{eq:symmetric-phase-expansion}\\
 q_0\le q&\le q_1,\qquad |b_3|+|b_4|\le C_0,
 \qquad |R(u)|\le C_0\frac{|u|^5}{N^{3/2}}, \notag
\end{align}
where $q,b_3,b_4$ are real.  Assume also
\begin{align}
 A(0)&=1,\qquad |A(u)|\le C_0, \notag\\
 |A'(u)|&\le C_0N^{-1/2},\qquad
 |A''(u)|\le C_0N^{-1},
 \label{eq:symmetric-pole-hyp}\\
 P(0,\xi)&=G(\xi),\qquad P(u,0)=1, \notag\\
 |P(u,\xi)|&\le e^{-c_0E(\xi)}, \notag\\
 |\partial_uP(u,\xi)|&\le C_0N^{-1/2}(1+E(\xi))e^{-c_0E(\xi)}, \notag\\
 |\partial_u^2P(u,\xi)|&\le C_0N^{-1}(1+E(\xi))^2e^{-c_0E(\xi)}.
 \label{eq:symmetric-product-hyp}
\end{align}
Define
\[
 I(\xi)=\int_{-U}^Ue^{\Psi(u)}A(u)P(u,\xi)\,du.
\]
Then
\begin{align}
 I(0)&=\sqrt{\frac{2\pi}{q}}+O(N^{-1}),
 \label{eq:symmetric-base}\\
 |I(\xi)-G(\xi)I(0)|&\le CN^{-1}e^{-cE(\xi)}
 \label{eq:symmetric-ratio-difference}
\end{align}
for a constant $c>0$ depending only on the displayed constants.  In
particular,
\begin{equation}
 \left|\frac{I(\xi)}{I(0)}-G(\xi)\right|\le CN^{-1}.
 \label{eq:symmetric-ratio}
\end{equation}
\end{lemma}

\begin{proof}
Put
\[
 g(u)=e^{-qu^2/2},\qquad
 \Theta(u)=\Psi(u)+\frac q2u^2.
\]
Choose $M$ depending only on $c_0$ and let
$R_N=M\sqrt{\log N}$.  By \eqref{eq:symmetric-phase-decay},
\eqref{eq:symmetric-pole-hyp}, and \eqref{eq:symmetric-product-hyp},
\begin{equation}
 \int_{R_N<|u|\le U}|e^{\Psi(u)}A(u)P(u,\xi)|\,du
 \le CN^{-3}e^{-c_0E(\xi)}
 \label{eq:symmetric-tail}
\end{equation}
after increasing $M$; the same estimate holds with $P\equiv1$, and the
Gaussian tail of $g$ is $O(N^{-3})$.  Since
$|G(\xi)|=|P(0,\xi)|\le e^{-c_0E(\xi)}$, the same tail bound applies
to the difference $I(\xi)-G(\xi)I(0)$.

On $|u|\le R_N$, \eqref{eq:symmetric-phase-expansion} gives
\begin{equation}
 e^{\Theta(u)}
 =1+\frac{i b_3}{\sqrt N}u^3
 +O\!\left(\frac{u^4+u^6}{N}\right).
 \label{eq:symmetric-exp-expansion}
\end{equation}
Increase $N_0$, if necessary, so that for $N\ge N_0$,
\[
 M\sqrt{\log N}\le\eta\sqrt N,
 \qquad
 M\sqrt{\frac{\log N}{N}}\le1,
 \qquad
 C_0M^3\frac{(\log N)^{3/2}}{\sqrt N}\le\frac14.
\]
The first inequality guarantees $[-R_N,R_N]\subseteq[-U,U]$.
Then $|R(u)|\le C u^4/N$ on $|u|\le R_N$, and
$|\Theta(u)|\le1/2$ after one further increase of $N_0$.  Using
$|e^z-1-z|\le C|z|^2$ for $|z|\le1/2$ gives
\eqref{eq:symmetric-exp-expansion}; the linear contribution of $b_4u^4/N$
and the remainder are $O(u^4/N)$, while the square of the cubic term is
$O(u^6/N)$.  Taylor's formula and \eqref{eq:symmetric-pole-hyp} give
\begin{equation}
 A(u)=1+A'(0)u+O(u^2/N).
 \label{eq:symmetric-pole-taylor}
\end{equation}
Consequently
\begin{equation}
 e^{\Theta(u)}A(u)
 =1+A'(0)u+\frac{i b_3}{\sqrt N}u^3
 +O\!\left(\frac{1+|u|^6}{N}\right).
 \label{eq:symmetric-base-integrand}
\end{equation}
The two displayed first-order terms are odd in $u$.  Since $g$ is even,
their integrals over $[-R_N,R_N]$ vanish.  Thus
\[
 \int_{-R_N}^{R_N}e^{\Psi(u)}A(u)\,du
 =\int_{-R_N}^{R_N}g(u)\,du+O(N^{-1}),
\]
and \eqref{eq:symmetric-tail} proves \eqref{eq:symmetric-base}.

For the ratio, Taylor's formula and \eqref{eq:symmetric-product-hyp} give
\begin{align}
 P(u,\xi)-G(\xi)
 &=\partial_uP(0,\xi)u+\mathcal E_P(u,\xi),
 \label{eq:symmetric-product-taylor}\\
 |\mathcal E_P(u,\xi)|
 &\le \frac{C_0u^2}{N}(1+E(\xi))^2e^{-c_0E(\xi)}. \notag
\end{align}
Multiplying \eqref{eq:symmetric-exp-expansion},
\eqref{eq:symmetric-pole-taylor}, and
\eqref{eq:symmetric-product-taylor}, and using the bound for
$\partial_uP(0,\xi)$, yields
\begin{equation}
 e^{\Theta(u)}A(u)\{P(u,\xi)-G(\xi)\}
 =\partial_uP(0,\xi)u+\mathcal E(u,\xi),
 \label{eq:symmetric-difference-integrand}
\end{equation}
with
\begin{equation}
 |\mathcal E(u,\xi)|
 \le \frac C N(1+|u|^8)(1+E(\xi))^2e^{-c_0E(\xi)}.
 \label{eq:symmetric-difference-error}
\end{equation}
Again the first term in \eqref{eq:symmetric-difference-integrand} is odd.
Therefore
\begin{align*}
 &\left|\int_{-R_N}^{R_N}e^{\Psi(u)}A(u)
       \{P(u,\xi)-G(\xi)\}\,du\right|\\
 &\qquad\le \frac C N(1+E(\xi))^2e^{-c_0E(\xi)}
       \int_{\mathbb R}(1+|u|^8)e^{-q_0u^2/2}\,du
 \le \frac C N e^{-cE(\xi)}.
\end{align*}
Together with \eqref{eq:symmetric-tail}, this proves
\eqref{eq:symmetric-ratio-difference}.  Finally
\eqref{eq:symmetric-base} and $q\in[q_0,q_1]$ imply $|I(0)|\ge c>0$ for
$N\ge N_0$, and \eqref{eq:symmetric-ratio} follows.
\end{proof}

\subsection{Small \texorpdfstring{$\alpha$}{alpha}: the scale is \texorpdfstring{$n$}{n}}

The estimates in this subsection require only a fixed sufficiently small
upper bound for the saddle radius.  For convenience, choose
\[
 \rho_0=2^{-11},\qquad \alpha_0=\mu(\rho_0).
\]
Since $\mu$ is strictly increasing,
\[
 0<\alpha\le\alpha_0
 \quad\Longleftrightarrow\quad
 0<\rho(\alpha)\le\rho_0.
\]

\begin{lemma}[small-range saddle and theta estimates]
\label{lem:small-scale}
Let $0<\alpha\le\alpha_0$ and $\rho=\rho(\alpha)$.

\emph{(i) Saddle and cumulants.}  One has
\begin{equation}
 \frac{\alpha}{16}\le\rho\le\alpha,
 \qquad
 \frac{\alpha}{32}\le v(\rho)\le64\alpha.
 \label{eq:small-rho-var}
\end{equation}
Moreover, for $j=3,4,5$,
\begin{equation}
 |\kappa_j(\rho)|\le C_j\rho. \label{eq:small-cumulants}
\end{equation}
Consequently, with
\[
 \lambda_j(\alpha)=\frac{\kappa_j(\rho(\alpha))}{\alpha},
 \qquad j=2,3,4,5,
\]
one has
\begin{equation}
 \frac1{32}\le\lambda_2(\alpha)\le64,
 \qquad |\lambda_j(\alpha)|\le C_j',\quad j=3,4,5.
 \label{eq:small-normalized-cumulants}
\end{equation}
In particular $d\rho\simeq n$ and $dv(\rho)\simeq n$.

\emph{(ii) Theta quotient.}  If $0<\rho\le2^{-7}$ and
$|t|\le\pi/3$, then $|h(\rho e^{it},0)|\ge1/2$ and
\begin{equation}
 \left|\frac{h(\rho e^{it},x)}{h(\rho e^{it},0)}\right|
 \le \exp\!\left(-\frac{\rho}{32}\delta(x)\right),
 \label{eq:small-contraction}
\end{equation}
where $\delta(x)=1-\cos(2\pi x)$.

\emph{(iii) Phase contraction.}  There are absolute constants
$\eta_s,c_s>0$ such that, for $|t|\le\eta_s$,
\begin{equation}
 \Re\Phi_\rho(t)\le-c_s\rho t^2.
 \label{eq:small-phase-decay}
\end{equation}

\emph{(iv) Remote contraction.}  With the fixed angular width
\[
 \eta_*:=\frac{\pi}{6},
\]
there is an absolute constant $c_*>0$ such that
\begin{equation}
 \sup_{x\in\mathbb T}
 \frac{|h(\rho e^{it},x)|}{h(\rho)}
 \le1-c_*\rho
 \qquad\text{if }\operatorname{dist}(t,\pi\mathbb Z)\ge\eta_*.
 \label{eq:small-full-remote-contraction}
\end{equation}
The corresponding statements near $t=\pi$ follow from
\eqref{eq:parity-theta}.
\end{lemma}

\begin{proof}
We first prove (i).  For $0<\rho\le2^{-7}$, separating the terms with
$|k|=1$ gives
\[
 2\rho\le\sum_{k\in\mathbb Z}k^2\rho^{k^2}
 \le2\rho+2\sum_{k\ge2}k^2\rho^{k^2}
 \le2\rho+30\rho^4,
\]
and
\[
 1+2\rho\le h(\rho)\le1+2\rho+4\rho^4.
\]
Thus $\rho\le\mu(\rho)\le16\rho$.  To make the variance estimate
explicit, set
\[
 m_4(\rho)=\frac1{h(\rho)}\sum_{k\in\mathbb Z}k^4\rho^{k^2}.
\]
Separating the terms $|k|=1$ once more gives
\[
 2\rho\le\sum_{k\in\mathbb Z}k^4\rho^{k^2}
 \le2\rho+64\rho^4.
\]
Since $1\le h(\rho)\le4/3$ in the present range,
\[
 \frac32\rho\le m_4(\rho)\le3\rho.
\]
The preceding second-moment bound also gives
$\mu(\rho)\le2\rho+30\rho^4\le3\rho$, and therefore
$\mu(\rho)^2\le9\rho^2\le\rho/2$.  Hence
\[
 \rho\le v(\rho)=m_4(\rho)-\mu(\rho)^2\le3\rho.
\]
Together with $\rho\le\alpha=\mu(\rho)\le16\rho$, this stronger estimate
implies the variance bounds in \eqref{eq:small-rho-var}.  Each $\kappa_j$
with $j\le5$ is a fixed polynomial in the moments of order at most $10$.  The contribution of the
terms with $|k|=1$ is $O(\rho)$, and
\[
 \sum_{|k|\ge2}|k|^{10}\rho^{k^2}\le C\rho^4,
\]
which proves \eqref{eq:small-cumulants}.  Dividing by $\alpha\simeq\rho$
gives \eqref{eq:small-normalized-cumulants}.

We next prove (ii).  For $|z|\le1/8$,
\begin{equation}
 h(z,x)-h(z,0)=-2z\delta(x)+T(z,x),
 \qquad |T(z,x)|\le64|z|^4\delta(x).
 \label{eq:small-tail}
\end{equation}
Indeed, pairing $k$ with $-k$ gives
$h(z,x)-h(z,0)=-2\sum_{k\ge1}z^{k^2}\{1-\cos(2\pi kx)\}$, and
$1-\cos(2\pi kx)\le k^2\delta(x)$, while the tail with $k\ge2$ is
geometric.  On $|t|\le\pi/3$,
\[
 |h(\rho e^{it},0)|\ge1-2\rho-4\rho^4\ge\frac12,
 \qquad
 \Re\!\left(\rho e^{it}\overline{h(\rho e^{it},0)}\right)
 \ge\frac{3\rho}{8}.
\]
Put $z=\rho e^{it}$, $h_0=h(z,0)$, $\delta=\delta(x)$ and
$T=T(z,x)$.  From \eqref{eq:small-tail},
\begin{align*}
 |h_0|^2-|h_0-2z\delta+T|^2
 ={}&4\delta\Re(z\overline{h_0})-4\rho^2\delta^2
      -2\Re(T\overline{h_0})\\
 &+4\delta\Re(z\overline T)-|T|^2.
\end{align*}
The first term is at least $(3/2)\rho\delta$.  Since
$0\le\delta\le2$, $|h_0|\le h(\rho)\le2$, and
$|T|\le64\rho^4\delta$, the sum of the absolute values of the remaining
four terms is at most $(1/2)\rho\delta$ for $\rho\le2^{-7}$.  Thus
\[
 |h_0|^2-|h(z,x)|^2\ge\rho\delta.
\]
Using $|h_0|^2\le4$ gives
\[
 \left|\frac{h(\rho e^{it},x)}{h(\rho e^{it},0)}\right|^2
 \le1-\frac\rho4\delta(x).
\]
Taking square roots and weakening the constant yields
\eqref{eq:small-contraction}.

For (iii), write
\[
 h(\rho e^{it},0)=1+2\rho e^{it}+R_\rho(t),
 \qquad |R_\rho(t)|\le4\rho^4.
\]
Summing the absolutely convergent double series gives the exact identity
\[
 h(\rho)^2-|h(\rho e^{it},0)|^2
 =\sum_{k,\ell\in\mathbb Z}\rho^{k^2+\ell^2}
   \{1-\cos((k^2-\ell^2)t)\}.
\]
Every summand is nonnegative.  Keeping only the four pairs
$(k,\ell)=(0,\pm1),(\pm1,0)$ yields
\[
 h(\rho)^2-|h(\rho e^{it},0)|^2
 \ge4\rho(1-\cos t).
\]
Since $h(\rho)$ is bounded above by an absolute constant in the small range
and $1-\cos t\ge c_0t^2$ for $|t|\le\pi/3$,
\[
 \frac{|h(\rho e^{it},0)|^2}{h(\rho)^2}
 \le1-c\rho(1-\cos t)
 \le e^{-c'\rho t^2},\qquad |t|\le\pi/3.
\]
This is \eqref{eq:small-phase-decay} with $\eta_s=\pi/3$, because
$|e^{\Phi_\rho(t)}|=|h(\rho e^{it},0)|/h(\rho)$.

It remains to prove (iv).  Put $c_x=\cos(2\pi x)$ and fix $\eta_*=\pi/6$.
Write
\[
 h(\rho e^{it},x)=1+2\rho e^{it}c_x+R_x(t),
 \qquad h(\rho)=1+2\rho+R_0,
 \qquad |R_x(t)|+|R_0|\le8\rho^4.
\]
If $\operatorname{dist}(t,\pi\mathbb Z)\ge\eta_*$, then
$|\cos t|\le\cos\eta_*$.  Since $|c_x|\le1$,
\begin{align*}
 (1+2\rho)^2-|1+2\rho e^{it}c_x|^2
 &=4\rho(1-c_x\cos t)+4\rho^2(1-c_x^2)\\
 &\ge4\rho(1-\cos\eta_*).
\end{align*}
The two remainder terms change the two squared moduli by at most
$20\rho^4$.  For $0<\rho\le2^{-7}$,
$20\rho^4\le2\rho(1-\cos\eta_*)$.  Hence
\[
 h(\rho)^2-|h(\rho e^{it},x)|^2
 \ge2\rho(1-\cos\eta_*).
\]
Since $h(\rho)$ is uniformly bounded above, dividing by $h(\rho)^2$ and
using $\sqrt{1-y}\le1-y/2$ gives
\eqref{eq:small-full-remote-contraction} with an absolute constant $c_*>0$.
The statement near the second saddle follows from \eqref{eq:parity-theta}.
\end{proof}

Multiplying \eqref{eq:small-contraction} over the coordinates yields
\begin{equation}
 |P_\rho(t,\xi)|
 \le\exp\!\left(-\frac\rho{32}
 \sum_{j=1}^d\delta(\xi_j)\right).
 \label{eq:small-product-contract}
\end{equation}
From now on, in the small range we fix the central width
$\eta:=\eta_*=\pi/6$, which is at most $\eta_s$.  On the complement of the
two central arcs, \eqref{eq:small-full-remote-contraction} gives
\begin{equation}
 \frac{\prod_{j=1}^d|h(\rho e^{it},\xi_j)|}{h(\rho)^d}
 \le(1-c_*\rho)^d\le e^{-c_*\rho d}\le e^{-c n}.
 \label{eq:small-remote-product}
\end{equation}
The pole is uniformly bounded in this range, so every part of the contour
outside the two central arcs is $O(e^{-cn})$ in absolute value.

On the central arcs we use the scale $t=u/\sqrt n$.

\begin{lemma}[small central expansion]
\label{lem:small-central}
There exist $n_0,C,c>0$ such that for $n_0\le n\le\alpha_0d$ the central
arcs $E_0=\{|t|\le\eta\}$ and $E_\pi=\pi+E_0$, with the fixed width
$\eta=\pi/6$ chosen above, satisfy
\begin{align}
 \left|\frac{J_+(\xi;E_0)}{J_+(0;E_0)}
       -\Gamma_{d,\rho}(\xi)\right|&\le \frac Cn,
 \label{eq:small-central-plus}\\
 \left|\frac{J_-(\xi;E_0)}{J_-(0;E_0)}
       -\Gamma_{d,\rho}(\xi)\right|&\le \frac Cn.
 \label{eq:small-central-minus}
\end{align}
Moreover,
\begin{equation}
 |J_+(0;E_0)|\asymp n^{-1/2},
 \qquad
 |J_-(0;E_0)|\asymp n^{-1/2},
 \label{eq:small-base-scale}
\end{equation}
and
\begin{equation}
 |\Gamma_{d,\rho}(\boldsymbol\omega)|\le e^{-cn}.
 \label{eq:small-cross-saddle}
\end{equation}
\end{lemma}

\begin{proof}
Let $s=\log\rho$ and put $t=u/\sqrt n$.  Taylor's theorem at $s$, the
saddle equation $\mu(\rho)=n/d$, and the definition of $\lambda_j$ give
\begin{align}
 d\Phi_\rho(u/\sqrt n)
 ={}&-\frac{\lambda_2(\alpha)}2u^2
 -\frac{i\lambda_3(\alpha)}{6\sqrt n}u^3
 +\frac{\lambda_4(\alpha)}{24n}u^4
 +E_{5,n}(u), \label{eq:small-phase-correct}\\
 |E_{5,n}(u)|&\le C\frac{|u|^5}{n^{3/2}}
 \qquad (|u|\le \eta\sqrt n).
 \label{eq:small-fifth}
\end{align}
To justify the remainder uniformly on the complex segment, note that for
$0<\rho\le2^{-7}$ and $|t|\le\pi/6$,
\[
 h(\rho e^{it},0)=1+O(\rho),
 \qquad
 \partial_s^j h(e^{s+it},0)
 =\sum_{k\in\mathbb Z}k^{2j}\rho^{k^2}e^{ik^2t}
 =O_j(\rho),\quad 1\le j\le5.
\]
Hence $h(\rho e^{it},0)$ is uniformly separated from zero and
\[
 \left|\partial_s^5\log h(e^{s+it},0)\right|\le C\rho.
\]
Since $d\rho\simeq n$, Taylor's formula with integral remainder gives
\eqref{eq:small-fifth}.

For the product amplitude, pairing $k$ and $-k$ gives, for $m=0,1,2$,
\begin{align*}
 &\partial_t^m\{h(\rho e^{it},x)-h(\rho e^{it},0)\}\\
 &\qquad=2i^m\sum_{k\ge1}k^{2m}\rho^{k^2}e^{ik^2t}
       \{\cos(2\pi kx)-1\},
\end{align*}
and therefore, using $1-\cos(2\pi kx)\le k^2\delta(x)$,
\[
 \left|\partial_t^m\{h(\rho e^{it},x)-h(\rho e^{it},0)\}\right|
 \le C_m\rho\delta(x),\qquad m=0,1,2.
\]
The same small-$\rho$ bounds give, uniformly in $x$ and $|t|\le\pi/6$,
\[
 |h(\rho e^{it},x)|\ge1-C\rho\ge\frac12,
 \qquad
 |h(\rho e^{it},0)|\ge\frac12.
\]
Two applications of the quotient rule consequently yield
\[
 \left|\partial_t^m
 \log\frac{h(\rho e^{it},x)}{h(\rho e^{it},0)}\right|
 \le C_m\rho\delta(x),\qquad m=1,2.
\]
Summing over the coordinates in
\[
 \log P_\rho(t,\xi)=
 \sum_{j=1}^d
 \log\frac{h(\rho e^{it},\xi_j)}{h(\rho e^{it},0)}
\]
and using \eqref{eq:small-product-contract}, the chain rule gives, after
the change of variables $t=u/\sqrt n$,
\begin{align}
 |P_\rho(u/\sqrt n,\xi)|
 &\le e^{-cE(\xi)}, \\
 |\partial_uP_\rho(u/\sqrt n,\xi)|
 &\le \frac C{\sqrt n}(1+E(\xi))e^{-cE(\xi)}, \\
 |\partial_u^2P_\rho(u/\sqrt n,\xi)|
 &\le \frac Cn(1+E(\xi))^2e^{-cE(\xi)},
 \label{eq:small-amplitude-derivatives}
\end{align}
where $E(\xi)=\rho\sum_j\delta(\xi_j)$.  Indeed, the first logarithmic
derivative is $O(E/\sqrt n)$, the second is $O(E/n)$, and
$P''=P\{(\log P)'^2+(\log P)''\}$.

Set
\[
 \Psi_n(u)=d\Phi_\rho(u/\sqrt n),
 \qquad
 A_{+,n}(u)=\frac{1-\rho}{1-\rho e^{iu/\sqrt n}},
 \qquad
 A_{-,n}(u)=\frac{1+\rho}{1+\rho e^{iu/\sqrt n}}.
\]
Direct differentiation, using $\rho\le2^{-11}$, gives
\begin{equation}
 |A_{\pm,n}(u)|\le C,
 \qquad
 |\partial_u^rA_{\pm,n}(u)|\le C_r n^{-r/2},
 \qquad r=1,2.
 \label{eq:small-pole-derivatives}
\end{equation}
Moreover, \eqref{eq:small-phase-decay} and $d\rho\simeq n$ imply
\begin{equation}
 \Re\Psi_n(u)\le-cu^2,
 \qquad |u|\le\eta\sqrt n.
 \label{eq:small-scaled-phase-decay}
\end{equation}
Thus \eqref{eq:small-phase-correct},
\eqref{eq:small-amplitude-derivatives},
\eqref{eq:small-pole-derivatives}, and
\eqref{eq:small-scaled-phase-decay} verify the hypotheses of
\cref{lem:symmetric-saddle} with
\[
 N=n,\qquad q=\lambda_2(\alpha),\qquad
 b_3=-\frac{\lambda_3(\alpha)}6,\qquad
 b_4=\frac{\lambda_4(\alpha)}{24},\qquad
 G(\xi)=\Gamma_{d,\rho}(\xi).
\]
Since
\[
 J_\pm(\xi;E_0)
 =\frac{1}{2\pi\sqrt n}\frac1{1\mp\rho}
   \int_{-\eta\sqrt n}^{\eta\sqrt n}
   e^{\Psi_n(u)}A_{\pm,n}(u)P_\rho(u/\sqrt n,\xi)\,du,
\]
where the lower sign means $1+\rho$, the ratio conclusion
\eqref{eq:symmetric-ratio} gives
\eqref{eq:small-central-plus}--\eqref{eq:small-central-minus}.
The base conclusion \eqref{eq:symmetric-base} gives
\begin{align}
 J_+(0;E_0)
 &=\frac1{2\pi\sqrt n}
 \left\{
 \frac1{1-\rho}
 \int_{\mathbb R}e^{-\lambda_2u^2/2}\,du+O(n^{-1})
 \right\}, \label{eq:small-base-plus-expansion}\\
 J_-(0;E_0)
 &=\frac1{2\pi\sqrt n}
 \left\{
 \frac1{1+\rho}
 \int_{\mathbb R}e^{-\lambda_2u^2/2}\,du+O(n^{-1})
 \right\}. \label{eq:small-base-minus-expansion}
\end{align}
Since $\lambda_2\in[1/32,64]$ and $\rho\le2^{-11}$, the leading
terms in braces are bounded below by an absolute positive constant.  After
enlarging $n_0$, the $O(n^{-1})$ remainders are at most half of this lower
bound.  Hence, separately for the two signs,
\[
 c n^{-1/2}\le |J_\pm(0;E_0)|\le C n^{-1/2},
\]
which proves \eqref{eq:small-base-scale}.  Finally,
\[
 \left|\phi_\rho(1/2)\right|
 \le1-c\rho
\]
by the $k=0,\pm1$ separation, and $d\rho\simeq n$ proves
\eqref{eq:small-cross-saddle}.
\end{proof}

\begin{proposition}[small two-saddle residual]
\label{prop:small-residual}
There exist $n_0,c,C>0$ such that for $n_0\le n\le\alpha_0d$,
\begin{equation}
 \sup_{\xi\in\mathbb T^d}|r_{d,n}(\xi)|
 \le \frac Cn+Ce^{-cn}.
 \label{eq:small-residual}
\end{equation}
\end{proposition}

\begin{proof}
We decompose the full $J_+$ contour into $E_0$, $E_\pi$, and the remaining
part.  By \eqref{eq:small-remote-product}, the remaining part is $O(e^{-cn})$
in absolute value.  Since \eqref{eq:small-base-scale} gives
$|J_+(0;E_0)|\asymp n^{-1/2}$,
\[
 \frac{e^{-cn}}{|J_+(0;E_0)|}
 \le C\sqrt n\,e^{-cn}
 \le C e^{-c'n}
\]
for a possibly smaller absolute $c'>0$.  Thus the remaining contour is
$O(e^{-c'n})$ relative to $J_+(0;E_0)$ in modulus.  By \eqref{eq:negative-arc-transport}, the
negative contribution at frequency zero is $(-1)^nJ_-(\boldsymbol\omega;E_0)$.
From \eqref{eq:small-central-minus}, \eqref{eq:small-base-scale}, and
\eqref{eq:small-cross-saddle},
\[
 |J_-(\boldsymbol\omega;E_0)|
 \le C(n^{-1}+e^{-cn})|J_+(0;E_0)|.
\]
Consequently the full denominator $D=J_+(0;[-\pi,\pi])$ satisfies
\[
 D=J_+(0;E_0)\{1+O(n^{-1}+e^{-cn})\},
\]
and therefore
$c_+:=J_+(0;E_0)/D=1+O(n^{-1}+e^{-cn})$ and $|c_+|\le2$.

Write $c_-=(-1)^nJ_-(0;E_0)/D$.  The base expansions
\eqref{eq:small-base-plus-expansion}--\eqref{eq:small-base-minus-expansion}
show $|c_-|\le2$ after increasing $n_0$.
Set
\[
 F_+(\xi)=\frac{J_+(\xi;E_0)}{J_+(0;E_0)},
 \qquad
 F_-(\xi)=\frac{J_-(\xi;E_0)}{J_-(0;E_0)}.
\]
The exact contour decomposition and
\eqref{eq:negative-arc-transport} give
\[
 m_{d,n}(\xi)=c_+F_+(\xi)+c_-F_-(\xi-\boldsymbol\omega)
 +E_{\rm rem}(\xi),
 \qquad
 \|E_{\rm rem}\|_{L^\infty}\le C(n^{-1}+e^{-cn}).
\]
Finally evaluate the exact contour decomposition at $\xi=\boldsymbol\omega$.  By
\eqref{eq:small-central-plus} and \eqref{eq:small-cross-saddle},
\[
 \frac{J_+(\boldsymbol\omega;E_0)}{J_+(0;E_0)}
 =O(n^{-1}+e^{-cn}).
\]
Since $F_-(0)=1$, $|c_+|\le2$, $m_{d,n}(\boldsymbol\omega)=a_{d,n}$, and the remote
contribution is $O(e^{-cn})$, it follows that
$|c_--a_{d,n}|\le C(n^{-1}+e^{-cn})$.

Using \eqref{eq:small-central-plus}--\eqref{eq:small-central-minus},
$|c_+-1|\lesssim n^{-1}+e^{-cn}$, and
$|c_--a_{d,n}|\lesssim n^{-1}+e^{-cn}$, the triangle inequality gives
\eqref{eq:small-residual}.
\end{proof}

\begin{corollary}[small maximal range]
\label{thm:small}
There is $C_S<\infty$, independent of $d$, such that
\[
 \left\|\sup_{0\le n\le\alpha_0d}|M_{d,n}f|\right\|_2
 \le C_S\|f\|_2.
\]
\end{corollary}

\begin{proof}
The centered branch is controlled by \eqref{eq:known-gaussian}, and the translated
branch by \eqref{eq:translated-gaussian-maximal}; recall that $|a_{d,n}|\le1$.
For the residual, Plancherel and \eqref{eq:small-residual} give
\[
 \left\|\sup_{n_0\le n\le\alpha_0d}|R_{d,n}f|\right\|_2
 \le\left(\sum_{n\ge n_0}
  (C/n+Ce^{-cn})^2\right)^{1/2}\|f\|_2.
\]
The finitely many squared-radius levels with $n<n_0$ are treated by the trivial
contraction bound.
\end{proof}

\subsection{A fixed critical window}

Fix $0<a\le b<\infty$ and write
\[
 K_{a,b}=[\rho(a),\rho(b)]\Subset(0,1).
\]
Since $v(\rho)=\rho\mu'(\rho)>0$ and all cumulants are continuous in
$\rho$, compactness gives
\begin{equation}
 0<c_{a,b}\le v(\rho)\le C_{a,b}<\infty,
 \label{eq:critical-variance}
\end{equation}
and the cumulants through order five are uniformly bounded on $K_{a,b}$.

\begin{lemma}[critical strip estimates]
\label{lem:critical-strip}
There exist $\eta,c,C>0$ such that for $\rho\in K_{a,b}$,
$|t|\le\eta$, and $x\in\mathbb T$, the theta factor
$h(\rho e^{it},x)$ is nonzero and
\begin{align}
 \left|\frac{h(\rho e^{it},x)}{h(\rho e^{it},0)}\right|
 &\le e^{-c\delta(x)}, \label{eq:critical-contraction}\\
 \left|\partial_s^r
 \log\frac{h(e^{s+it},x)}{h(e^{s+it},0)}\right|
 &\le C\delta(x),\qquad r=1,2. \label{eq:critical-logder}
\end{align}
Writing $L(s+it)=\log h(e^{s+it},0)$ with the analytic branch fixed by the
real logarithm at $t=0$, one also has
\begin{equation}
 \sup_{\rho\in K_{a,b},\ |t|\le\eta}
 \left|\partial_s^5L(\log\rho+it)\right|\le C_{a,b}.
 \label{eq:critical-complex-fifth}
\end{equation}
After decreasing $\eta$ if necessary, one also has
\begin{equation}
 \Re\Phi_\rho(t)\le-c_{a,b}t^2,
 \qquad |t|\le\eta.
 \label{eq:critical-phase-decay}
\end{equation}
On the complement of the two arcs $|t|\le\eta$ and
$|t-\pi|\le\eta$,
\begin{equation}
 \left|\frac{\prod_jh(\rho e^{it},\xi_j)}{h(\rho)^d}\right|
 \le q^d
 \label{eq:critical-remote}
\end{equation}
for some $q=q(a,b,\eta)<1$.
\end{lemma}

\begin{proof}
Poisson summation on the real axis gives $h(\rho,x)>0$.  Compactness of
$K_{a,b}\times\mathbb T$ gives a uniform positive lower bound and, after
shrinking $\eta$, a zero-free complex strip.  Moreover
\[
 h(\rho,0)-h(\rho,x)
 =2\sum_{k\ge1}\rho^{k^2}(1-\cos2\pi kx)
 \ge2\rho\delta(x).
\]
Since $\rho\ge\rho(a)>0$ and $h(\rho,0)$ is uniformly bounded above,
there is $c_0=c_0(a,b)>0$ such that
\begin{equation}
 \log\frac{h(\rho,x)}{h(\rho,0)}
 \le\log(1-c_0\delta(x))\le-c_0\delta(x).
 \label{eq:critical-real-log-gap}
\end{equation}
Also, for $m=0,1,2$ and throughout the zero-free strip,
\begin{align*}
 &\left|\partial_s^m
 \{h(e^{s+it},x)-h(e^{s+it},0)\}\right|\\
 &\qquad\le
 2\sum_{k\ge1}k^{2m}e^{sk^2}(1-\cos2\pi kx)
 \le C_{m,a,b}\delta(x),
\end{align*}
where we used $1-\cos(2\pi kx)\le k^2\delta(x)$.  The theta factors and
their first two radial derivatives are uniformly bounded, while the theta
factors themselves are uniformly bounded away from zero.  Quotient
differentiation therefore gives \eqref{eq:critical-logder}.  The same
absolute-convergence argument gives, for $0\le j\le5$,
\[
 \sup_{\rho\in K_{a,b},\ |t|\le\eta}
 \left|\partial_s^j h(e^{s+it},0)\right|\le C_{j,a,b}.
\]
Since $h(e^{s+it},0)$ is uniformly separated from zero on the strip, five
applications of the quotient rule give \eqref{eq:critical-complex-fifth}.

For
\[
 Q(s+it,x)=\frac{h(e^{s+it},x)}{h(e^{s+it},0)},
\]
choose the analytic logarithm agreeing with the real logarithm at $t=0$.
Analyticity gives
$\partial_t\log Q=i\partial_s\log Q$.  Hence, by
\eqref{eq:critical-real-log-gap} and the case $r=1$ of
\eqref{eq:critical-logder},
\begin{align*}
 \log|Q(s+it,x)|
 &=\log Q(s,x)
   +\Re\int_0^t\partial_v\log Q(s+iv,x)\,dv\\
 &\le-c_0\delta(x)+C|t|\delta(x).
\end{align*}
After decreasing $\eta$ so that $C\eta\le c_0/2$, this proves
\eqref{eq:critical-contraction} with $c=c_0/2$.  Since
$\Phi_\rho(0)=\Phi_\rho'(0)=0$ and
$\Phi_\rho''(0)=-v(\rho)$, while the third and fourth derivatives are
uniformly bounded on the compact strip, Taylor's theorem and
\eqref{eq:critical-variance} give
$\Re\Phi_\rho(t)\le-c_{a,b}t^2$ after shrinking $\eta$.

For the remote estimate, equality in
$|h(\rho e^{it},x)|\le h(\rho)$ would force all terms in the theta
series to have the same phase as the $k=0$ term.  In particular the
$k=1$ and $k=-1$ terms imply
\[
 t+2\pi x\in2\pi\mathbb Z,
 \qquad
 t-2\pi x\in2\pi\mathbb Z,
\]
and hence $t\in\pi\mathbb Z$.  Thus the inequality is strict whenever
$\operatorname{dist}(t,\pi\mathbb Z)\ge\eta$.  Compactness on the set
$K_{a,b}\times\{t\in\mathbb R/(2\pi\mathbb Z):
\operatorname{dist}(t,\pi\mathbb Z)
\ge\eta\}\times\mathbb T$ then gives a uniform $q<1$, and multiplication
over the coordinates yields \eqref{eq:critical-remote}.
\end{proof}

Fix a width $\eta=\eta_{a,b}>0$ for which all conclusions of
\cref{lem:critical-strip} hold, and for the remainder of this subsection set
\[
 E_0=\{t\in[-\pi,\pi]:|t|\le\eta\},
 \qquad
 E_\pi=\pi+E_0\pmod{2\pi}.
\]

\begin{lemma}[critical central expansion]
\label{lem:critical-central}
There are $d_0,C>0$ such that for $d\ge d_0$ and $ad\le n\le bd$,
\begin{align}
 \left|\frac{J_+(\xi;E_0)}{J_+(0;E_0)}
 -\Gamma_{d,\rho}(\xi)\right|&\le C_{a,b}d^{-1},
 \label{eq:critical-central-plus}\\
 \left|\frac{J_-(\xi;E_0)}{J_-(0;E_0)}
 -\Gamma_{d,\rho}(\xi)\right|&\le C_{a,b}d^{-1}.
 \label{eq:critical-central-minus}
\end{align}
Furthermore
\begin{equation}
 |J_+(0;E_0)|\asymp_{a,b}d^{-1/2},
 \qquad
 |J_-(0;E_0)|\asymp_{a,b}d^{-1/2},
 \qquad
 |\Gamma_{d,\rho}(\boldsymbol\omega)|\le q_{a,b}^d
 \label{eq:critical-base-cross}
\end{equation}
with $q_{a,b}<1$.
\end{lemma}

\begin{proof}
Set $t=u/\sqrt d$.  Taylor's theorem with the complex fifth-derivative
bound \eqref{eq:critical-complex-fifth} gives
\begin{equation}
 d\Phi_\rho(u/\sqrt d)
 =-\frac{v(\rho)}2u^2
 -\frac{i\kappa_3(\rho)}{6\sqrt d}u^3
 +\frac{\kappa_4(\rho)}{24d}u^4
 +O_{a,b}\!\left(\frac{|u|^5}{d^{3/2}}\right).
 \label{eq:critical-phase}
\end{equation}
For $E(\xi)=\sum_j\delta(\xi_j)$,
\eqref{eq:critical-contraction}--\eqref{eq:critical-logder} and the chain
rule give
\begin{align*}
 |P_\rho(u/\sqrt d,\xi)|&\le e^{-cE(\xi)},\\
 |\partial_uP_\rho(u/\sqrt d,\xi)|
 &\le C_{a,b}d^{-1/2}(1+E(\xi))e^{-cE(\xi)},\\
 |\partial_u^2P_\rho(u/\sqrt d,\xi)|
 &\le C_{a,b}d^{-1}(1+E(\xi))^2e^{-cE(\xi)}.
\end{align*}
Set
\[
 \Psi_d(u)=d\Phi_\rho(u/\sqrt d),
 \qquad
 A_{+,d}(u)=\frac{1-\rho}{1-\rho e^{iu/\sqrt d}},
 \qquad
 A_{-,d}(u)=\frac{1+\rho}{1+\rho e^{iu/\sqrt d}}.
\]
Since $\rho\in K_{a,b}\Subset(0,1)$, direct differentiation gives
\begin{equation}
 |A_{\pm,d}(u)|\le C_{a,b},
 \qquad
 |\partial_u^rA_{\pm,d}(u)|\le C_{r,a,b}d^{-r/2},
 \qquad r=1,2.
 \label{eq:critical-pole-derivatives}
\end{equation}
Also \eqref{eq:critical-phase-decay} implies
\begin{equation}
 \Re\Psi_d(u)\le-c_{a,b}u^2,
 \qquad |u|\le\eta\sqrt d.
 \label{eq:critical-scaled-phase-decay}
\end{equation}
Thus \eqref{eq:critical-phase}, the three product bounds above,
\eqref{eq:critical-pole-derivatives}, and
\eqref{eq:critical-scaled-phase-decay} verify
\cref{lem:symmetric-saddle} with
\[
 N=d,\qquad q=v(\rho),\qquad
 b_3=-\frac{\kappa_3(\rho)}6,\qquad
 b_4=\frac{\kappa_4(\rho)}{24},\qquad
 G(\xi)=\Gamma_{d,\rho}(\xi).
\]
As in the small range,
\[
 J_\pm(\xi;E_0)
 =\frac{1}{2\pi\sqrt d}\frac1{1\mp\rho}
   \int_{-\eta\sqrt d}^{\eta\sqrt d}
   e^{\Psi_d(u)}A_{\pm,d}(u)P_\rho(u/\sqrt d,\xi)\,du.
\]
The ratio conclusion \eqref{eq:symmetric-ratio} proves
\eqref{eq:critical-central-plus}--\eqref{eq:critical-central-minus}, while
\eqref{eq:symmetric-base} gives
\begin{align}
 J_+(0;E_0)
 &=\frac1{2\pi\sqrt d}
 \left\{
 \frac1{1-\rho}\sqrt{\frac{2\pi}{v(\rho)}}+O_{a,b}(d^{-1})
 \right\}, \label{eq:critical-base-plus}\\
 J_-(0;E_0)
 &=\frac1{2\pi\sqrt d}
 \left\{
 \frac1{1+\rho}\sqrt{\frac{2\pi}{v(\rho)}}+O_{a,b}(d^{-1})
 \right\}. \label{eq:critical-base-minus}
\end{align}
The leading quantities in braces are uniformly bounded below by a positive
constant on the compact interval $K_{a,b}$.  After increasing the lower
dimension threshold, the $O_{a,b}(d^{-1})$ remainders are at most half of
this bound, and hence the asserted modulus lower bounds follow.  Finally,
$|h(\rho,1/2)|<h(\rho)$ uniformly on $K_{a,b}$, which gives the last part of
\eqref{eq:critical-base-cross}.
\end{proof}

\begin{proposition}[critical two-saddle residual]
\label{prop:critical-residual}
For every $0<a\le b<\infty$ there exists $C_{a,b}$ such that, for all
$d\ge1$ and $ad\le n\le bd$,
\begin{equation}
 \sup_{\xi\in\mathbb T^d}|r_{d,n}(\xi)|\le C_{a,b}d^{-1}.
 \label{eq:critical-residual}
\end{equation}
\end{proposition}

\begin{proof}
For $d\ge d_0$, we split the contour into the two central arcs and the
remote part.  The remote contribution is $O(q^d)$ by
\eqref{eq:critical-remote}.  At frequency zero the negative central arc is
$(-1)^nJ_-(\boldsymbol\omega;E_0)$.  By
\eqref{eq:critical-central-minus} and \eqref{eq:critical-base-cross}, it is
$O(d^{-1}+q_1^d)$ relative to $J_+(0;E_0)$ in modulus, for some $q_1<1$.
Therefore, with $D=J_+(0;[-\pi,\pi])$,
\[
 D=J_+(0;E_0)\{1+O(d^{-1}+q_2^d)\},\qquad q_2<1,
\]
and hence $c_+=J_+(0;E_0)/D=1+O(d^{-1}+q_2^d)$.  The explicit base formulas
\eqref{eq:critical-base-plus}--\eqref{eq:critical-base-minus} also give
$|c_+|+|c_-|\le4$ for large $d$, where
$c_-=(-1)^nJ_-(0;E_0)/D$.

Set
\[
 F_+(\xi)=\frac{J_+(\xi;E_0)}{J_+(0;E_0)},
 \qquad
 F_-(\xi)=\frac{J_-(\xi;E_0)}{J_-(0;E_0)}.
\]
Since $|D|^{-1}\asymp d^{1/2}$, we may replace $q<1$ by a slightly larger
number, still smaller than one, and absorb this polynomial factor into the
exponential decay.  The exact contour decomposition therefore gives
\[
 m_{d,n}(\xi)=c_+F_+(\xi)+c_-F_-(\xi-\boldsymbol\omega)+E_{\rm rem}(\xi),
 \qquad \|E_{\rm rem}\|_{L^\infty}\le C_{a,b}q^d.
\]
Evaluating this identity at $\xi=\boldsymbol\omega$ gives
\[
 |c_--a_{d,n}|\le C_{a,b}(d^{-1}+q_2^d).
\]
Together with \eqref{eq:critical-central-plus}--
\eqref{eq:critical-central-minus} and the preceding coefficient bounds,
this gives \eqref{eq:critical-residual} by the triangle inequality.  The
finitely many dimensions $d<d_0$ are absorbed by increasing $C_{a,b}$.
\end{proof}

\begin{corollary}[critical maximal range]
\label{thm:critical}
For every $0<a\le b<\infty$ there is $C'_{a,b}<\infty$ such that
\[
 \left\|\sup_{ad\le n\le bd}|M_{d,n}f|\right\|_2
 \le C'_{a,b}\|f\|_2.
\]
\end{corollary}

\begin{proof}
There are at most $(b-a)d+2$ squared-radius levels.  Hence, by Plancherel and
\eqref{eq:critical-residual},
\[
 \left\|\sup_{ad\le n\le bd}|R_{d,n}f|\right\|_2
 \le \frac{C_{a,b}}d\sqrt{(b-a)d+2}\,\|f\|_2.
\]
Adding the centered and translated Gaussian branches, using
\eqref{eq:known-gaussian} and \eqref{eq:translated-gaussian-maximal}, finishes the proof.
\end{proof}

\subsection{Growing \texorpdfstring{$\alpha$}{alpha}: the central arcs}

Fix $A>0$.  Write
\[
 \rho=e^{-\tau},\qquad \alpha=\frac nd.
\]
For $w\in\mathbb C$ with $\Re w>0$, Poisson summation gives
\begin{equation}
 h(e^{-w},x)=\sqrt{\frac\pi w}
 \sum_{m\in\mathbb Z}e^{-\pi^2(m+x)^2/w}.
 \label{eq:growing-poisson}
\end{equation}
Here and below the square root is the analytic branch on $\Re w>0$ that
is positive for positive real $w$.

\Needspace{5\baselineskip}
\begin{lemma}[growing saddle scale]
\label{lem:growing-scale}
There exists $\alpha_*>0$ such that, for $\alpha\ge\alpha_*$,
\begin{equation}
 \frac1{4\alpha}\le\tau(\alpha)\le\frac1{2\alpha}.
 \label{eq:growing-tau}
\end{equation}
If $L(-\tau)=\log h(e^{-\tau})$, then, for $r=2,3,4,5$,
\begin{equation}
 \tau^rL^{(r)}(-\tau)=\frac{(r-1)!}{2}+O(e^{-c/\tau}).
 \label{eq:growing-cumulants}
\end{equation}
In particular, after increasing $\alpha_*$ if necessary,
\begin{equation}
 q_g(\tau):=\tau^2L''(-\tau)\in\left[\frac38,\frac58\right].
 \label{eq:growing-exact-quadratic}
\end{equation}
Moreover, after increasing $\alpha_*$ if necessary,
\begin{equation}
 \sup_{|t|\le\tau/2}|L^{(5)}(-\tau+it)|\le C\tau^{-5}.
 \label{eq:growing-complex-fifth}
\end{equation}
\end{lemma}

\begin{proof}
At $x=0$, \eqref{eq:growing-poisson} gives
\[
 h(e^{-\tau})=\sqrt{\frac\pi\tau}\,(1+S(\tau)),
 \qquad
 S(\tau)=2\sum_{m\ge1}e^{-\pi^2m^2/\tau}.
\]
Since $S'(\tau)>0$ and $\alpha=\mu(e^{-\tau})
=-\frac{d}{d\tau}\log h(e^{-\tau})$,
\begin{equation}
 \alpha=\frac1{2\tau}-\frac{S'(\tau)}{1+S(\tau)}
 <\frac1{2\tau}.
 \label{eq:growing-tau-upper-sign}
\end{equation}
Thus $\tau<(2\alpha)^{-1}$.  On the other hand, termwise
differentiation gives
\[
 0\le\tau S'(\tau)
 \le C\tau^{-1}e^{-\pi^2/\tau}.
\]
Since $\tau(\alpha)\to0$ as $\alpha\to\infty$, after increasing
$\alpha_*$ the last display is at most $1/4$.  Hence
\[
 \alpha\tau
 =\frac12-\frac{\tau S'(\tau)}{1+S(\tau)}
 \ge\frac14,
\]
which proves the lower bound in \eqref{eq:growing-tau}.

Taking the logarithm of the Poisson formula gives
\begin{equation}
 L(-\tau)=\frac12\log\pi-\frac12\log\tau+\log(1+S(\tau)).
 \label{eq:growing-real-log-poisson}
\end{equation}
Since the radial variable is $s=-\tau$, one has $\partial_s=-\partial_\tau$,
and therefore, for $r\ge1$,
\[
 (-\partial_\tau)^r\!\left(-\frac12\log\tau\right)
 =\frac{(r-1)!}{2\tau^r}.
\]
For each fixed $r\le5$, termwise differentiation of $S$ gives
$|\partial_\tau^rS(\tau)|\le C_r\tau^{-M_r}e^{-\pi^2/\tau}$ for some
integer $M_r$.  The finite differentiation formula for $\log(1+S)$,
together with the fact that the exponential absorbs every fixed power of
$\tau^{-1}$, yields
\[
 \tau^r\big|\partial_\tau^r\log(1+S(\tau))\big|
 \le C_re^{-c/\tau}.
\]
Applying $(-\partial_\tau)^r$ to \eqref{eq:growing-real-log-poisson}
proves \eqref{eq:growing-cumulants}.  In particular,
\[
 \tau^2L''(-\tau)=\frac12+O(e^{-c/\tau}),
\]
and increasing $\alpha_*$ so that the error is at most $1/8$ gives
\eqref{eq:growing-exact-quadratic}.

It remains to record the complex estimate needed below.  Put
$w=\tau-it$ and, for $\Re w>0$,
\[
 S(w)=2\sum_{m\ge1}e^{-\pi^2m^2/w}.
\]
If $|t|\le\tau/2$, then
\[
 \Re\frac1w=\frac{\tau}{\tau^2+t^2}\ge\frac4{5\tau}.
\]
For $0\le j\le5$, direct differentiation has the form
\[
 \partial_w^j e^{-a/w}
 =w^{-j}P_j(a/w)e^{-a/w},
\]
with a fixed polynomial $P_j$.  Therefore
\begin{equation}
 |\partial_w^jS(w)|\le C_j\tau^{-j}e^{-c/\tau},
 \qquad 0\le j\le5.
 \label{eq:growing-S-complex-derivatives}
\end{equation}
For $\alpha_*$ large, \eqref{eq:growing-S-complex-derivatives} with
$j=0$ gives $|S(w)|\le1/2$, so
\[
 L(-w)=\frac12\log\pi-\frac12\log w+\log(1+S(w))
\]
with one analytic branch on this sector.  Since
$\partial_{(-w)}=-\partial_w$, five differentiations and
\eqref{eq:growing-S-complex-derivatives} yield
\[
 |L^{(5)}(-w)|\le C|w|^{-5}+C\tau^{-5}e^{-c/\tau}
 \le C\tau^{-5},
\]
which is \eqref{eq:growing-complex-fifth}.
\end{proof}

For $x$ reduced to $[-1/2,1/2]$, one dual Gaussian dominates near $x=0$,
while two must be retained near $x=\pm1/2$.

\Needspace{5\baselineskip}
\begin{lemma}[central theta quotient and logarithmic derivatives]
\label{lem:two-image}
Let
\[
 r(x)=\operatorname{dist}(x,\mathbb Z).
\]
There are $\tau_0,c,C>0$ such that, whenever
$0<\tau\le\tau_0$, $|\sigma|\le\tau/2$, $w=\tau+i\sigma$, and
$x\in\mathbb R$, both $h(e^{-w},x)$ and $h(e^{-w},0)$ are nonzero.  With
$R(w,x)=h(e^{-w},x)/h(e^{-w},0)$, one has
\begin{align}
 |R(w,x)|&\le \exp\!\left(-c\frac{r(x)^2}{\tau}\right),
 \label{eq:growing-one-coordinate-contract}\\
 |\partial_w\log R(w,x)|&\le C\frac{r(x)^2}{\tau^2},
 \label{eq:growing-one-coordinate-log1}\\
 |\partial_w^2\log R(w,x)|&\le C\frac{r(x)^2}{\tau^3}.
 \label{eq:growing-one-coordinate-log2}
\end{align}
\end{lemma}

\begin{proof}
By periodicity and evenness it is enough to take $r=r(x)\in[0,1/2]$.
Poisson summation gives
\begin{equation}
 R(w,r)=
 \frac{\sum_{m\in\mathbb Z}e^{-\pi^2(m+r)^2/w}}
      {\sum_{m\in\mathbb Z}e^{-\pi^2m^2/w}},
 \qquad
 \Re\frac1w=\frac{\tau}{\tau^2+\sigma^2}\ge\frac4{5\tau}.
 \label{eq:growing-ratio-poisson}
\end{equation}
The denominator in \eqref{eq:growing-ratio-poisson} is
$1+O(e^{-c/\tau})$ after the common Poisson factor has been removed, and
is therefore bounded away from zero.

Suppose first that $0\le r\le1/4$.  After factoring out the image $m=0$
and pairing $m$ with $-m$, write
\begin{equation}
 R(w,r)=e^{-\pi^2r^2/w}\frac{1+E(w,r)}{1+E(w,0)},
 \qquad
 E(w,r)=2\sum_{m\ge1}e^{-\pi^2m^2/w}
              \cosh\!\left(\frac{2\pi^2mr}{w}\right).
 \label{eq:growing-nearest-image-explicit}
\end{equation}
For $0\le r\le1/4$ and $|\sigma|\le\tau/2$, the real part of the
exponent in every term with $m\ge1$ is at most $-c m^2/\tau$.  Thus, for
$j=0,1,2$,
\begin{equation}
 \left|\partial_r^2\partial_w^j E(w,r)\right|
 \le C_j\tau^{-2j-2}
       \sum_{m\ge1}m^{4j+4}e^{-c m^2/\tau}
 \le C_j\tau^{-2j-2}e^{-c/\tau}.
 \label{eq:growing-nearest-image-r2}
\end{equation}
The function $E(w,r)$ is even in $r$, hence
$\partial_rE(w,0)=0$.  Integrating \eqref{eq:growing-nearest-image-r2}
twice gives
\begin{equation}
 \left|\partial_w^j\{E(w,r)-E(w,0)\}\right|
 \le C_j r^2\tau^{-2j-2}e^{-c/\tau},
 \qquad j=0,1,2.
 \label{eq:growing-nearest-image-error}
\end{equation}
After decreasing $\tau_0$, $|E(w,r)|+|E(w,0)|\le1/4$.  Hence
\begin{align*}
 \log R(w,r)
 &=-\frac{\pi^2r^2}{w}
   +\log(1+E(w,r))-\log(1+E(w,0)),\\
 \left|\log\frac{1+E(w,r)}{1+E(w,0)}\right|
 &\le C r^2\tau^{-2}e^{-c/\tau},
\end{align*}
and the same use of the quotient rule together with
\eqref{eq:growing-nearest-image-error} gives
\begin{align*}
 \left|\partial_w\log R(w,r)-\frac{\pi^2r^2}{w^2}\right|
 &\le C r^2\tau^{-4}e^{-c/\tau},\\
 \left|\partial_w^2\log R(w,r)+\frac{2\pi^2r^2}{w^3}\right|
 &\le C r^2\tau^{-6}e^{-c/\tau}.
\end{align*}
For every fixed integer $m\ge0$ and $c>0$,
\begin{equation}
 \tau^{-m}e^{-c/\tau}
 \le C_{m,c}e^{-c/(2\tau)},
 \qquad 0<\tau\le1,
 \label{eq:exp-absorbs-powers}
\end{equation}
so the exponentially small factors absorb the excess powers of
$\tau^{-1}$.  Together with $\Re(1/w)\ge4/(5\tau)$, these inequalities prove
\eqref{eq:growing-one-coordinate-contract}--
\eqref{eq:growing-one-coordinate-log2} for $0\le r\le1/4$.

Now let $1/4\le r\le1/2$.  The two nearest images are $m=0,-1$.  Put
\[
 b(r)=\pi^2(1-2r),\qquad z(w,r)=e^{-b(r)/w}.
\]
After factoring out $e^{-\pi^2r^2/w}$ the numerator is
\begin{equation}
 e^{-\pi^2r^2/w}\{1+z(w,r)+\mathcal T(w,r)\},
 \qquad
 \mathcal T(w,r)=
 \sum_{m\notin\{0,-1\}}
 e^{-\pi^2((m+r)^2-r^2)/w},
 \label{eq:growing-two-image-explicit}
\end{equation}
and the denominator is $1+\mathcal T_0(w)$ with
\[
 \mathcal T_0(w)=\sum_{m\ne0}e^{-\pi^2m^2/w}.
\]
The two retained images cannot nearly cancel.  Indeed, if
$a=b(r)\Re(1/w)\ge\log2$, then $|z|\le1/2$.  If $a<\log2$, then
\[
 |\arg z|=b(r)|\Im(1/w)|
 \le\frac{|\sigma|}{\tau}a
 \le\frac12\log2<\frac\pi2,
\]
so $\Re z>0$.  Hence
\begin{equation}
 |1+z(w,r)|\ge\frac12.
 \label{eq:growing-two-image-noncancel}
\end{equation}

The remaining images are separated from the retained pair by a fixed
Gaussian gap.  More precisely, for $1/4\le r\le1/2$,
\[
 (1+r)^2-r^2=1+2r\ge\frac32,
 \qquad
 (2-r)^2-r^2=4-4r\ge2,
\]
and all farther images have a larger separation.  Differentiating the
series in \eqref{eq:growing-two-image-explicit} therefore gives, for
$j=0,1,2$,
\begin{equation}
 |\partial_w^j\mathcal T(w,r)|
 +|\partial_w^j\mathcal T_0(w)|
 \le C_j\tau^{-2j}e^{-c/\tau}.
 \label{eq:growing-two-image-relative-tail}
\end{equation}
In view of \eqref{eq:growing-two-image-noncancel}, the numerator bracket
in \eqref{eq:growing-two-image-explicit} is bounded away from zero after
$\tau_0$ is decreased.

It remains to estimate the retained pair.  Since
\[
 \partial_w\log(1+z)
 =\frac{b(r)z}{w^2(1+z)},
\]
\[
 \partial_w^2\log(1+z)
 =-\frac{2b(r)z}{w^3(1+z)}
   +\frac{b(r)^2z}{w^4(1+z)^2},
\]
we split according to $b(r)\ge\tau$ or $b(r)<\tau$.  In the first case
$|z|\le e^{-c b(r)/\tau}$, while in the second $b(r)\le\tau$.
Using \eqref{eq:growing-two-image-noncancel} in both cases yields
\begin{equation}
 |\partial_w\log(1+z(w,r))|\le C\tau^{-2},
 \qquad
 |\partial_w^2\log(1+z(w,r))|\le C\tau^{-3}.
 \label{eq:growing-two-image-logder}
\end{equation}
The perturbation estimate \eqref{eq:growing-two-image-relative-tail} and
one more application of the quotient rule give the same bounds after
$1+z$ is replaced by $1+z+\mathcal T$, and also for the denominator
$1+\mathcal T_0$.

Consequently
\[
 |R(w,r)|\le C e^{-c r^2/\tau},
 \qquad
 |\partial_w\log R(w,r)|\le C\tau^{-2},
 \qquad
 |\partial_w^2\log R(w,r)|\le C\tau^{-3}.
\]
Since $r^2\ge1/16$ in this region, the prefactor in the first inequality
is absorbed by decreasing the exponent, and the derivative bounds imply
\eqref{eq:growing-one-coordinate-log1}--
\eqref{eq:growing-one-coordinate-log2} after changing the constants.
The displayed lower bounds give the required nonvanishing, so a single
analytic branch of the logarithm may be used throughout the sector.
\end{proof}

For $\xi\in\mathbb T^d$ set
\begin{equation}
 \mathcal E_\tau(\xi)=\frac1\tau\sum_{j=1}^d
 \operatorname{dist}(\xi_j,\mathbb Z)^2.
 \label{eq:growing-central-energy}
\end{equation}
The preceding one-coordinate estimates give the product bounds needed for
the saddle expansion.

\begin{lemma}[growing product derivatives]
\label{lem:growing-product-derivatives}
After decreasing $\tau_0$, for $|u|\le\sqrt d/2$ and
$t=\tau u/\sqrt d$,
\begin{align}
 |P_\rho(t,\xi)|
 &\le e^{-c\mathcal E_\tau(\xi)},
 \label{eq:growing-product-contract-central}\\
 |\partial_uP_\rho(t,\xi)|
 &\le \frac C{\sqrt d}\,(1+\mathcal E_\tau(\xi))
 e^{-c\mathcal E_\tau(\xi)},
 \label{eq:growing-product-first}\\
 |\partial_u^2P_\rho(t,\xi)|
 &\le \frac Cd\,(1+\mathcal E_\tau(\xi))^2
 e^{-c\mathcal E_\tau(\xi)}.
 \label{eq:growing-product-second}
\end{align}
\end{lemma}

\begin{proof}
Put $w=\tau-it$.  By \eqref{eq:growing-one-coordinate-contract} and
multiplication over the coordinates,
\[
 |P_\rho(t,\xi)|\le
 \exp\!\left(-\frac c\tau\sum_jr(\xi_j)^2\right)
 =e^{-c\mathcal E_\tau(\xi)}.
\]
The chain rule and \eqref{eq:growing-one-coordinate-log1} give
\[
 |\partial_u\log P_\rho|
 \le\frac{\tau}{\sqrt d}
 \sum_j C\frac{r(\xi_j)^2}{\tau^2}
 =\frac C{\sqrt d}\mathcal E_\tau(\xi),
\]
while \eqref{eq:growing-one-coordinate-log2} gives
\[
 |\partial_u^2\log P_\rho|
 \le\frac Cd\mathcal E_\tau(\xi).
\]
Since
$P''=P\{(\log P)'^2+(\log P)''\}$, the last two estimates and
$\sup_{y\ge0}(1+y)^2e^{-cy}<\infty$ give
\eqref{eq:growing-product-first}--\eqref{eq:growing-product-second}.
\end{proof}

\begin{lemma}[growing central phase decay]
\label{lem:growing-phase-decay}
There are $\alpha_1$ large and $0<\eta_0\le1/2$ such that, whenever
$\alpha=n/d\ge\alpha_1$ and $|t|\le\eta_0\tau(\alpha)$,
\begin{equation}
 \Re\Phi_\rho(t)\le-c\frac{t^2}{\tau^2}.
 \label{eq:growing-phase-decay}
\end{equation}
Consequently, for $t=\tau u/\sqrt d$,
\begin{equation}
 |e^{d\Phi_\rho(t)}|\le e^{-cu^2},
 \qquad |u|\le\eta_0\sqrt d.
 \label{eq:growing-scaled-phase-decay}
\end{equation}
\end{lemma}

\begin{proof}
From the differentiated Poisson series one has, uniformly for
$|t|\le\tau/2$,
\[
 |L^{(3)}(-\tau+it)|\le C\tau^{-3}.
\]
By \eqref{eq:growing-cumulants}, after increasing $\alpha_1$,
$L''(-\tau)\ge(3/8)\tau^{-2}$.  Hence for
$|t|\le\eta_0\tau$, with $\eta_0$ fixed sufficiently small,
\[
 \Re L''(-\tau+it)\ge\frac1{4\tau^2}.
\]
Put $F(t)=\Re\Phi_\rho(t)$.  Then
\[
 F''(t)=-\Re L''(-\tau+it)\le-\frac1{4\tau^2},
 \qquad F(0)=F'(0)=0.
\]
For $0\le t\le\eta_0\tau$,
\[
 F'(t)=\int_0^tF''(s)\,ds\le-\frac{t}{4\tau^2},
 \qquad
 F(t)=\int_0^tF'(r)\,dr\le-\frac{t^2}{8\tau^2}.
\]
The same inequality for negative $t$ follows by integrating from $t$ to
$0$ (equivalently, from the evenness of $F$).  This proves
\eqref{eq:growing-phase-decay}.  With $t=\tau u/\sqrt d$ it gives
$d\Re\Phi_\rho(t)\le-cu^2$, and hence
\eqref{eq:growing-scaled-phase-decay}.
\end{proof}

In the growing range, set
\[
 t=\frac{\tau u}{\sqrt d}.
\]

\Needspace{5\baselineskip}
\begin{lemma}[growing central expansion]
\label{lem:growing-central}
For every fixed $A>0$ there are $\alpha_1,d_0,C>0$ such that whenever
$d\ge d_0$ and
\[
 \alpha_1\le\frac nd\le Ad,
\]
then, for the central arc $E_0=\{|t|\le\eta_0\tau\}$,
\begin{align}
 \left|\frac{J_+(\xi;E_0)}{J_+(0;E_0)}
 -\Gamma_{d,\rho}(\xi)\right|&\le Cd^{-1},
 \label{eq:growing-central-plus}\\
 \left|\frac{J_-(\xi;E_0)}{J_-(0;E_0)}
 -\Gamma_{d,\rho}(\xi)\right|&\le Cd^{-1}.
 \label{eq:growing-central-minus}
\end{align}
Moreover
\begin{align}
 J_+(0;E_0)
 &=\frac1{2\pi\sqrt d}
 \left\{
 \frac{\tau}{1-e^{-\tau}}\sqrt{\frac{2\pi}{q_g(\tau)}}
 +O_A(d^{-1})\right\},
 \label{eq:growing-base-plus-explicit}\\
 J_-(0;E_0)
 &=\frac{\tau}{2\pi\sqrt d}
 \left\{
 \frac1{1+e^{-\tau}}\sqrt{\frac{2\pi}{q_g(\tau)}}
 +O_A(d^{-1})\right\}.
 \label{eq:growing-base-minus-explicit}
\end{align}
Consequently
\begin{equation}
 |J_+(0;E_0)|\ge c d^{-1/2},
 \qquad
 |J_-(0;E_0)|\ge c\tau d^{-1/2}.
 \label{eq:growing-base-lower}
\end{equation}
Finally
\begin{equation}
 |\Gamma_{d,\rho}(\boldsymbol\omega)|\le e^{-cd/\tau}.
 \label{eq:growing-cross-saddle}
\end{equation}
\end{lemma}

\begin{proof}
Put $t=\tau u/\sqrt d$.  For $|u|\le\eta_0\sqrt d$ one has
$|t|\le\tau/2$.  Taylor's theorem along the segment
$[-\tau,-\tau+it]$, together with
\eqref{eq:growing-complex-fifth}, gives
\begin{align*}
 &d\left|L(-\tau+it)-L(-\tau)-itL'(-\tau)
 -\frac{(it)^2}{2}L''(-\tau)
 -\frac{(it)^3}{6}L^{(3)}(-\tau)
 -\frac{(it)^4}{24}L^{(4)}(-\tau)\right|\\
 &\qquad\le C d\tau^{-5}|t|^5
 \le C\frac{|u|^5}{d^{3/2}}.
\end{align*}
Define the exact scaled coefficients
\begin{equation}
 q_g(\tau)=\tau^2L''(-\tau),\qquad
 b_g(\tau)=-\frac{\tau^3L^{(3)}(-\tau)}6,\qquad
 c_g(\tau)=\frac{\tau^4L^{(4)}(-\tau)}{24}.
 \label{eq:growing-exact-coefficients}
\end{equation}
Then the preceding Taylor estimate gives
\begin{equation}
 d\Phi_\rho(\tau u/\sqrt d)
 =-\frac{q_g(\tau)}2u^2
 +\frac{ib_g(\tau)}{\sqrt d}u^3
 +\frac{c_g(\tau)}d u^4
 +O\!\left(\frac{|u|^5}{d^{3/2}}\right).
 \label{eq:growing-central-phase}
\end{equation}
We retain the exact coefficient $q_g(\tau)$ in the Gaussian term.  By
\eqref{eq:growing-cumulants}--\eqref{eq:growing-exact-quadratic},
$q_g\in[3/8,5/8]$ and $|b_g|+|c_g|\le C$.

Set
\[
 \Psi_d(u)=d\Phi_\rho(\tau u/\sqrt d),
 \qquad
 A_{+,d}(u)=\frac{1-e^{-\tau}}
 {1-e^{-\tau+i\tau u/\sqrt d}},
 \qquad
 A_{-,d}(u)=\frac{1+e^{-\tau}}
 {1+e^{-\tau+i\tau u/\sqrt d}}.
\]
For $r=1,2$, direct differentiation gives
\begin{equation}
 |A_{\pm,d}(u)|\le C,
 \qquad
 |\partial_u^rA_{\pm,d}(u)|\le C_r d^{-r/2}.
 \label{eq:growing-pole-derivatives}
\end{equation}
Indeed $1-e^{-\tau}\simeq\tau$, $1+e^{-\tau}\simeq1$, and
$|u|\le\eta_0\sqrt d$ with $\eta_0\le1/2$.  The phase bound
\eqref{eq:growing-scaled-phase-decay}, the product estimates
\eqref{eq:growing-product-contract-central}--\eqref{eq:growing-product-second},
\eqref{eq:growing-central-phase}, and \eqref{eq:growing-pole-derivatives}
therefore verify \cref{lem:symmetric-saddle} with
\[
 N=d,\qquad q=q_g(\tau),\qquad b_3=b_g(\tau),\qquad b_4=c_g(\tau),
 \qquad G(\xi)=\Gamma_{d,\rho}(\xi),
 \quad E(\xi)=\mathcal E_\tau(\xi).
\]
Since
\begin{align*}
 J_+(\xi;E_0)
 &=\frac{1}{2\pi\sqrt d}\frac{\tau}{1-e^{-\tau}}
   \int_{-\eta_0\sqrt d}^{\eta_0\sqrt d}
   e^{\Psi_d(u)}A_{+,d}(u)P_\rho(\tau u/\sqrt d,\xi)\,du,\\
 J_-(\xi;E_0)
 &=\frac{\tau}{2\pi\sqrt d}\frac1{1+e^{-\tau}}
   \int_{-\eta_0\sqrt d}^{\eta_0\sqrt d}
   e^{\Psi_d(u)}A_{-,d}(u)P_\rho(\tau u/\sqrt d,\xi)\,du,
\end{align*}
the ratio conclusion \eqref{eq:symmetric-ratio} proves
\eqref{eq:growing-central-plus}--\eqref{eq:growing-central-minus}.  The base
conclusion \eqref{eq:symmetric-base} gives precisely
\eqref{eq:growing-base-plus-explicit}--\eqref{eq:growing-base-minus-explicit}.
Since
\[
 1\le\frac{\tau}{1-e^{-\tau}}\le2,
 \qquad \frac12\le\frac1{1+e^{-\tau}}\le1
\]
for small $\tau$, the leading terms dominate the $O_A(d^{-1})$ errors for
large $d$, which proves \eqref{eq:growing-base-lower}.

At $x=1/2$, \eqref{eq:growing-poisson} gives
$|h(e^{-\tau},1/2)|/h(e^{-\tau})\le Ce^{-c/\tau}$, and multiplication over
the $d$ coordinates proves \eqref{eq:growing-cross-saddle}.
\end{proof}
\section{Arithmetic separation of the growing remote arcs and the \texorpdfstring{$\ell^2$}{l2} theorem}
\label{sec:arithmetic-remote}

It remains to control the contour away from $t=0$ and $t=\pi$.  We first
prove a one-dimensional estimate uniform in the frequency and then take its
$d$-fold product.

Define
\[
 \Theta_\tau(t,x)=h(e^{-\tau+it},x)
 =\sum_{k\in\mathbb Z}e^{-\tau k^2}e^{itk^2}e^{2\pi ikx},
 \qquad H_\tau=\Theta_\tau(0,0).
\]
Poisson summation at $t=x=0$ gives the useful lower bound
\begin{equation}
 H_\tau=\sqrt{\frac\pi\tau}
 \sum_{m\in\mathbb Z}e^{-\pi^2m^2/\tau}
 \ge\sqrt{\frac\pi\tau}.
 \label{eq:H-lower}
\end{equation}
Write
\[
 \Delta(t)=\operatorname{dist}(t,\pi\mathbb Z).
\]

\begin{lemma}[quadratic Weyl estimate followed by Abel summation]
\label{lem:weyl-abel}
There is an absolute constant $C$ with the following property.  Suppose
$p,q\in\mathbb Z$, $q\ge1$, $(p,q)=1$, and
\[
 \left|\frac t\pi-\frac pq\right|\le q^{-2}.
\]
Then for every $B>0$, $0<\tau\le1$, and $x\in\mathbb R$,
\begin{equation}
 |\Theta_\tau(t,x)|
 \le
 2C\left(\frac{B/\sqrt\tau+1}{\sqrt q}+\sqrt q\right)+3
 +\frac{2e^{-B^2}}{B\sqrt\tau}.
 \label{eq:weyl-abel-raw}
\end{equation}
Consequently
\begin{align}
 \frac{|\Theta_\tau(t,x)|}{H_\tau}
 \le{}&
 \frac{2CB}{\sqrt{\pi q}}
 +\frac{2C\sqrt\tau}{\sqrt{\pi q}}
 +2C\sqrt{\frac{q\tau}{\pi}}
 +3\sqrt{\frac\tau\pi}
 +\frac{2e^{-B^2}}{B\sqrt\pi}.
 \label{eq:weyl-abel-normalized}
\end{align}
\end{lemma}

\begin{proof}
Let
\[
 S_M(t,x)=\sum_{1\le k\le M}e^{itk^2}e^{2\pi ikx}.
\]
Put $X=t/\pi$, so that $S_M(t,x)=\sum_{1\le k\le M}e(\alpha k^2+xk)$ with
$\alpha=X/2$.  Write $p/(2q)$ in lowest terms as $a/Q$: if $p$ is even,
then $q$ is odd, $a=p/2$, and $Q=q$; if $p$ is odd, then $(p,2q)=1$,
$a=p$, and $Q=2q$.  In both cases $q\le Q\le2q$ and
\[
 \left|\alpha-\frac aQ\right|
 =\frac12\left|X-\frac pq\right|
 \le\frac1{2q^2}\le2Q^{-2}<3Q^{-2}.
\]
Hence \eqref{eq:weyl-sum}, applied with $\theta=x$ and $L=M$ for every
$1\le M\le N$, gives
\[
 |S_M(t,x)|\le C_{\mathrm W}\left(\frac M{\sqrt Q}+\sqrt Q\right)
 \le C\left(\frac{M+1}{\sqrt q}+\sqrt q\right)
\]
with an absolute constant $C$, because $\sqrt q\le\sqrt Q\le\sqrt{2q}$.
For $N=\lceil B/\sqrt\tau\rceil$, Abel summation and the monotonicity of
$e^{-\tau k^2}$ imply
\[
 \left|\sum_{k=1}^Ne^{-\tau k^2}e^{itk^2}e^{2\pi ikx}\right|
 \le \max_{M\le N}|S_M(t,x)|.
\]
The same estimate applies to negative $k$.  The Gaussian tail satisfies
\[
 \sum_{k>N}e^{-\tau k^2}
 \le\int_N^\infty e^{-\tau u^2}\,du
 \le\frac{e^{-\tau N^2}}{2\tau N}
 \le\frac{e^{-B^2}}{2B\sqrt\tau}.
\]
Since $N+1\le2(B/\sqrt\tau+1)$, the term $k=0$, the two truncated
half-lines, and the two tails give \eqref{eq:weyl-abel-raw} after enlarging
the absolute constant $C$.  Dividing by \eqref{eq:H-lower} gives
\eqref{eq:weyl-abel-normalized}.
\end{proof}

For bounded rational denominators we use a direct residue-class
computation.  For coprime $a,Q$, define the quadratic Gauss sums
\[
 G_Q(a,b)=\sum_{r=0}^{Q-1}e\!\left(\frac{ar^2+br}{Q}\right).
\]

\begin{lemma}[Gauss sums and the rational cusp bound]
\label{lem:rational-cusp}
\emph{(i)}  If $(a,Q)=1$, then
\begin{equation}
 |G_Q(a,b)|\le
 \begin{cases}
 \sqrt Q,&Q\ \text{odd},\\
 \sqrt{2Q},&Q\ \text{even}.
 \end{cases}
 \label{eq:gauss-sum-bound}
\end{equation}
In particular, for $Q\ge3$,
\begin{equation}
 \max_b\frac{|G_Q(a,b)|}{Q}\le\frac1{\sqrt2}.
 \label{eq:gauss-height-gap}
\end{equation}

\emph{(ii)}  Let $(a,Q)=1$, $Q\ge1$, $0<\tau\le1$, and write
\[
 t=\frac{2\pi a}{Q}+\beta,
 \qquad w=\tau-i\beta,
 \qquad |w|=(\tau^2+\beta^2)^{1/2}.
\]
If $Q\ge3$, then, uniformly in $x$,
\begin{equation}
 \frac{|\Theta_\tau(t,x)|}{H_\tau}
 \le
 \sqrt{\frac{\tau}{2|w|}}
 +\sqrt2\,Q\sqrt{\frac{|w|}{\pi}}.
 \label{eq:cusp-qge3}
\end{equation}
For $Q=1$ or $Q=2$,
\begin{equation}
 \frac{|\Theta_\tau(t,x)|}{H_\tau}
 \le
 \sqrt{\frac{\tau}{|w|}}
 +4\sqrt{\frac{|w|}{\pi}}.
 \label{eq:cusp-exceptional}
\end{equation}
\end{lemma}

\begin{proof}
For (i), we square the modulus and write $r=s+u$ modulo $Q$:
\begin{align*}
 |G_Q(a,b)|^2
 &=\sum_{u\bmod Q}e((au^2+bu)/Q)
   \sum_{s\bmod Q}e(2aus/Q).
\end{align*}
The inner sum is zero unless $Q\mid2au$, in which case its modulus is $Q$.
If $Q$ is odd, $(2a,Q)=1$, so only $u=0$ contributes and
$|G_Q(a,b)|^2=Q$.  If $Q$ is even, the congruence $Q\mid2au$ has at most
two solutions modulo $Q$, hence $|G_Q(a,b)|^2\le2Q$.  For odd $Q\ge3$,
$Q^{-1/2}\le3^{-1/2}<2^{-1/2}$; for even $Q\ge4$ the second bound gives
\eqref{eq:gauss-height-gap}.

For (ii), observe that the coefficient $k\mapsto e(ak^2/Q)$ is
$Q$-periodic.  Fourier inversion on $\mathbb Z/Q\mathbb Z$ gives
\[
 e(ak^2/Q)=\frac1Q\sum_{b=0}^{Q-1}G_Q(a,-b)e(bk/Q).
\]
Therefore, by the one-dimensional Poisson formula,
\begin{align*}
 \Theta_\tau(t,x)
 &=\frac1Q\sum_{b=0}^{Q-1}G_Q(a,-b)
 \sum_{k\in\mathbb Z}e^{-wk^2}e^{2\pi ik(x+b/Q)}\\
 &=\frac1Q\sqrt{\frac\pi w}
 \sum_{b=0}^{Q-1}G_Q(a,-b)
 \sum_{m\in\mathbb Z}
 e^{-\pi^2(m-x-b/Q)^2/w}.
\end{align*}
Taking absolute values and combining the $Q$ shifted lattices into a
single sum gives
\begin{equation}
 |\Theta_\tau(t,x)|
 \le \frac{G_*}{Q}\sqrt{\frac\pi{|w|}}
 \sum_{j\in\mathbb Z}
 \exp\!\left(-\frac{\pi^2\tau}{|w|^2}(j/Q-x)^2\right),
 \label{eq:cusp-poisson-bound}
\end{equation}
where $G_*=\max_b|G_Q(a,b)|$.  Choose an integer $j_0$ with
$|j_0/Q-x|\le(2Q)^{-1}$; comparing the tail sum with an integral then
gives
\[
 \sum_{j\in\mathbb Z}
 e^{-\pi^2\tau(j/Q-x)^2/|w|^2}
 \le1+\frac{2Q|w|}{\sqrt{\pi\tau}}.
\]
Using \eqref{eq:H-lower} in \eqref{eq:cusp-poisson-bound} therefore yields
\begin{equation}
 \frac{|\Theta_\tau(t,x)|}{H_\tau}
 \le \frac{G_*}{Q}\sqrt{\frac\tau{|w|}}
 +2G_*\sqrt{\frac{|w|}{\pi}}.
 \label{eq:cusp-master}
\end{equation}
For $Q\ge3$, \eqref{eq:gauss-sum-bound}--\eqref{eq:gauss-height-gap}
give \eqref{eq:cusp-qge3}.  For $Q=1,2$, the explicit sums give
$G_*/Q\le1$ and $G_*\le2$.  Hence \eqref{eq:cusp-exceptional} follows
directly from \eqref{eq:cusp-master}.
\end{proof}

\Needspace{5\baselineskip}
\begin{proposition}[uniform scalar remote gap]
\label{prop:uniform-remote-gap}
For every $C_0>0$ there exist $\eta\in(0,1/4)$ and
$\tau_0\in(0,1/4]$ such that
\begin{equation}
 \sup_{x\in\mathbb R}
 \frac{|\Theta_\tau(t,x)|}{H_\tau}
 \le1-\eta
 \label{eq:uniform-remote-gap}
\end{equation}
whenever $0<\tau\le\tau_0$ and $\Delta(t)\ge C_0\tau$.
\end{proposition}

\begin{proof}
We choose the parameters in the order in which they enter the two rational
regimes.  First take $B\ge1$ so large that
\[
 \frac{2e^{-B^2}}{B\sqrt\pi}\le\frac1{32},
\]
and then choose an integer $Q_0\ge4$ such that
\[
 \frac{2CB}{\sqrt{\pi Q_0}}\le\frac1{32},
\]
where $C$ is the constant in \cref{lem:weyl-abel}.

For the bounded-denominator cusps set
\[
 \gamma_0=(1+C_0^2)^{-1/4},
 \qquad \gamma_1=2^{-1/2}.
\]
Both numbers are strictly smaller than $1$.  Choose $\eta>0$ so that
\begin{equation}
 4\eta<\min\{1-\gamma_0,1-\gamma_1,1/4\}.
 \label{eq:remote-eta-choice}
\end{equation}
Next choose $\delta\in(0,\pi/4]$ so small that
\begin{equation}
 \sqrt2(2Q_0)\sqrt{\frac{\sqrt2\delta}{\pi}}
 \le\eta,
 \qquad
 4\sqrt{\frac{\sqrt2\delta}{\pi}}\le\eta.
 \label{eq:delta-choice}
\end{equation}
Finally choose $\tau_0\in(0,1/4]$ so small that $\tau_0\le\delta$,
\begin{equation}
 \frac{2C\sqrt{\tau_0}}{\sqrt\pi}\le\eta,
 \qquad
 2C\sqrt{\frac{\sqrt{\tau_0}}{\pi}}\le\eta,
 \qquad
 3\sqrt{\frac{\tau_0}{\pi}}\le\eta,
 \qquad
 2\pi\sqrt{\tau_0}\le\delta,
 \label{eq:remote-tau-choice}
\end{equation}
and
\[
 N_\tau:=\lfloor\tau^{-1/2}\rfloor\ge Q_0
 \qquad (0<\tau\le\tau_0).
\]

Fix $0<\tau\le\tau_0$ and put $X=t/\pi$.  By Dirichlet's theorem there
are coprime integers $p,q$ with $1\le q\le N_\tau$ and
\begin{equation}
 \left|X-\frac pq\right|
 \le\frac1{qN_\tau}\le q^{-2}.
 \label{eq:dirichlet-strong}
\end{equation}

Suppose first that $q\ge Q_0$.  Since
$q\tau\le N_\tau\tau\le\sqrt\tau$, the five terms on the right-hand side
of \eqref{eq:weyl-abel-normalized} are bounded respectively by
\[
 \frac1{32},\qquad \eta,\qquad \eta,\qquad \eta,\qquad \frac1{32}.
\]
Therefore
\begin{equation}
 \frac{|\Theta_\tau(t,x)|}{H_\tau}
 \le \frac1{16}+3\eta
 <\frac14<1-\eta,
 \label{eq:remote-large-denominator-close}
\end{equation}
uniformly in $x$, where \eqref{eq:remote-eta-choice} gives $\eta<1/16$.

It remains to consider $q<Q_0$.  The reduced fraction $p/q$ can be written
as $2a/Q$ with $(a,Q)=1$ and $Q\le2q$: if $p$ is even, take
$a=p/2$ and $Q=q$; if $p$ is odd, take $a=p$ and $Q=2q$.  Thus
\[
 t=\frac{2\pi a}{Q}+\beta,
 \qquad
 |\beta|=\pi\left|X-\frac pq\right|
 \le\frac{\pi}{N_\tau}
 \le2\pi\sqrt\tau\le\delta.
\]
Here the penultimate inequality uses $\tau\le1/4$, hence
$N_\tau=\lfloor\tau^{-1/2}\rfloor\ge\tfrac12\tau^{-1/2}$.  Since also
$\tau\le\delta$, we have
\begin{equation}
 |w|=(\tau^2+\beta^2)^{1/2}\le\sqrt2\,\delta.
 \label{eq:remote-w-small}
\end{equation}

If $Q\ge3$, then \eqref{eq:cusp-qge3}, $Q<2Q_0$, and
\eqref{eq:delta-choice} give
\[
 \frac{|\Theta_\tau(t,x)|}{H_\tau}
 \le\gamma_1+\eta.
\]
By \eqref{eq:remote-eta-choice},
$\gamma_1+\eta<1-3\eta<1-\eta$.

Finally suppose $Q=1$ or $Q=2$.  Then $2\pi a/Q\in\pi\mathbb Z$.
Because $|\beta|\le\delta\le\pi/4$, this is the unique nearest point of
$\pi\mathbb Z$ to $t=2\pi a/Q+\beta$, and therefore
\[
 \Delta(t)=|\beta|.
\]
The remote condition now gives $|\beta|\ge C_0\tau$.  Hence
\[
 \sqrt{\frac\tau{|w|}}
 =\left(1+(\beta/\tau)^2\right)^{-1/4}
 \le\gamma_0.
\]
Using \eqref{eq:cusp-exceptional}, \eqref{eq:remote-w-small}, and the
second inequality in \eqref{eq:delta-choice}, we obtain
\[
 \frac{|\Theta_\tau(t,x)|}{H_\tau}
 \le\gamma_0+\eta<1-3\eta<1-\eta.
\]
The three cases prove \eqref{eq:uniform-remote-gap}.
\end{proof}

Let $\eta_0$ be the central-width constant from
\cref{lem:growing-phase-decay}.  The scalar gap immediately gives a
$d$-dimensional remote estimate.  Put
\[
 \mathcal R_\tau=\{t\in[-\pi,\pi]:\Delta(t)\ge\eta_0\tau\},
 \qquad T_\pm(\xi)=J_\pm(\xi;\mathcal R_\tau).
\]

\begin{lemma}[growing remote estimates at every algebraic order]
\label{lem:growing-remote}
For every fixed $A>0$ there exist $\alpha_r(A)\ge1$, $d_r(A)\ge1$,
$\eta_A>0$, and $C_A<\infty$ such that, whenever
$d\ge d_r(A)$ and
\[
 \alpha_r(A)d\le n\le Ad^2,
\]
the complements of the two central windows satisfy, uniformly in $\xi$,
\begin{equation}
 \frac{|T_+(\xi)|+|T_-(\xi)|+|T_+(0)|+|T_-(0)|}
 {|J_+(0;E_0)|}
 \le C_A d^{3/2}(1-\eta_A)^d.
 \label{eq:growing-remote-exponential}
\end{equation}
Consequently, for every integer $K\ge1$ there are
$d_{r,K}(A)\ge d_r(A)$ and $C_{A,K}<\infty$ such that, whenever
$d\ge d_{r,K}(A)$ and $\alpha_r(A)d\le n\le Ad^2$,
\begin{equation}
 |T_+(\xi)|+|T_-(\xi)|+|T_+(0)|+|T_-(0)|
 \le C_{A,K}d^{-K}|J_+(0;E_0)|.
 \label{eq:growing-remote-budget}
\end{equation}
\end{lemma}

\begin{proof}
Take $C_0:=\eta_0$, with $\eta_0$ fixed in
\cref{lem:growing-phase-decay}; thus $\Delta(t)<C_0\tau$ is exactly the
union of the two central windows used in \cref{lem:growing-central}.
Choose $\alpha_r(A)$ and $d_r(A)$ large enough that
$\tau\le\tau_0$ throughout the stated range and that
\eqref{eq:growing-base-lower} from \cref{lem:growing-central} holds whenever
$d\ge d_r(A)$ and $n/d\ge\alpha_r(A)$.  Denote the gap supplied by
\cref{prop:uniform-remote-gap} by $\eta_A>0$.  On the complement,
\eqref{eq:uniform-remote-gap} gives, for every coordinate,
\[
 \frac{|h(e^{-\tau+it},\xi_j)|}{h(e^{-\tau})}\le1-\eta_A.
\]
Therefore
\begin{equation}
 \frac{\prod_{j=1}^d|h(e^{-\tau+it},\xi_j)|}{h(e^{-\tau})^d}
 \le(1-\eta_A)^d.
 \label{eq:product-remote-gap}
\end{equation}
For either pole,
\[
 |1\mp e^{-\tau+it}|^{-1}
 \le(1-e^{-\tau})^{-1}\le2\tau^{-1}
\]
when $\tau\le1$.  Since the total angular length is at most $2\pi$,
\begin{equation}
 |T_\pm(\xi)|\le2\tau^{-1}(1-\eta_A)^d.
 \label{eq:remote-integral-pre}
\end{equation}
In the range $n/d=\alpha\le Ad$, \eqref{eq:growing-tau} implies
$\tau^{-1}\le4Ad$.  Hence
\[
 |T_\pm(\xi)|\le C_A d(1-\eta_A)^d.
\]
On the other hand, \eqref{eq:growing-base-lower} gives
$|J_+(0;E_0)|\ge c d^{-1/2}$.  Dividing and summing the four terms gives
\eqref{eq:growing-remote-exponential}.  Finally, for every fixed $K$,
exponential decay dominates $d^{-K-3/2}$; after increasing the fixed lower
dimension threshold this yields \eqref{eq:growing-remote-budget}.
\end{proof}

\subsection{Growing residual and maximal estimate}

\begin{theorem}[growing two-saddle residual]
\label{thm:growing-residual}
For every fixed $A>0$ there exist $\alpha_1(A)\ge1$ and $C_A<\infty$ such
that, for all $d\ge1$ and
\[
 \alpha_1(A)d\le n\le Ad^2,
\]
one has
\begin{equation}
 \sup_{\xi\in\mathbb T^d}|r_{d,n}(\xi)|\le C_A d^{-1}.
 \label{eq:growing-residual}
\end{equation}
\end{theorem}

\begin{proof}
Choose $\alpha_c(A)$ and $d_c(A)$ so that the conclusions of
\cref{lem:growing-central} hold whenever $n/d\ge\alpha_c(A)$ and
$d\ge d_c(A)$.  Let $\alpha_r(A)$ and $d_r(A)$ be supplied by
\cref{lem:growing-remote}, and set
\[
 \alpha_1(A)=\max\{\alpha_c(A),\alpha_r(A)\},\qquad
 d_0(A)=\max\{d_c(A),d_r(A)\}.
\]
For $d\ge d_0(A)$, set $D=J_+(0;[-\pi,\pi])$ and combine the two central
arcs from \cref{lem:growing-central} with the remote remainder from
\cref{lem:growing-remote}.  At frequency zero, the negative central arc is
a cross-saddle term: after \eqref{eq:negative-arc-transport} it contains
$J_-(\boldsymbol\omega;E_0)$.  By
\eqref{eq:growing-central-minus} and \eqref{eq:growing-cross-saddle},
\[
 \frac{|J_-(\boldsymbol\omega;E_0)|}{|J_-(0;E_0)|}
 \le C_A d^{-1}+e^{-cd/\tau}.
\]
Since $|J_-(0;E_0)|/|J_+(0;E_0)|\lesssim\tau\le1$, and the remote part is
$O_A(d^{-2})$ by \cref{lem:growing-remote}, the full denominator differs
from $J_+(0;E_0)$ by $O_A(d^{-1})$ relative to $J_+(0;E_0)$ in modulus.  Thus
$c_+:=J_+(0;E_0)/D=1+O_A(d^{-1})$ and $|c_+|\le2$.  Moreover
\eqref{eq:growing-base-plus-explicit}--
\eqref{eq:growing-base-minus-explicit} give
$c_-:=(-1)^nJ_-(0;E_0)/D$ and $|c_-|\le2$ after enlarging $d_0$.

Set
\[
 F_+(\xi)=\frac{J_+(\xi;E_0)}{J_+(0;E_0)},
 \qquad
 F_-(\xi)=\frac{J_-(\xi;E_0)}{J_-(0;E_0)}.
\]
The exact contour decomposition gives
\[
 m_{d,n}(\xi)=c_+F_+(\xi)+c_-F_-(\xi-\boldsymbol\omega)+E_{\rm rem}(\xi),
 \qquad \|E_{\rm rem}\|_{L^\infty}\le C_A d^{-2}.
\]
For the coefficient of the negative branch, evaluate this identity at
$\xi=\boldsymbol\omega$.  By
\eqref{eq:growing-central-plus} and \eqref{eq:growing-cross-saddle},
\[
 \frac{J_+(\boldsymbol\omega;E_0)}{J_+(0;E_0)}
 =O_A(d^{-1}+e^{-cd/\tau}).
\]
Since $|c_+|\le2$, $m_{d,n}(\boldsymbol\omega)=a_{d,n}$, and the remote
contribution is $O_A(d^{-2})$ by \eqref{eq:growing-remote-budget}, it
follows that $|c_--a_{d,n}|\le C_A d^{-1}$.

The central estimates \eqref{eq:growing-central-plus}--
\eqref{eq:growing-central-minus} and the coefficient bounds above imply
\eqref{eq:growing-residual} directly.  For $d<d_0(A)$, the trivial bound
$|m_{d,n}|+|\mathsf G_{d,n}|\le3$ is absorbed by enlarging $C_A$.
\end{proof}

\begin{corollary}[growing maximal range]
\label{thm:growing}
For every $A>0$ there is $C_A'<\infty$ such that, with
$\alpha_1=\alpha_1(A)$ from \cref{thm:growing-residual},
\[
 \left\|\sup_{\alpha_1d\le n\le Ad^2}|M_{d,n}f|\right\|_2
 \le C_A'\|f\|_2.
\]
\end{corollary}

\begin{proof}
There are at most $Ad^2+1$ squared-radius levels.  By Plancherel,
\[
 \left\|\sup_{\alpha_1d\le n\le Ad^2}|R_{d,n}f|\right\|_2
 \le\left(\sum_n\|R_{d,n}f\|_2^2\right)^{1/2}
 \le C_A\sqrt{A+d^{-2}}\,\|f\|_2.
\]
Adding the centered and translated Gaussian branches, using
\eqref{eq:known-gaussian} and \eqref{eq:translated-gaussian-maximal}, finishes the proof.
\end{proof}

\subsection{Choice of constants}

Fix $\alpha_0$ as in the small range.  Let $C_L$ be the radius threshold in
\eqref{eq:known-large-radius} and put
\[
 A=(C_L+1)^2.
\]
Choose $\alpha_1=\alpha_1(A)$ from \cref{thm:growing-residual}, and apply
\cref{prop:critical-residual} with
\[
 a=\frac{\alpha_0}{2},\qquad b=2\alpha_1.
\]
These choices are independent of $d$, $n$, and $f$.

Consider the three intervals
\[
 [0,\alpha_0d],\qquad
 [\tfrac12\alpha_0d,2\alpha_1d],\qquad
 [\alpha_1d,Ad^2].
\]
Their union contains every squared-radius level $0\le n\le Ad^2$.  If the
third interval is empty, the second already extends beyond $Ad^2$; otherwise
the adjacent intervals overlap.  The corresponding maximal estimates are
\cref{thm:small,thm:critical,thm:growing}.
If $n>Ad^2$, then
\[
 \sqrt n>\sqrt A\,d=(C_L+1)d>C_Ld,
\]
so $M_{d,n}=\Avg_{\sqrt n}$ belongs to the large-radius family in
\eqref{eq:known-large-radius}.  Therefore
\begin{equation}
 \left\|\sup_{n\ge0}|M_{d,n}f|\right\|_2\le C_2\|f\|_2,
 \qquad d\ge1,
 \label{eq:first-order-full-l2}
\end{equation}
with $C_2$ independent of $d$.  By \cref{lem:integer-radius-reduction},
this is also the full real-radius $\ell^2$ estimate.
\section{Higher-order two-saddle approximation}
\label{sec:higher-order}

For $1<p<2$ we need faster residual decay.  We therefore retain an arbitrary
finite number of terms in the two saddle expansions.  The correction terms are
radial derivatives of the normalized Gaussian multiplier; finite differences
replace them by bounded linear combinations of normalized Gaussians at nearby
real scales.

Fix an integer $K\ge1$.  Constants may depend on $K$ in the small range,
on $K,a,b$ in the fixed critical window $ad\le n\le bd$, and on $K,A$ in
the growing range $n\le Ad^2$.  In arguments written once for all three
regimes we suppress $a,b,A$ from the notation; the range dependence is
restored explicitly in the regime-specific conclusions.

\subsection{Uniform derivatives to arbitrary fixed order}

Write $s=\log\rho$ and
\[
 \Gamma_s(\xi)=\Gamma_{d,e^s}(\xi)
 =\prod_{j=1}^d\frac{h(e^s,\xi_j)}{h(e^s,0)}.
\]
On any zero-free strip used below we use the same notation for the analytic
continuation
\[
 \Gamma_z(\xi)=\prod_{j=1}^d
 \frac{h(e^z,\xi_j)}{h(e^z,0)}.
\]
For the three finite regimes we use the following notation:
\[
\begin{array}{c|c|c|c}
\text{range}&N&\sigma&\mathcal E(\xi)\\ \hline
n\le\alpha_0d&n&1&\rho\sum_j\delta(\xi_j)\\
 ad\le n\le bd&d&1&\sum_j\delta(\xi_j)\\
 \alpha_1d\le n\le Ad^2&d&\tau&
 \tau^{-1}\sum_j\operatorname{dist}(\xi_j,\mathbb Z)^2.
\end{array}
\]
Thus the local variable is always
\begin{equation}
 t=\frac{\sigma u}{\sqrt N}.
 \label{eq:unified-local-variable}
\end{equation}
In the three ranges the Gaussian curvature parameter of
\cref{lem:symmetric-saddle} is $q=\lambda_2(\alpha)$, $q=v(\rho)$, and
$q=q_g(\tau)$, respectively, by \eqref{eq:small-normalized-cumulants},
\eqref{eq:critical-variance}, and \eqref{eq:growing-exact-quadratic}, and
the positive central arc is $E_0=\{|t|\le\eta\sigma\}$ with
$\eta=\pi/6$, $\eta=\eta_{a,b}$, and $\eta=\eta_0$, respectively
(\cref{lem:small-central}, the width fixed after \cref{lem:critical-strip},
and \cref{lem:growing-phase-decay}).
We normalize the two Cauchy kernels at the saddle by
\begin{equation}
 \mathfrak a_\pm(u)=\frac{1\mp\rho}{1\mp\rho e^{i\sigma u/\sqrt N}},
 \qquad \mathfrak a_\pm(0)=1.
 \label{eq:unified-pole-factor}
\end{equation}
The upper sign corresponds to $J_+$ and the lower sign to $J_-$.  The
constants $a,b$ are fixed in the critical range, and $\rho=e^{-\tau}$ in
the growing range.

\begin{lemma}[differentiation formulas]
\label{lem:finite-differentiation-algebra}
Fix an integer $m\ge1$.  The coefficients in the following differentiation
identities depend only on $m$.

If $G$ is $C^m$, then
\begin{equation}
 \partial^m e^G
 =e^G
 \sum_{\substack{\nu_1,\ldots,\nu_m\ge0\\
       \sum_{j=1}^m j\nu_j=m}}
 c_{m,\nu}\prod_{j=1}^m(\partial^jG)^{\nu_j}.
 \label{eq:finite-exp-differentiation}
\end{equation}
If $F$ is $C^m$ and nonvanishing, then
\begin{equation}
 \partial^m\log F
 =\sum_{\substack{\nu_1,\ldots,\nu_m\ge0\\
       \sum_{j=1}^m j\nu_j=m}}
 d_{m,\nu}\,
 F^{-(\nu_1+\cdots+\nu_m)}
 \prod_{j=1}^m(\partial^jF)^{\nu_j}.
 \label{eq:finite-log-differentiation}
\end{equation}
Consequently, if, for $i=0,1$,
\[
 |F_i|\ge\kappa>0,
 \qquad
 \max_{0\le j\le m}|\partial^jF_i|\le M,
\]
and
\[
 \max_{0\le j\le m}|\partial^j(F_1-F_0)|\le\varepsilon,
\]
then
\begin{equation}
 |\partial^m\log F_1-\partial^m\log F_0|
 \le C_{m,\kappa,M}\varepsilon.
 \label{eq:finite-log-stability}
\end{equation}
Also, if
\[
 |\partial^jG|\le A_j\lambda^jE,
 \qquad 1\le j\le m,
 \qquad
 |e^G|\le e^{-cE},
\]
with $E\ge0$, then
\begin{equation}
 |\partial^m e^G|
 \le C_{m,c,A_1,\ldots,A_m}\lambda^m(1+E)^m e^{-cE}.
 \label{eq:finite-exp-bound}
\end{equation}
\end{lemma}

\begin{proof}
Both formulas follow by induction on $m$.  Differentiating a monomial in
\eqref{eq:finite-exp-differentiation} either differentiates one factor or
uses $(e^G)'=e^GG'$, and hence preserves the constraint
$\sum j\nu_j=m+1$.  The logarithmic formula is obtained in the same way from
$(\log F)'=F'/F$.  The stability estimate follows by subtracting the two
finite sums in \eqref{eq:finite-log-differentiation}; every difference of
products is telescoped; the bounds $\kappa$ and $M$ control the denominators
and the remaining factors.  Finally,
\eqref{eq:finite-exp-differentiation} gives a factor
$\lambda^{\sum j\nu_j}=\lambda^m$, while
$E^{\sum\nu_j}\le(1+E)^m$, which proves
\eqref{eq:finite-exp-bound}.
\end{proof}

\begin{lemma}[all-order central bounds]
\label{lem:all-order-central-bounds}
Fix an integer $M\ge1$.  There is a level $n_M$ for the small range; for
every fixed $0<a\le b<\infty$ there is a dimension $d_{M,a,b}$ for the
critical range; and for every fixed $A>0$ there are
$\alpha_{M,A}$ and $d_{M,A}$ for the growing range.  Subject to these
lower thresholds, the following estimates hold uniformly in the
corresponding range.

Constants may depend on $M$ and on the fixed regime parameters ($a,b$ or
$A$); the nearby-scale estimates may also depend on $B$.  They are independent
of $d$, $n$, and $\xi$.

Let
\[
 \Psi(u)=d\Phi_\rho(\sigma u/\sqrt N),
 \qquad
 P(u,\xi)=P_\rho(\sigma u/\sqrt N,\xi).
\]
There are $q_0,q_1,c>0$ such that
\begin{equation}
 q_0\le q:=\frac{d\sigma^2}{N}L''(s)\le q_1,
 \qquad
 \Re\Psi(u)\le-cu^2
 \label{eq:all-order-phase-decay}
\end{equation}
throughout the central interval $|u|\le\eta\sqrt N$, with the central
width $\eta$ fixed in the corresponding range.  Moreover, for
$2\le r\le M$,
\begin{equation}
 \left|\frac{d\sigma^r}{N}L^{(r)}(s+it)\right|\le C_M,
 \qquad |t|\le\eta\sigma,
 \label{eq:all-order-phase-derivatives}
\end{equation}
and for $0\le r\le M$,
\begin{align}
 |\partial_u^r \mathfrak a_\pm(u)|&\le C_MN^{-r/2},
 \label{eq:all-order-pole-derivatives}\\
 |\partial_u^rP(u,\xi)|&\le
 C_MN^{-r/2}(1+\mathcal E(\xi))^r e^{-c\mathcal E(\xi)}.
 \label{eq:all-order-product-derivatives}
\end{align}
Here $\mathfrak a_\pm$ is the normalized pole factor in
\eqref{eq:unified-pole-factor}.

Finally, for each fixed $B>0$, after increasing the lower threshold once
more, if $|s'-s|\le B\sigma/\sqrt N$, then $e^{s'}\in(0,1)$ and
\begin{equation}
 \left|\sigma^r\partial_s^r\Gamma_{s'}(\xi)\right|\le C_{M,B},
 \qquad 0\le r\le M.
 \label{eq:all-order-real-scale-derivatives}
\end{equation}
\end{lemma}

\begin{proof}
In the small range, for every fixed $m\ge0$ and $|t|\le\pi/6$,
\begin{align*}
 &\partial_s^m\{h(e^{s+it},x)-h(e^{s+it},0)\}\\
 &\qquad=2\sum_{k\ge1}k^{2m}\rho^{k^2}e^{ik^2t}
       \{\cos(2\pi kx)-1\},
\end{align*}
so that
\begin{equation}
 \left|\partial_s^m\{h(e^{s+it},x)-h(e^{s+it},0)\}\right|
 \le C_m\rho\delta(x).
 \label{eq:small-all-order-difference}
\end{equation}
The theta factors are bounded away from zero on this strip, and their
positive-order radial derivatives are $O_m(\rho)$.  Applying
\eqref{eq:finite-log-stability} to
$F_1=h(e^{s+it},x)$ and $F_0=h(e^{s+it},0)$ gives
\begin{equation}
 \left|\partial_s^m\log
 \frac{h(e^{s+it},x)}{h(e^{s+it},0)}\right|
 \le C_m\rho\delta(x),\qquad m\ge1.
 \label{eq:small-all-order-log}
\end{equation}
The same finite differentiation formula at $x=0$ gives
$|L^{(m)}(s+it)|\le C_m\rho$ for $m\ge1$.  Since
$n=d\mu(\rho)$ and $\mu(\rho)\simeq\rho$, this proves
\eqref{eq:all-order-phase-derivatives} in the small range.  Summing
\eqref{eq:small-all-order-log} over the coordinates and using
$t=u/\sqrt n$ gives
\[
 |\partial_u^j\log P(u,\xi)|
 \le C_j n^{-j/2}\mathcal E(\xi),
 \qquad 1\le j\le M.
\]
Together with the contraction \eqref{eq:small-product-contract},
\eqref{eq:finite-exp-bound} now gives
\eqref{eq:all-order-product-derivatives}.  Direct differentiation of the two
pole factors gives \eqref{eq:all-order-pole-derivatives}.

In a fixed critical window, the proof of \cref{lem:critical-strip} gives a
uniform zero-free strip.  Absolute convergence yields, for every fixed
$m$,
\[
 \left|\partial_s^m\{h(e^{s+it},x)-h(e^{s+it},0)\}\right|
 \le C_m\delta(x),
\]
and all radial derivatives of the theta factors are uniformly bounded.
The logarithmic stability estimate \eqref{eq:finite-log-stability} gives the
analogue of \eqref{eq:small-all-order-log} with $\rho\delta(x)$ replaced by
$\delta(x)$.  After summing over the coordinates and using $t=u/\sqrt d$,
\eqref{eq:finite-exp-bound} gives
\eqref{eq:all-order-product-derivatives}.  The phase and pole bounds,
\eqref{eq:all-order-phase-derivatives} and
\eqref{eq:all-order-pole-derivatives}, follow directly from compactness.

It remains to justify the arbitrary-order statement in the growing range.
Put $w=\tau+i\sigma_1$, $|\sigma_1|\le\tau/2$.  For $\Re w>0$, write
\[
 S(w):=2\sum_{m\ge1}e^{-\pi^2m^2/w}.
\]
For $a>0$ and every $j\ge0$,
\begin{equation}
 \partial_w^j e^{-a/w}=w^{-j}P_j(a/w)e^{-a/w},
 \label{eq:all-order-exp-derivative}
\end{equation}
where $P_j$ is a polynomial depending only on $j$.  Since
$\Re(1/w)\ge4/(5\tau)$, summing \eqref{eq:all-order-exp-derivative} in the
Poisson series and using \eqref{eq:exp-absorbs-powers} gives, after changing
$c>0$ from line to line,
\begin{equation}
 |\partial_w^jS(w)|\le C_j\tau^{-j}e^{-c/\tau},
 \qquad j\ge0,
 \label{eq:growing-all-order-S}
\end{equation}
while differentiation of
$L(-w)=\frac12\log\pi-\frac12\log w+\log(1+S(w))$ gives
\begin{equation}
 |L^{(j)}(-w)|\le C_j\tau^{-j},
 \qquad j\ge1.
 \label{eq:growing-all-order-L}
\end{equation}

For the quotient $R(w,x)=h(e^{-w},x)/h(e^{-w},0)$ we prove, for
$r=r(x)=\operatorname{dist}(x,\mathbb Z)$ and every $j\ge1$,
\begin{equation}
 |\partial_w^j\log R(w,x)|
 \le C_j\frac{r^2}{\tau^{j+1}}.
 \label{eq:growing-all-order-log}
\end{equation}
Suppose first that $0\le r\le1/4$.  Recall from
\eqref{eq:growing-nearest-image-explicit} that
$R(w,r)=e^{-\pi^2r^2/w}\frac{1+E(w,r)}{1+E(w,0)}$, where $E(w,r)$ is the even
remainder from the paired nonzero images.  Termwise differentiation and
\eqref{eq:all-order-exp-derivative} yield, for every fixed $j\ge0$,
\begin{equation}
 |\partial_r^2\partial_w^jE(w,r)|
 \le C_j\tau^{-2j-2}
       \sum_{m\ge1}m^{4j+4}e^{-cm^2/\tau}
 \le C_j\tau^{-2j-2}e^{-c/\tau}.
 \label{eq:growing-all-order-nearest-r2}
\end{equation}
Since $E(w,r)$ is even in $r$, integration twice in $r$ gives
\begin{equation}
 |\partial_w^j\{E(w,r)-E(w,0)\}|
 \le C_jr^2\tau^{-2j-2}e^{-c/\tau}.
 \label{eq:growing-all-order-nearest-difference}
\end{equation}
The same termwise estimate without the $r$-difference gives
\begin{equation}
 |\partial_w^jE(w,r)|+|\partial_w^jE(w,0)|
 \le C_j\tau^{-2j}e^{-c/\tau}.
 \label{eq:growing-all-order-nearest-plain}
\end{equation}
For $\tau$ small, $|E(w,r)|+|E(w,0)|\le1/4$.  For $j\ge1$ one may write
explicitly
\begin{equation}
 \partial_w^j\log(1+E)
 =\sum_{\nu_1+2\nu_2+\cdots+j\nu_j=j}
 C_{\nu}
 \frac{\prod_{m=1}^j(\partial_w^mE)^{\nu_m}}
 {(1+E)^{\nu_1+\cdots+\nu_j}},
 \label{eq:growing-log-faa}
\end{equation}
with constants depending only on $j$.  Subtracting the expressions at $r$
and at $0$ and telescoping the numerator products and the inverse powers of
$1+E$, we see from \eqref{eq:growing-all-order-nearest-difference}--
\eqref{eq:growing-all-order-nearest-plain} that every resulting term contains
a factor carrying $r^2$ together with an exponentially small factor.  Hence, for
some integer $M_j$,
\begin{equation}
 \left|\partial_w^j\left\{
 \log(1+E(w,r))-\log(1+E(w,0))\right\}\right|
 \le C_jr^2\tau^{-M_j}e^{-c/\tau}
 \le C_jr^2\tau^{-j-1}.
 \label{eq:growing-all-order-nearest-logerror}
\end{equation}
The last inequality follows from \eqref{eq:exp-absorbs-powers}.  Combining
this with
\[
 \left|\partial_w^j\left(-\frac{\pi^2r^2}{w}\right)\right|
 \le C_jr^2\tau^{-j-1}
\]
proves \eqref{eq:growing-all-order-log} for $r\le1/4$.

Now let $1/4\le r\le1/2$.  Put $b=b(r)=\pi^2(1-2r)$ and
$y=b/w$.  The noncancellation \eqref{eq:growing-two-image-noncancel} gives
\begin{equation}
 |1+e^{-y}|\ge\frac12,
 \qquad
 \Re y\ge0,
 \qquad
 |\Im y|\le\frac12\Re y.
 \label{eq:growing-all-order-sector}
\end{equation}
Let
\[
 F_0(y)=\log(1+e^{-y}),
 \qquad
 F_{j+1}(y)=-jF_j(y)-yF_j'(y).
\]
An induction using $\partial_w y=-y/w$ gives the exact identity
\begin{equation}
 \partial_w^jF_0(b/w)=w^{-j}F_j(b/w),
 \qquad j\ge0.
 \label{eq:growing-Fj}
\end{equation}
Every $F_j$ is a finite linear combination of terms
$y^mF_0^{(m)}(y)$, $0\le m\le j$.  On the sector
\eqref{eq:growing-all-order-sector}, the denominator $1+e^{-y}$ is bounded
away from zero; hence $F_0^{(m)}$ is bounded on compact subsets and, for
$m\ge1$, is $O_m(e^{-\Re y})$ as $\Re y\to\infty$.  Since
$|y|^me^{-\Re y}\le C_m$ there, we obtain
\begin{equation}
 \sup_{\substack{\Re y\ge0\\|\Im y|\le\frac12\Re y}}|F_j(y)|\le C_j,
 \qquad
 |\partial_w^jF_0(b/w)|\le C_j\tau^{-j}.
 \label{eq:growing-Fj-bound}
\end{equation}

It remains to check that the omitted Poisson images remain negligible after
arbitrarily many fixed derivatives.  The Gaussian separation in
\eqref{eq:growing-two-image-explicit} gives, for every $j\ge0$,
\begin{equation}
 |\partial_w^j\mathcal T(w,r)|
 +|\partial_w^j\mathcal T_0(w)|
 \le C_j\tau^{-2j}e^{-c/\tau}.
 \label{eq:growing-all-order-two-image-tail}
\end{equation}
Indeed, for $m\notin\{0,-1\}$ set
\[
 a_m(r)=\pi^2\big((m+r)^2-r^2\big).
\]
The relative quadratic gaps satisfy $a_m(r)\ge3\pi^2/2$ uniformly for
$1/4\le r\le1/2$, and \eqref{eq:all-order-exp-derivative} gives the
single-summand estimate
\[
 \left|\partial_w^j e^{-a_m(r)/w}\right|
 \le C_j\tau^{-j}
       \left(1+\frac{a_m(r)}{\tau}\right)^j
       e^{-4a_m(r)/(5\tau)}.
\]
The polynomial factor is absorbed by the exponential, and summing over the
quadratically growing gaps proves the estimate for $\mathcal T$.  The
denominator tail $\mathcal T_0$ is identical, with
$a_m(0)=\pi^2m^2$ for $m\ne0$.  This proves
\eqref{eq:growing-all-order-two-image-tail}.  Since
$|1+e^{-y}|\ge1/2$, put
\[
 U(w,r)=\frac{\mathcal T(w,r)}{1+e^{-y}}.
\]
Repeated product differentiation, \eqref{eq:growing-Fj-bound}, and
\eqref{eq:growing-all-order-two-image-tail} give, for every fixed $j$,
\[
 |\partial_w^jU(w,r)|\le C_j\tau^{-M_j}e^{-c/\tau}.
\]
After increasing the lower growing threshold, $|U|\le1/4$.  Applying the
finite formula \eqref{eq:growing-log-faa} with $E$ replaced by $U$ shows
that every positive $w$-derivative of $\log(1+U)$ still contains an
exponentially small factor; hence
\begin{equation}
 |\partial_w^j\log(1+U(w,r))|
 +|\partial_w^j\log(1+\mathcal T_0(w))|
 \le C_j\tau^{-M_j}e^{-c/\tau}.
 \label{eq:growing-all-order-two-image-logtail}
\end{equation}
Define the remainder
\[
 H(w,r):=\log(1+U(w,r))-\log(1+\mathcal T_0(w)).
\]
Then for every fixed $j$ and every $0\le m\le j$,
\begin{equation}
 |\partial_w^mH(w,r)|
 \le C_j\tau^{-M_j}e^{-c/\tau}.
 \label{eq:growing-all-order-H}
\end{equation}
Consequently
\[
 \log R(w,r)
 =-\frac{\pi^2r^2}{w}+F_0(b/w)+H(w,r).
\]
The first term is bounded by $C_jr^2\tau^{-j-1}$ after $j$ derivatives.
For the second, \eqref{eq:growing-Fj-bound} gives $C_j\tau^{-j}$; since
$r^2\ge1/16$, after increasing the lower growing threshold so that
$\tau\le1/16$,
\[
 \tau^{-j}\le16r^2\tau^{-j-1}.
\]
The exponentially small last term is absorbed by
\eqref{eq:exp-absorbs-powers}.  This proves
\eqref{eq:growing-all-order-log} in the two-image region as well.

Summing \eqref{eq:growing-all-order-log} over the coordinates and using
$t=\tau u/\sqrt d$ gives
\[
 |\partial_u^j\log P(u,\xi)|
 \le C_jd^{-j/2}\mathcal E_\tau(\xi).
\]
The exponential differentiation estimate \eqref{eq:finite-exp-bound},
together with the growing contraction, gives
\eqref{eq:all-order-product-derivatives}.  The pole estimates follow by
differentiating
\[
 \frac{1-e^{-\tau}}{1-e^{-\tau+i\tau u/\sqrt d}},
 \qquad
 \frac{1+e^{-\tau}}{1+e^{-\tau+i\tau u/\sqrt d}};
\]
each $u$-derivative contributes $\tau/\sqrt d$, while every extra inverse
power of $1-e^{-\tau}$ costs at most $C\tau^{-1}$.  Hence the net factor is
$d^{-1/2}$ per derivative.  Formula \eqref{eq:growing-all-order-L} gives
\eqref{eq:all-order-phase-derivatives} with $N=d$ and $\sigma=\tau$.

Finally consider the real shifts in
\eqref{eq:all-order-real-scale-derivatives}.  In the growing range write
$s'=-\tau'$; if $|s'-s|\le B\tau/\sqrt d$ and $d\ge4B^2$, then
\begin{equation}
 \frac12\tau\le\tau'\le\frac32\tau.
 \label{eq:growing-nearby-tau-comparable}
\end{equation}
Thus the preceding real-axis estimates hold at $s'$ with constants depending
only on the fixed derivative order and $B$.  In the small range the shifted
parameter remains in a fixed small-$\rho$ region after raising the lower
$n$-threshold, and in the critical range it remains in a fixed compact
subinterval of $(0,1)$.

At the shifted parameter, let $\mathcal E_{s'}(\xi)$ denote the corresponding
energy (with the shifted $\rho$ in the small range and the shifted $\tau'$ in
the growing range).  The logarithmic estimates above give,
for $1\le j\le M$,
\[
 \big|\sigma^j\partial_s^j\log\Gamma_{s'}(\xi)\big|
 \le C_{M,B}\mathcal E_{s'}(\xi),
\]
where $C_{M,B}$ is replaced by $C_{M,B,a,b}$ or $C_{M,B,A}$ in the critical
or growing range.  The product contraction gives
\[
 |\Gamma_{s'}(\xi)|\le e^{-c_B\mathcal E_{s'}(\xi)}.
\]
The shifted and unshifted energies are uniformly comparable by the preceding
parameter bounds.  Applying \eqref{eq:finite-exp-differentiation} to
$\Gamma_{s'}=\exp(\log\Gamma_{s'})$ therefore yields
\[
 \big|\sigma^r\partial_s^r\Gamma_{s'}(\xi)\big|
 \le C_{M,B}(1+\mathcal E_{s'}(\xi))^r
       e^{-c_B\mathcal E_{s'}(\xi)}
 \le C_{M,B},
 \qquad 0\le r\le M.
\]
This proves \eqref{eq:all-order-real-scale-derivatives} and completes the proof.
\end{proof}

\Needspace{5\baselineskip}
\begin{lemma}[nearby real scales]
\label{lem:nearby-real-scales}
Fix $K\ge1$ and $B>0$.  After increasing the fixed lower thresholds in the
three finite regimes, let
\[
 s_\ell=s+\ell\frac{\sigma}{\sqrt N},
 \qquad |\ell|\le B.
\]
Then $s_\ell<0$.  In the small range,
\begin{equation}
 \frac12\rho\le e^{s_\ell}\le2\rho\le2^{-10},
 \label{eq:nearby-small-rho}
\end{equation}
after increasing the lower $n$-threshold.  In the critical range all
$e^{s_\ell}$ lie in a fixed compact subinterval of $(0,1)$ containing the
original critical saddle parameters.  In the growing range, writing
$s_\ell=-\tau_\ell$, one has
\begin{equation}
 \frac12\tau\le\tau_\ell\le\frac32\tau.
 \label{eq:nearby-growing-tau}
\end{equation}
In particular,
\begin{equation}
 \sup_{|\ell|\le B}\sup_\xi
 \left|\sigma^r\partial_s^r\Gamma_{s_\ell}(\xi)\right|
 \le C_{K,B},
 \qquad 0\le r\le2K+2.
 \label{eq:nearby-real-scale-derivatives}
\end{equation}
The constant on the right of
\eqref{eq:nearby-real-scale-derivatives} is $C_{K,B}$ in the small range,
$C_{K,B,a,b}$ in the critical range, and $C_{K,B,A}$ in the growing range.
Moreover there are $c_{K,B}>0$ in the small range,
$q_{K,B,a,b}\in(0,1)$ in a fixed critical interval, and $c_{K,B,A}>0$ in
the growing range such that
\begin{equation}
 \sup_{|\ell|\le B}|\Gamma_{s_\ell}(\boldsymbol\omega)|
 \le
 \begin{cases}
 C_{K,B}e^{-c_{K,B}n},& n\le\alpha_0d,\\
 C_{K,B,a,b}q_{K,B,a,b}^d,& ad\le n\le bd,\\
 C_{K,B,A}e^{-c_{K,B,A}d/\tau},& \alpha_1d\le n\le Ad^2.
 \end{cases}
 \label{eq:nearby-cross-saddle}
\end{equation}
\end{lemma}

\begin{proof}
The derivative estimate is \eqref{eq:all-order-real-scale-derivatives}.  We
only spell out the stability of the ranges and of the cross-saddle decay.
In the growing range,
\[
 \tau_\ell=\tau\left(1-\frac{\ell}{\sqrt d}\right),
\]
so \eqref{eq:nearby-growing-tau} holds once $d\ge4B^2$.  Thus
$\tau_\ell\simeq_B\tau$, and the proof of
\eqref{eq:growing-cross-saddle}, applied at $\tau_\ell$, gives
\[
 |\Gamma_{s_\ell}(\boldsymbol\omega)|
 \le C_{K,B,A}e^{-cd/\tau_\ell}
 \le C_{K,B,A}e^{-c_{K,B,A}d/\tau}.
\]
In the critical range, $s$ varies in a compact subset of $(-\infty,0)$; the shifts are
$O_B(d^{-1/2})$, so a slightly larger compact interval contains every
$s_\ell$, and the compact contraction used in
\eqref{eq:critical-base-cross} yields
$C_{K,B,a,b}q_{K,B,a,b}^d$.  In the small range,
\[
 e^{s_\ell}=\rho\exp(\ell/\sqrt n).
\]
After increasing the lower $n$-threshold so that
$e^{B/\sqrt n}\le2$, we obtain \eqref{eq:nearby-small-rho}.  In particular
the shifted parameters stay inside the fixed region $\rho'\le2^{-10}$,
which is contained in the range $\rho'\le2^{-7}$ where the theta estimates
of \cref{lem:small-scale} are uniform.  Since $d e^{s_\ell}\simeq d\rho\simeq n$,
the proof of \eqref{eq:small-cross-saddle} gives
$C_{K,B}e^{-c_{K,B}n}$.  The same bounds show $s_\ell<0$ in all three cases
after increasing the fixed lower thresholds.
\end{proof}

\subsection{The arbitrary-order symmetric expansion}
Before stating the higher-order expansion, we record the normalization
relating the scaled integrals to the original Cauchy integrals.  Take
$(N,\sigma)$ from the corresponding row above, set
\[
 \varepsilon=N^{-1/2},\qquad t=\sigma\varepsilon u,
 \qquad \Psi(u)=d\Phi_\rho(t),\qquad P(u,\xi)=P_\rho(t,\xi),
\]
and let $E_0=\{|t|\le\eta\sigma\}$ be the positive central arc in that
range.  Define
\begin{equation}
 I_\pm(\xi)=\int_{-\eta\sqrt N}^{\eta\sqrt N}
 e^{\Psi(u)}\mathfrak a_\pm(u)P(u,\xi)\,du.
 \label{eq:higher-order-I}
\end{equation}
Then the change of variables $t=\sigma u/\sqrt N$ and
\eqref{eq:unified-pole-factor} give the exact identity
\begin{equation}
 J_\pm(\xi;E_0)
 =\frac{\sigma}{2\pi\sqrt N}\frac1{1\mp\rho}\,I_\pm(\xi).
 \label{eq:higher-order-JI-normalization}
\end{equation}
In particular,
\begin{equation}
 \frac{J_\pm(\xi;E_0)}{J_\pm(0;E_0)}
 =\frac{I_\pm(\xi)}{I_\pm(0)}.
 \label{eq:higher-order-JI-ratio}
\end{equation}

\begin{lemma}[higher-order symmetric saddle expansion]
\label{lem:higher-order-symmetric}
Fix $K\ge1$.  Suppose the central integral \eqref{eq:higher-order-I}
satisfies the bounds of \cref{lem:all-order-central-bounds} through order
$2K+3$.  Put $\varepsilon=N^{-1/2}$.  Then there are coefficients
$c_{j,r}^{\pm}$, depending on the saddle parameters but satisfying
\begin{equation}
 |c_{j,r}^{\pm}|\le C_K,
 \qquad 1\le j\le K-1,\quad 1\le r\le2j,
 \label{eq:higher-order-coefficient-bound}
\end{equation}
such that
\begin{equation}
 \frac{I_\pm(\xi)}{I_\pm(0)}
 =\Gamma_s(\xi)
 +\sum_{j=1}^{K-1}N^{-j}\sum_{r=1}^{2j}
 c_{j,r}^{\pm}\,\sigma^r\partial_s^r\Gamma_s(\xi)
 +\mathcal E_{K,\pm}(\xi),
 \label{eq:higher-order-ratio}
\end{equation}
where
\begin{equation}
 \sup_\xi|\mathcal E_{K,\pm}(\xi)|\le C_KN^{-K}.
 \label{eq:higher-order-ratio-error}
\end{equation}
Here the displayed $C_K$ is $C_K$, $C_{K,a,b}$, or $C_{K,A}$ in the small,
critical, or growing range.  Moreover, after increasing the corresponding
lower threshold if necessary,
\[
 |I_\pm(0)|\ge
 \begin{cases}
  c_K,&\text{small range},\\
  c_{K,a,b},&\text{critical range},\\
  c_{K,A},&\text{growing range}.
 \end{cases}
\]
\end{lemma}

\begin{proof}
Set
\[
 X(u)=\Psi(u)+\frac q2u^2.
\]
Taylor's theorem through order $2K+1$, together with
\eqref{eq:all-order-phase-derivatives}, gives
\begin{equation}
 X(u)=\sum_{m=3}^{2K+1}\gamma_m(iu)^m\varepsilon^{m-2}+R_X(u),
 \qquad
 |R_X(u)|\le C_K\varepsilon^{2K}|u|^{2K+2},
 \label{eq:higher-order-phase-taylor}
\end{equation}
with $|\gamma_m|\le C_K$.  Indeed the Taylor remainder of order $2K+2$
has size
\[
 d\left(\frac{\sigma|u|}{\sqrt N}\right)^{2K+2}
 \sup_{|t|\le\eta\sigma}|L^{(2K+2)}(s+it)|
 \le C_KN^{-K}|u|^{2K+2}.
\]

Taylor's theorem for the normalized pole factor gives
\begin{equation}
 \mathfrak a_\pm(u)=\sum_{m=0}^{2K-1}a_m^\pm u^m\varepsilon^m+R_A(u),
 \qquad
 |a_m^\pm|\le C_K,
 \qquad
 |R_A(u)|\le C_K\varepsilon^{2K}|u|^{2K},
 \label{eq:higher-order-pole-taylor}
\end{equation}
where $a_0^\pm=1$; this follows directly from
\eqref{eq:all-order-pole-derivatives}.  For the product we use Taylor's
formula in the real variable $u$.  The analytic identity
$P(u,\xi)=\Gamma_{s+i\sigma\varepsilon u}(\xi)$ gives, at $u=0$,
\[
 \partial_u^mP(0,\xi)
 =(i\sigma\varepsilon)^m\partial_s^m\Gamma_s(\xi).
\]
Hence
\begin{equation}
 P(u,\xi)=\sum_{m=0}^{2K-1}
 \frac{(iu)^m\varepsilon^m}{m!}
 \sigma^m\partial_s^m\Gamma_s(\xi)+R_P(u,\xi),
 \label{eq:higher-order-product-taylor}
\end{equation}
and the integral remainder together with
\eqref{eq:all-order-product-derivatives} yields
\begin{equation}
 |R_P(u,\xi)|
 \le C_K\varepsilon^{2K}|u|^{2K}
 (1+\mathcal E(\xi))^{2K}e^{-c\mathcal E(\xi)}
 \le C_K\varepsilon^{2K}(1+|u|)^{2K}.
 \label{eq:higher-order-product-remainder}
\end{equation}

For fixed $K$, expand the exponential of the finite phase jet and retain the
terms of total $\varepsilon$-weight less than $2K$.  Multiplying these terms by
one term from each of the pole and product jets produces only finitely many
scalar polynomials.  Let $D_K^*$ and $B_K^*$ bound their degrees and the
absolute values of their coefficients, respectively.  These bounds depend
only on $K$ and on the fixed parameters of the range.  Put
\[
 c_*:=\min\{c,q_0/2\}>0.
\]
Choose $M_K$ so large that, for every polynomial
$Q(u)=\sum_{m=0}^{D_K^*}q_m u^m$ with $\max_m|q_m|\le B_K^*$ and every
$\beta\ge c_*$,
\begin{equation}
 \int_{|u|>M_K\sqrt{\log N}}|Q(u)|e^{-\beta u^2}\,du
 \le C_KN^{-K-2}.
 \label{eq:higher-order-gaussian-tail}
\end{equation}
This choice applies both to the Gaussian factor, where
$\beta=q/2\ge q_0/2$, and to the actual central phase, where the decay
rate is bounded below by $c$.  Put
\[
 \varepsilon=N^{-1/2},
 \qquad
 R_N=M_K\sqrt{\log N}.
\]
After increasing the fixed lower threshold we may assume
$R_N\le\eta\sqrt N$.  By \eqref{eq:all-order-phase-decay} and the product
and pole bounds, the part of \eqref{eq:higher-order-I} with
$R_N<|u|\le\eta\sqrt N$ is $O_K(N^{-K-2})$.

On $|u|\le R_N$, \eqref{eq:higher-order-phase-taylor} gives
\[
 |X(u)|
 \le C_K\sum_{\nu=1}^{2K-1}
       \varepsilon^\nu(1+R_N)^{\nu+2}
       +C_K\varepsilon^{2K}(1+R_N)^{2K+2}.
\]
The right-hand side tends to zero with $N$; increasing the lower threshold
once more, we may assume
\begin{equation}
 |X(u)|\le\frac12,
 \qquad |R_X(u)|\le1,
 \qquad |u|\le R_N.
 \label{eq:higher-order-X-small}
\end{equation}

Write the finite weighted phase jet as
\[
 X_0(u)=\sum_{\nu=1}^{2K-1}\varepsilon^\nu x_\nu(u),
 \qquad
 x_\nu(u)=\gamma_{\nu+2}(iu)^{\nu+2},
\]
so $X=X_0+R_X$ and $x_\nu(-u)=(-1)^\nu x_\nu(u)$.  By
\eqref{eq:higher-order-X-small}, $|X_0|\le3/2$ on $|u|\le R_N$, and
Taylor's formula gives
\[
 e^{X_0}=\sum_{j=0}^{2K-1}\frac{X_0^j}{j!}
 +O_K(|X_0|^{2K}).
\]
Since every monomial in $X_0$ has positive $\varepsilon$-weight and
\[
 |X_0(u)|\le C_K\varepsilon(1+|u|)^{2K+1},
\]
the Taylor remainder and all terms in the finite sum of total weight at
least $2K$ are bounded, for a fixed integer $D_K$, by
\[
 C_K\varepsilon^{2K}(1+|u|)^{D_K}.
\]
Moreover,
\[
 |e^X-e^{X_0}|
 \le e^{|X_0|+|R_X|}|R_X|
 \le C_K\varepsilon^{2K}(1+|u|)^{2K+2}.
\]
Thus, after collecting equal $\varepsilon$-weights,
\[
 e^X=\sum_{\nu=0}^{2K-1}\varepsilon^\nu Q_\nu(u)+R_e(u),
 \qquad
 |R_e(u)|\le C_K\varepsilon^{2K}(1+|u|)^{D_K},
\]
where $Q_\nu(-u)=(-1)^\nu Q_\nu(u)$.

Write the finite parts in \eqref{eq:higher-order-pole-taylor} and
\eqref{eq:higher-order-product-taylor} as $\mathfrak a_{\pm,0}$ and $P_0$,
respectively.  Multiplying
\[
 \left(\sum_{\nu=0}^{2K-1}\varepsilon^\nu Q_\nu+R_e\right)
 (\mathfrak a_{\pm,0}+R_A)(P_0+R_P),
\]
we retain only terms of total $\varepsilon$-weight less than $2K$.  Terms
formed entirely from the three finite jets and having total weight at
least $2K$ are
$O_K(\varepsilon^{2K}(1+|u|)^{D_K})$.  Every other discarded mixed term
contains at least one of $R_e$, $R_A$, or $R_P$.  The remainder bounds above,
together with the polynomial bounds for the remaining finite factors, give
the same estimate for each such term.
Consequently, on $|u|\le R_N$,
\begin{equation}
 e^{X(u)}\mathfrak a_\pm(u)P(u,\xi)
 =\sum_{\nu=0}^{2K-1}\varepsilon^\nu H_{\nu,\pm}(u,\xi)
 +R_K(u,\xi),
 \label{eq:higher-order-polynomial-expansion}
\end{equation}
where
\begin{equation}
 H_{\nu,\pm}(-u,\xi)=(-1)^\nu H_{\nu,\pm}(u,\xi),
 \qquad
 |R_K(u,\xi)|\le
 C_K\varepsilon^{2K}(1+|u|)^{D_K}.
 \label{eq:higher-order-parity-remainder}
\end{equation}
Indeed, a basic factor of weight $m$ has parity $m$, and multiplication
preserves the congruence between total weight and parity.  The quantities
$\sigma^r\partial_s^r\Gamma_s(\xi)$ occur only through the single product
jet \eqref{eq:higher-order-product-taylor}; the phase and pole jets are
independent of them.  Hence every $H_{\nu,\pm}$ is linear in these
quantities, and a derivative of order $r$ can occur only when $r\le\nu$.
In particular, for the even weight $\nu=2j$ one has $r\le2j$.

Multiplying \eqref{eq:higher-order-polynomial-expansion} by
$e^{-qu^2/2}$ and integrating over $[-R_N,R_N]$, every odd $\nu$ vanishes
exactly.  The remainder contributes at most
\[
 C_K\varepsilon^{2K}
 \int_{\mathbb R}(1+|u|)^{D_K}e^{-q_0u^2/2}\,du
 \le C_KN^{-K}.
\]
For each even coefficient we may replace the integral over $[-R_N,R_N]$ by the
integral over $\mathbb R$ using \eqref{eq:higher-order-gaussian-tail}.  Since
\begin{equation}
 \sup_{q_0\le q\le q_1}
 \int_{\mathbb R}|u|^m e^{-qu^2/2}\,du\le C_m,
 \qquad m\ge0,
 \label{eq:uniform-gaussian-moments}
\end{equation}
all resulting scalar coefficients are uniformly bounded.  Restoring the
original central tails gives
\begin{equation}
 I_\pm(\xi)
 =\sum_{j=0}^{K-1}N^{-j}
 \sum_{r=0}^{2j}b_{j,r}^{\pm}
 \sigma^r\partial_s^r\Gamma_s(\xi)
 +O_K(N^{-K}),
 \qquad |b_{j,r}^{\pm}|\le C_K.
 \label{eq:higher-order-I-expansion}
\end{equation}
The leading coefficient is
$b_{0,0}^{\pm}=\sqrt{2\pi/q}$, since $\mathfrak a_\pm(0)=1$.

Set
\[
 B_N^\pm=\sum_{j=0}^{K-1}b_{j,0}^{\pm}N^{-j}.
\]
Because $b_{0,0}^{\pm}\ge\sqrt{2\pi/q_1}>0$, after increasing the fixed
lower threshold we have $|B_N^\pm|\ge c_K$.  Since
$\Gamma_s(0)=1$ and $\partial_s^r\Gamma_s(0)=0$ for $r\ge1$,
\eqref{eq:higher-order-I-expansion} gives
\[
 I_\pm(0)=B_N^\pm+O_K(N^{-K}).
\]
Write the numerator as
\[
 I_\pm(\xi)=B_N^\pm\Gamma_s(\xi)+D_N^\pm(\xi)+O_K(N^{-K}),
\]
where $D_N^\pm$ is the sum in \eqref{eq:higher-order-I-expansion} over
$r\ge1$.  Define recursively
\[
 \beta_0^\pm=(b_{0,0}^\pm)^{-1},
 \qquad
 \beta_j^\pm
 =-(b_{0,0}^\pm)^{-1}
   \sum_{m=1}^j b_{m,0}^\pm\beta_{j-m}^\pm,
 \qquad 1\le j\le K-1.
\]
The bounds on $b_{j,0}^\pm$ and the positive lower bound for
$b_{0,0}^\pm$ give $|\beta_j^\pm|\le C_K$, and direct multiplication
shows
\[
 (B_N^\pm)^{-1}
 =\sum_{j=0}^{K-1}\beta_j^\pm N^{-j}+O_K(N^{-K}).
\]
Since $I_\pm(0)=B_N^\pm+O_K(N^{-K})$ and both quantities are bounded
away from zero,
\[
 I_\pm(0)^{-1}=(B_N^\pm)^{-1}+O_K(N^{-K}).
\]
Multiplying this finite reciprocal by $D_N^\pm$ preserves the same order
bound: a scalar factor $N^{-m}$ can only raise the total asymptotic order
from $j$ to $j+m$, while $r\le2j\le2(j+m)$.  Hence, using
$|I_\pm(0)|\ge c_K$,
\[
 \frac{I_\pm(\xi)}{I_\pm(0)}
 =\Gamma_s(\xi)
 +\sum_{j=1}^{K-1}N^{-j}\sum_{r=1}^{2j}
 c_{j,r}^{\pm}\sigma^r\partial_s^r\Gamma_s(\xi)
 +O_K(N^{-K}).
\]
The bounded Gaussian moments and the finite reciprocal recursion give
\eqref{eq:higher-order-coefficient-bound}.  This proves
\eqref{eq:higher-order-ratio}--\eqref{eq:higher-order-ratio-error}; the
normalization in \eqref{eq:higher-order-JI-normalization} then restores the
original $J_\pm$ integrals.
\end{proof}

\subsection{Finite differences at nearby Gaussian scales}

The differential corrections in \eqref{eq:higher-order-ratio} can be
replaced by normalized Gaussians at neighboring real scales.

\begin{lemma}[finite differences at nearby Gaussian scales]
\label{lem:gaussian-finite-difference}
For every $K\ge1$ there is an integer $\Lambda_K$ and coefficients
$b_{r,\ell}^{(K)}$, $|\ell|\le \Lambda_K$, $0\le r\le2K-2$, such that
\begin{equation}
 \sum_{|\ell|\le \Lambda_K}b_{r,\ell}^{(K)}\ell^m
 =r!\,\mathbf 1_{\{m=r\}},
 \qquad 0\le m\le2K-1.
 \label{eq:finite-difference-moments}
\end{equation}
Consequently, if $\vartheta=\sigma/\sqrt N$ and $F$ is $C^{2K}$ on the
interval $|y-s|\le \Lambda_K\vartheta$, then
\begin{equation}
 \vartheta^rF^{(r)}(s)
 =\sum_{|\ell|\le \Lambda_K}b_{r,\ell}^{(K)}F(s+\ell\vartheta)
 +O_K\left(\vartheta^{2K}
 \sup_{|y-s|\le \Lambda_K\vartheta}|F^{(2K)}(y)|\right).
 \label{eq:finite-difference-formula}
\end{equation}
\end{lemma}

\begin{proof}
Choose $2K$ distinct integers among $\{-K,\ldots,K\}$ and solve the
Vandermonde system \eqref{eq:finite-difference-moments}; one may take
$\Lambda_K=K$.  Set the coefficients at the unused nodes equal to zero.  Taylor
expansion of $F(s+\ell\vartheta)$ through degree $2K-1$ then gives
\eqref{eq:finite-difference-formula}.  All coefficients depend only on $K$.
\end{proof}

Combining \cref{lem:higher-order-symmetric,lem:gaussian-finite-difference}
and \eqref{eq:all-order-real-scale-derivatives} gives the following form
which is adapted to maximal estimates.

\begin{proposition}[finite Gaussian form of the central ratios]
\label{prop:finite-gaussian-central}
For each fixed $K$ and each sign $\pm$, there are an integer $\Lambda_K$ and
a level $n_K$.  For every fixed $0<a\le b<\infty$ there is a dimension
$d_{K,a,b}$, and for every fixed $A>0$ there are
$\alpha_{1,K}(A)$ and $d_{K,A}$, such that
\begin{equation}
 \frac{I_\pm(\xi)}{I_\pm(0)}
 =\sum_{|\ell|\le\Lambda_K}c_{\ell,d,n}^{\pm}
 \Gamma_{d,\rho_{\ell,d,n}^{\pm}}(\xi)
 +O(N^{-K})
 \label{eq:finite-gaussian-central}
\end{equation}
holds in the small range when $n_K\le n\le\alpha_0d$, in the critical
range when $d\ge d_{K,a,b}$ and $ad\le n\le bd$, and in the growing range
when $d\ge d_{K,A}$ and $\alpha_{1,K}(A)d\le n\le Ad^2$.  Every
$\rho_{\ell,d,n}^{\pm}$ lies in $(0,1)$, and
\begin{equation}
 \sum_{|\ell|\le\Lambda_K}|c_{\ell,d,n}^{\pm}|
 \le C.
 \label{eq:finite-gaussian-coefficient-sum}
\end{equation}
The constant in this estimate and the implicit constant in
\eqref{eq:finite-gaussian-central} may be taken to be $C_K$,
$C_{K,a,b}$, and $C_{K,A}$ in the small, critical, and growing ranges,
respectively; the error is uniform in $\xi$.
\end{proposition}

\begin{proof}
In a term of \eqref{eq:higher-order-ratio}, put
$\vartheta=\sigma N^{-1/2}$.  Since $r\le2j$,
\[
 N^{-j}\sigma^r\partial_s^r\Gamma_s
 =N^{-j+r/2}\,\vartheta^r\partial_s^r\Gamma_s.
\]
The prefactor $N^{-j+r/2}$ is at most one.  Apply
\eqref{eq:finite-difference-formula} to $F(s)=\Gamma_s(\xi)$.  The error
there is
\[
 \vartheta^{2K}\sup|\partial_s^{2K}\Gamma|
 \le C_KN^{-K}
\]
by \eqref{eq:nearby-real-scale-derivatives}; multiplication by
$N^{-j+r/2}$ does not enlarge it.  There are only finitely many pairs
$(j,r)$, so all errors are absorbed into $O_K(N^{-K})$.

The real shifts are precisely $s_\ell=s+\ell\vartheta$.  Apply
\cref{lem:nearby-real-scales} with $B=\Lambda_K$.  After increasing the fixed
lower threshold, every $s_\ell$ is negative and lies in the corresponding
uniform parameter region described there; in the growing case, writing
$s_\ell=-\tau_\ell$, one has $\tau/2\le\tau_\ell\le3\tau/2$.
Thus every term produced by the finite difference is a genuine normalized
Gaussian multiplier at a real parameter for which the same uniform estimates
apply.  The moment coefficients in
\cref{lem:gaussian-finite-difference}, together with
\eqref{eq:higher-order-coefficient-bound} and $N^{-j+r/2}\le1$, give
\eqref{eq:finite-gaussian-coefficient-sum}.
\end{proof}

\subsection{The \texorpdfstring{$K$}{K}th-order two-saddle residual}

To distinguish the constructions on overlapping ranges, let
$\mathfrak r\in\{\mathrm{sm},\mathrm{cr},\mathrm{gr}\}$ denote the small,
critical, or growing regime.  For a central ratio in regime $\mathfrak r$,
let $Q_{d,n,\mathfrak r,+}^{[K]}$ and $Q_{d,n,\mathfrak r,-}^{[K]}$ denote
the two finite Gaussian sums furnished by
\cref{prop:finite-gaussian-central}.  Define
\begin{equation}
 \mathsf G_{d,n,\mathfrak r}^{[K]}(\xi)
 =Q_{d,n,\mathfrak r,+}^{[K]}(\xi)
 +a_{d,n}Q_{d,n,\mathfrak r,-}^{[K]}(\xi-\boldsymbol\omega).
 \label{eq:K-model}
\end{equation}
Let $\mathcal G_{d,n,\mathfrak r}^{[K]}$ be the corresponding finite linear
combination of normalized Gaussian convolution operators and their parity
conjugates.  For finitely supported $f$ it is equivalently specified by
\begin{equation}
 \widehat{\mathcal G_{d,n,\mathfrak r}^{[K]}f}(\xi)
 =\mathsf G_{d,n,\mathfrak r}^{[K]}(\xi)\widehat f(\xi).
 \label{eq:K-model-operator}
\end{equation}
By \eqref{eq:finite-gaussian-coefficient-sum}, the coefficient sums are uniformly
bounded.  Since normalized Gaussian convolutions and their parity conjugates are
$\ell^p$-contractions, $\mathcal G_{d,n,\mathfrak r}^{[K]}$ extends to a bounded
operator on every $\ell^p(\mathbb Z^d)$, $1\le p\le\infty$.  Each regime uses
its own model on overlaps.

\begin{theorem}[arbitrary-order two-saddle residual]
\label{thm:K-residual}
For every integer $K\ge1$ there is a level $n_K$.  For every fixed
$0<a\le b<\infty$ there is a dimension $d_{K,a,b}$, and for every fixed
$A>0$ there are $\alpha_{1,K}(A)$ and $d_{K,A}$.  The models in
\eqref{eq:K-model} may be chosen so that the sums of the absolute values of
the coefficients in each branch are bounded respectively by
$C_K$, $C_{K,a,b}$, and $C_{K,A}$ in the three regimes, and so that
\begin{align}
 \sup_\xi|m_{d,n}(\xi)-\mathsf G_{d,n,\mathrm{sm}}^{[K]}(\xi)|
 &\le C_K n^{-K},
 && n_K\le n\le\alpha_0d,
 \label{eq:K-small-residual}\\
 \sup_\xi|m_{d,n}(\xi)-\mathsf G_{d,n,\mathrm{cr}}^{[K]}(\xi)|
 &\le C_{K,a,b}d^{-K},
 && d\ge d_{K,a,b},\quad ad\le n\le bd,
 \label{eq:K-critical-residual}\\
 \sup_\xi|m_{d,n}(\xi)-\mathsf G_{d,n,\mathrm{gr}}^{[K]}(\xi)|
 &\le C_{K,A}d^{-K},
 && d\ge d_{K,A},\quad \alpha_{1,K}(A)d\le n\le Ad^2.
 \label{eq:K-growing-residual}
\end{align}
The constants are independent of $d$ and $n$.
\end{theorem}

\begin{proof}
Fix a regime and omit $\mathfrak r$ from the central quantities.  We recover
the two coefficients from the values at $0$ and $\boldsymbol\omega$; all
constants retain the dependence stated in the theorem.  In the growing
regime we enlarge the thresholds, if necessary, so
that $\alpha_{1,K}(A)\ge\alpha_r(A)$ and
$d_{K,A}\ge d_{r,K}(A)$ from \cref{lem:growing-remote}.  Let $E_0$ be the
positive central arc, and let $T_+(\xi)$ denote the contribution of the
complement of the two central arcs to the unnormalized integral
$J_+(\xi;[-\pi,\pi])$.  Write
\[
 D=J_+(0;[-\pi,\pi]),\qquad
 c_+=\frac{J_+(0;E_0)}D,\qquad
 c_-=\frac{(-1)^nJ_-(0;E_0)}D.
\]
The exact contour identity has the form
\begin{equation}
 m_{d,n}(\xi)=c_+F_+(\xi)+c_-F_-(\xi-\boldsymbol\omega)
 +E_{\rm rem}(\xi),
 \label{eq:K-exact-decomposition}
\end{equation}
where
\[
 F_\pm(\xi)=\frac{J_\pm(\xi;E_0)}{J_\pm(0;E_0)},
 \qquad
 E_{\rm rem}(\xi)=\frac{T_+(\xi)}D.
\]

For every fixed $K$,
\begin{equation}
 \frac{\sup_\xi|T_+(\xi)|}{|J_+(0;E_0)|}\le C_KN^{-K},
 \label{eq:K-raw-remote}
\end{equation}
with the fixed-range dependence stated in the theorem.  In the small range
the left-hand side is bounded by $C\sqrt n\,e^{-cn}$; in a fixed critical
window it is bounded by $C\sqrt d\,q^d$ with $q<1$; and in the growing range
the assertion follows from \cref{lem:growing-remote} with expansion order
$K$.
The same is true of the cross-saddle values of every real Gaussian scale
appearing in \cref{prop:finite-gaussian-central}.  Indeed,
\cref{lem:nearby-real-scales} gives uniformly over all those finitely many
scales
\begin{equation}
 |\Gamma_{s_\ell}(\boldsymbol\omega)|
 \le
 \begin{cases}
 C_Ke^{-c_Kn},& n\le\alpha_0d,\\
 C_{K,a,b}q_{K,a,b}^d,& ad\le n\le bd,\\
 C_{K,A}e^{-c_{K,A}d/\tau},& \alpha_{1,K}(A)d\le n\le Ad^2.
 \end{cases}
 \label{eq:K-nearby-cross-saddle}
\end{equation}
Each right-hand side is $O_K(N^{-K})$ above the fixed $K$-dependent lower
threshold.

The central base estimates proved earlier also give
\[
 \left|\frac{J_-(0;E_0)}{J_+(0;E_0)}\right|\le C
\]
in the small and critical ranges, and $O(\tau)$ in the growing range.
Moreover \cref{prop:finite-gaussian-central} evaluated at
$\boldsymbol\omega$, together with the preceding cross-saddle estimates,
gives
\[
 |F_-(\boldsymbol\omega)|\le C_KN^{-K}+\text{(exponentially small)}.
\]
At frequency zero the contour decomposition gives
\[
 D=J_+(0;E_0)+(-1)^nJ_-(\boldsymbol\omega;E_0)
   +T_+(0).
\]
Dividing by $J_+(0;E_0)$, we obtain
\[
 \frac{D}{J_+(0;E_0)}
 =1+(-1)^n\frac{J_-(0;E_0)}{J_+(0;E_0)}
       F_-(\boldsymbol\omega)
       +\frac{T_+(0)}{J_+(0;E_0)}
 =1+O_K(N^{-K}).
\]
Consequently, \eqref{eq:K-raw-remote} also gives
$\|E_{\rm rem}\|_{L^\infty}\le C_KN^{-K}$, and
\begin{equation}
 c_+=\frac{J_+(0;E_0)}D=1+O_K(N^{-K}),
 \qquad |c_-|\le C_K.
 \label{eq:K-cplus}
\end{equation}
Now evaluate \eqref{eq:K-exact-decomposition} at
$\xi=\boldsymbol\omega$.  The exact parity identity gives
$m_{d,n}(\boldsymbol\omega)=a_{d,n}$, while $F_-(0)=1$ and
\cref{prop:finite-gaussian-central,lem:nearby-real-scales} give
$|F_+(\boldsymbol\omega)|\le C_KN^{-K}$.  Together with the remote
estimate and $|c_+|\le C_K$, this yields
\begin{equation}
 c_-=a_{d,n}+O_K(N^{-K}).
 \label{eq:K-cminus}
\end{equation}
Combining \eqref{eq:K-exact-decomposition}--\eqref{eq:K-cminus} with
\cref{prop:finite-gaussian-central} proves the three displayed residual
bounds above the stated thresholds.
\end{proof}
\section{Proof of the full \texorpdfstring{$\ell^p$}{lp} theorem}
\label{sec:assembly}

For $1<p<2$, we interpolate each fixed-level residual between $\ell^1$ and
$\ell^2$ and use the normalized Gaussian maximal theorem for the model.

\subsection{The Gaussian models and the residuals}

Let $1<p<\infty$.  For fixed $K$, and with $a,b,A$ fixed when the
corresponding regime is considered, set
\[
 \begin{aligned}
 \mathcal R_{\mathrm{sm}}(d)&=\{n:0\le n\le\alpha_0d\},\\
 \mathcal R_{\mathrm{cr}}(d)&=\{n:ad\le n\le bd\},\\
 \mathcal R_{\mathrm{gr}}(d)&=\{n:\alpha_{1,K}(A)d\le n\le Ad^2\}.
 \end{aligned}
\]
We interpret a supremum over an empty range as zero.  Below the asymptotic
thresholds we set the corresponding approximating operator equal to zero.
The exceptional sets contain uniformly bounded numbers of levels: at most
$n_K$ in the small range, at most $bd_{K,a,b}+1$ in the critical range, and
at most $Ad_{K,A}^2+1$ in the growing range.

For $\mathfrak r\in\{\mathrm{sm},\mathrm{cr},\mathrm{gr}\}$, the coefficient
bounds in \cref{thm:K-residual} and $|a_{d,n}|\le1$ give, pointwise,
\[
 \sup_{n\in\mathcal R_{\mathfrak r}(d)}
 |\mathcal G_{d,n,\mathfrak r}^{[K]}f|
 \le C_{K,\mathfrak r}
 \left(\sup_{0<\rho<1}|G_{d,\rho}f|
 +\sup_{0<\rho<1}|U_{\boldsymbol\omega}G_{d,\rho}
 U_{\boldsymbol\omega}f|\right).
\]
The centered and translated Gaussian maximal bounds
\eqref{eq:known-gaussian} and \eqref{eq:translated-gaussian-maximal}
therefore imply
\begin{equation}
 \left\|\sup_{n\in\mathcal R_{\mathfrak r}(d)}
 \big|\mathcal G_{d,n,\mathfrak r}^{[K]}f\big|\right\|_p
 \le C_{p,K,\mathfrak r}\|f\|_p,
 \label{eq:K-model-maximal}
\end{equation}
where the constant has the dependence stated in
\cref{thm:K-residual}.

Define
\[
 E_{d,n,\mathfrak r}^{[K]}
 :=M_{d,n}-\mathcal G_{d,n,\mathfrak r}^{[K]}.
\]
The normalized ball averages are $\ell^1$ contractions by
\eqref{eq:ball-contraction}; every normalized Gaussian convolution is also an
$\ell^1$ contraction, and parity modulation preserves the $\ell^1$ norm.
Thus
\begin{equation}
 \|E_{d,n,\mathfrak r}^{[K]}\|_{\ell^1\to\ell^1}
 \le C_{K,\mathfrak r}.
 \label{eq:K-residual-l1}
\end{equation}
Above the asymptotic thresholds, Plancherel and \cref{thm:K-residual} give
\begin{align}
 \|E_{d,n,\mathrm{sm}}^{[K]}\|_{\ell^2\to\ell^2}
 &\le C_Kn^{-K},
 &&n_K\le n\le\alpha_0d, \label{eq:K-residual-l2-small}\\
 \|E_{d,n,\mathrm{cr}}^{[K]}\|_{\ell^2\to\ell^2}
 &\le C_{K,a,b}d^{-K},
 &&d\ge d_{K,a,b},\quad ad\le n\le bd,
 \label{eq:K-residual-l2-critical}\\
 \|E_{d,n,\mathrm{gr}}^{[K]}\|_{\ell^2\to\ell^2}
 &\le C_{K,A}d^{-K},
 &&d\ge d_{K,A},\quad \alpha_{1,K}(A)d\le n\le Ad^2.
 \label{eq:K-residual-l2-growing}
\end{align}

Let $1<p\le2$ and put
\[
 \theta=2\left(1-\frac1p\right)=\frac{2(p-1)}p.
\]
Riesz--Thorin interpolation of each fixed-radius residual between
\eqref{eq:K-residual-l1} and the appropriate line above yields
\begin{align}
 \|E_{d,n,\mathrm{sm}}^{[K]}\|_{\ell^p\to\ell^p}
 &\le C_{p,K}n^{-2K(p-1)/p},
 &&n_K\le n\le\alpha_0d, \label{eq:K-residual-lp-small}\\
 \|E_{d,n,\mathrm{cr}}^{[K]}\|_{\ell^p\to\ell^p}
 &\le C_{p,K,a,b}d^{-2K(p-1)/p},
 &&d\ge d_{K,a,b},\quad ad\le n\le bd,
 \label{eq:K-residual-lp-critical}\\
 \|E_{d,n,\mathrm{gr}}^{[K]}\|_{\ell^p\to\ell^p}
 &\le C_{p,K,A}d^{-2K(p-1)/p},
 &&d\ge d_{K,A},\quad \alpha_{1,K}(A)d\le n\le Ad^2.
 \label{eq:K-residual-lp-growing}
\end{align}
For any family $g_n$ and $p\ge1$,
\begin{equation}
 \left\|\sup_n|g_n|\right\|_p^p
 \le\sum_n\|g_n\|_p^p.
 \label{eq:lp-sup-sum}
\end{equation}

\begin{proposition}[finite-range maximal bounds below $2$]
\label{prop:finite-range-p}
Let $1<p\le2$ and let $K$ be an integer such that
\begin{equation}
 K(p-1)\ge1.
 \label{eq:K-p-condition}
\end{equation}
Then the maximal functions on each of the three finite ranges are bounded on
$\ell^p(\mathbb Z^d)$ by constants independent of $d$.
\end{proposition}

\begin{proof}
The Gaussian parts are controlled by \eqref{eq:K-model-maximal}.  In the
small range, \eqref{eq:K-residual-lp-small} and
\eqref{eq:lp-sup-sum} give
\[
 \left\|\sup_{n_K\le n\le\alpha_0d}
 |E_{d,n,\mathrm{sm}}^{[K]}f|\right\|_p^p
 \le C_{p,K}\sum_{n\ge n_K}n^{-2K(p-1)}\|f\|_p^p
 \le C_{p,K}\|f\|_p^p.
\]
The omitted levels $n<n_K$ form a fixed finite family and are bounded by the
contraction estimate.

A fixed critical interval contains at most $C_{a,b}d+1$ levels.  Thus, for
$d\ge d_{K,a,b}$,
\[
 \left\|\sup_{ad\le n\le bd}|E_{d,n,\mathrm{cr}}^{[K]}f|\right\|_p^p
 \le C_{p,K,a,b}d^{1-2K(p-1)}\|f\|_p^p
 \le C_{p,K,a,b}\|f\|_p^p.
\]
The dimensions below $d_{K,a,b}$ form a fixed finite set; the corresponding
levels are handled directly by \eqref{eq:ball-contraction}.

Finally the growing interval contains at most $Ad^2+1$ levels.  For
$d\ge d_{K,A}$,
\[
 \left\|\sup_{\alpha_{1,K}(A)d\le n\le Ad^2}
 |E_{d,n,\mathrm{gr}}^{[K]}f|\right\|_p^p
 \le C_{p,K,A}d^{2-2K(p-1)}\|f\|_p^p
 \le C_{p,K,A}\|f\|_p^p.
\]
The finitely many dimensions below $d_{K,A}$ are treated in the same way.
The growing-range estimate requires $K(p-1)\ge1$, which also suffices for
the other two ranges.
\end{proof}

\subsection{Coverage of all radii}

We first treat $1<p<2$.  Choose
\begin{equation}
 K=K(p):=\left\lceil\frac1{p-1}\right\rceil.
 \label{eq:K-choice}
\end{equation}
Then \eqref{eq:K-p-condition} holds.  Fix the constants in the same order
as in the $\ell^2$ proof: let $C_L$ be the threshold in
\eqref{eq:known-large-radius}, put
\[
 A=(C_L+1)^2.
\]
For this fixed $K=K(p)$, write
$\alpha_1:=\alpha_{1,K}(A)$ for the growing-range threshold, and take a
fixed critical interval joining the small and growing regimes, for instance
\[
 a=\frac{\alpha_0}{2},\qquad b=2\alpha_1.
\]
With these choices, all constants from the critical and growing ranges
depend only on $p$ and on the fixed
absolute thresholds, never on $d,n$, or $f$.  By
\cref{prop:finite-range-p}, the maximal operator is dimension-free on
each of the three finite ranges.  The tail
\[
 \left\|\sup_{t\ge C_Ld}|\Avg_tf|\right\|_p
 \le C_p\|f\|_p
\]
is \eqref{eq:known-large-radius}.

As in the $\ell^2$ proof, the union of
\[
 [0,\alpha_0d],\qquad
 [\tfrac12\alpha_0d,2\alpha_1d],\qquad
 [\alpha_1d,Ad^2]
\]
contains all squared-radius levels up to $Ad^2$; the third interval may be
empty in small dimensions, in which case the second already extends beyond
$Ad^2$.  Since
$A=(C_L+1)^2$, every $n>Ad^2$ satisfies
$\sqrt n>(C_L+1)d>C_Ld$, and hence belongs to the large-radius tail.
Therefore
\begin{equation}
 \left\|\sup_{t\ge0}|\Avg_tf|\right\|_p
 \le C_p\|f\|_p,
 \qquad 1<p<2,
 \label{eq:full-p-below2}
\end{equation}
with $C_p$ independent of $d$.

For $p=2$, the full first-order estimate is \eqref{eq:first-order-full-l2}.
Marcinkiewicz interpolation of the sublinear maximal
operator between the strong $\ell^2$ bound and the trivial strong
$\ell^\infty$ contraction gives the bound for $2<p<\infty$; the endpoint
$p=\infty$ is the contraction itself.  This completes the proof of
\cref{thm:main}.

\paragraph{Acknowledgments.}
We thank Zecheng Gan for facilitating this collaboration.
Kaiwen Jin also thanks Kelan for being the one certainty through all the oscillations,
and for always standing by him and encouraging him through difficult times.
AI tools produced most of the mathematical proofs and the initial manuscript
draft; the authors revised and edited the manuscript.

\bigskip
\begingroup
\small
\noindent\textsc{Kaiwen Jin}\\
The Hong Kong University of Science and Technology (Guangzhou)\\
\textit{Email:} \href{mailto:kjin327@connect.hkust-gz.edu.cn}{\texttt{kjin327@connect.hkust-gz.edu.cn}}

\medskip
\noindent\textsc{Qingtang Su}\\
Morningside Center of Mathematics\\
Academy of Mathematics and Systems Science, Chinese Academy of Sciences\\
\textit{Email:} \href{mailto:suqingtang@amss.ac.cn}{\texttt{suqingtang@amss.ac.cn}}
\endgroup

\end{document}